\documentclass[10pt,a4paper]{amsart}
\usepackage[utf8]{inputenc}
\usepackage[T1]{fontenc}
\usepackage{amsmath}
\usepackage{amssymb,amsthm, enumitem}
\usepackage[dvipsnames]{xcolor}
\usepackage[english]{babel}

\usepackage{physics}
\usepackage{tikz}
\usepackage{mathdots}
\usepackage{yhmath}
\usepackage{cancel}
\usepackage{color}
\usepackage{siunitx}
\usepackage{array}
\usepackage{multirow}
\usepackage{gensymb}
\usepackage{tabularx}
\usepackage{extarrows}
\usepackage{booktabs}
\usetikzlibrary{fadings}
\usetikzlibrary{patterns}
\usetikzlibrary{shadows.blur}
\usetikzlibrary{shapes}

\newcommand{\Glambda}{\mathcal{G}_{\lambda,c}}
\newcommand{\ve}{\varepsilon}
\newcommand{\q}{\tilde q^{\ve}_{\delta,k}}

\theoremstyle{plain}
\newtheorem{lemma}{Lemma}
\newtheorem{prop}{Proposition}
\newtheorem{thm}{Theorem}
\newtheorem{defn}{Definition}

\author{Stefano Baranzini, Gian Marco Canneori, Susanna Terracini}
\title{Frozen planet orbits for the $n$-electron atom}
\date{}
\address{Dipartimento di Matematica ``G. Peano''
		\newline\indent
		Universit\`a degli Studi di Torino
		\newline\indent
		Via Carlo Alberto 10, 10123 Torino, Italy\\}
	\email{stefano.baranzini@unito.it}
	\email{gianmarco.canneori@unito.it}
	\email{susanna.terracini@unito.it}
	
	\date{\today}
	\thanks{The authors are partially supported by INDAM-GNAMPA research group. The first author is partially supported by the PRIN 2022 project 2022FPZEES -- \emph{Stability in Hamiltonian Dynamics and Beyond}. The third author is partially supported by the PRIN 2022 project 20227HX33Z -- \emph{Pattern formation in nonlinear phenomena}}
	\thanks{We wish to sincerely thank the anonymous referees for carefully reading our paper and giving useful comments and suggestions that have greatly improved our paper}

\allowdisplaybreaks

\begin{document}	

	\maketitle
	\begin{abstract}
		We investigate periodic trajectories in a classical system consisting of multiple mutually repelling electrons constrained to move along a half-line, with an attractive nucleus fixed at the origin. Adopting a variational framework, we seek critical points of the associated Lagrangian action functional using a modified Lusternik-Schnirelmann theory for manifolds with boundary. Furthermore, in the limit where the electron charges vanish, we demonstrate that frozen planet orbits converge to segments of a brake orbit in a Kepler-type problem, thereby drawing a strong analogy with Schubart orbits in the gravitational $N$-body problem.
	\end{abstract}
%	\tableofcontents

\section{Introduction}
We consider the classical Hamiltonian of an atom with $n$ electrons, where the nucleus is treated as infinitely massive and fixed at the origin. Electrons interact through Coulomb (electrostatic) and Newtonian gravitational forces, while relativistic and quantum effects are neglected:
\[
H = \sum_{i=1}^{n} \left( \frac{\mathbf{p}_i^2}{2} - \frac{Z(1 + G)}{|\mathbf{r}_i|} \right)
+ \sum_{1 \le i < j \le n} \left( \frac{1 - G}{|\mathbf{r}_i - \mathbf{r}_j|} \right).
\]
Here, $Z$ denotes the atomic number (nuclear charge), and $\mathbf{r}_i$ the position of the $i$th electron. The Hamiltonian is expressed in atomic units (a.u.), where $\hbar = 1$, $m = 1$, $e = 1$, and $4\pi\epsilon_0 = 1$. The unit of energy is the Hartree (\( \approx 27.2 \) eV), and distances are measured in Bohr radii. The gravitational constant $G$ is extremely small in atomic units:
\[
G \approx 1.21 \times 10^{-45} \ \text{a.u.}
\]
Thus, gravitational terms are negligible compared to the electrostatic interactions but are included for completeness. For neutral atoms, $n = Z$; for ions, $n \ne Z$, with cations ($Z > n$) and anions ($Z < n$). The total atomic charge is $Z - n$.
This work investigates periodic trajectories in a one-dimensional model of an atom with $n$ electrons constrained to a half-line. The nucleus remains fixed at the origin, while the electrons experience mutual repulsion and are attracted to the nucleus via Coulomb's law. Normalizing the Coulomb constant to 1 yields the system of differential equations:
\begin{equation}
\label{eq:helium_newton}
\ddot{q}_i = -\frac{Z(1+G)}{q_i^2} + \sum_{j=1}^{i-1} \frac{1-G}{(q_i-q_j)^2}-\sum_{j = i+1}^n \frac{1-G}{(q_j-q_i)^2},
\end{equation}
with the constraint $q_1<q_2<\dots<q_n$.
We seek generalized periodic solutions where the first electron, $q_1$, reflects elastically upon colliding with the nucleus. This reflection rule aligns with the classical Levi-Civita regularization of collisions in the Kepler problem (cf. \cite{BarOrtVer,RebSim} and references therein). Any bound state must include such a collision due to the one-dimensional confinement of electrons.
 
In this paper, we focus on periodic orbits that exhibit brake symmetry. To be precise, we look for solutions of \eqref{eq:helium_newton} on $[0,T]$ with the following boundary conditions:
\[
	\begin{cases}
	q_1(0) =0,\ \dot{q}_1(T) =0 \\
	\dot{q}_i(0)= \dot{q}_i(T)=0 \text{ for all  }i \ge2
	\end{cases}
\]
For the $2$-electrons problem ($n=2$), such solutions have been recently studied in \cite{zhao,Cieliebak_Langmuir,Cieliebak_Non_deg,Cieliebak_variational} and, by the authors, in \cite{helium_frozen}. In the literature, inspired by similarly shaped solutions in orbital mechanics, they were called \emph{frozen planet orbits}. Indeed, while $q_1$ bounces repeatedly against the nucleus, the remaining electrons oscillate slowly and far off from this binary cluster. 

We can also draw a parallel with the one dimensional gravitational $n-$body problem. Frozen planet orbits are indeed the natural counterpart of Schubart orbits (cf. \cite{schubart1956}). We expect that, as collinear periodic solutions, they are isolated as the Schubart's ones. Thus, an intriguing open problem is whether they can be continued in a family of \emph{symmetric planar} periodic trajectories with non-zero angular momentum, as shown to be the case for Schubart orbits (see \cite{henon1976}). We refer the reader to \cite{davies,fenucci} for numerical evidence of the existence of planar and spatial symmetric orbits of \eqref{eq:helium_newton}, and to \cite{spatialRelativeEquilibria} for a computer-assisted study of its central configurations. 

Following the approach in \cite{helium_frozen}, we adopt a framework based on critical points theory to establish frozen planet orbits.  This perspective enables us to treat a broader class of models beyond equation \eqref{eq:helium_newton}, replacing purely Newtonian interactions with more general potentials that satisfy weaker homogeneity assumptions, including those involving different negative powers of the distance. The flexibility to consider interactions beyond the classical Coulomb case further allows for a more abstract formulation of the problem, helping to isolate the fundamental properties that ensure the existence of solutions.

Thus, we will investigate the following general system of singular differential equations:
\begin{equation}
	\label{eq:helium_equation}
	\ddot{q}_i = f_i'(\vert q_i \vert ) - \sum_{j=1}^{i-1} g'_{ij}(\vert q_i-q_j\vert )+\sum_{j = i+1}^n g'_{ij}(\vert q_j-q_i \vert).
\end{equation} 
The function $f'_i$ represents the attraction force exerted on $q_i$ by the nucleus, whereas $g'_{ij}$ is the repulsion force between the $i$-th  and $j$-th electrons. We will assume that $g_{ij} = g_{ji}$ for all $i \ne j$ and $f_i,g_{ij} :(0,\infty)\to (0,\infty)$ are $C^{1,1}(0,\infty)$ and satisfy the following assumptions:
\begin{align*}
	\label{assumption:values_f_g}
	&\begin{aligned}
	&f_i(s),g_{ij}(s),f_i'(s),g_{ij}'(s) \to 0,\ \text{ as } s \to +\infty \text{ for all }i,j, \\
	&f_i'(s),g_{ij}'(s) < 0,\  \text{ for all } s>0 \text{ and for all }i,j.
	\end{aligned} \tag{H1} \\
	&\\
	\label{assumption:homogenuity}
    &\begin{aligned}
    &\exists \alpha \in (0,2) \text{ such that }\\
    &\forall i, \forall s >0, \quad  s f_i'(s)+\alpha f_i(s)\ge 0, \\ 
	&\forall i<j, \forall s >0, \quad s g_{ij}'(s)+\alpha g_{ij}(s)\le 0.
	\end{aligned} \tag{H2}\\
	\label{assumption:attraction_infinity}
	&\forall j, \,\forall t_0 >0\ \exists s_0>t_0  : \forall 0<t_i\le t_0, \forall s>s_0, \quad f'_j(s)-\sum_{i=1}^{j-1}g_{ij}'(s-t_i)< 0. \tag{H3}\\
	\label{assumption:convexity}
	&\begin{aligned}
        &\forall j,\ \forall s>0,\quad f_j\;\text{is convex}, \\
        &\exists\ C>0\ \text{such that}\quad f'_{j}(t) \ge f'_l(s)-C,\ \text{ for any } j> l \text{ and } 0\le s\le t.\end{aligned} \tag{H4}
\end{align*}
Note that potentials  $f_i(s)=a_i/s^{\alpha_i},\, g_{ij}(s)=b_{ij}/s^{\beta_{ij}}$, with $a_i, b_{ij}>0$  always fulfill these assumptions provided that the following conditions hold:  
\begin{itemize}
     \item  $\alpha_l\ge\alpha_j$ for any $j\ge l$ and $a_j\le a_l$ if $\alpha_l=\alpha_j$;
	\item for all $i$, $\alpha_i\in (0,2)$ and  $\beta_{ij}\geq\alpha^*:=\max_{i}\alpha_i$;
	\item if the set $S_{j} = \{i:\alpha_j = \alpha^*, \beta_{ji}=\alpha^*\}$  is non empty, then $\sum\limits_{i\in S_j}b_{ji}<a_j$.
\end{itemize}
In particular, the potentials appearing in system~\eqref{eq:helium_newton} also satisfy assumptions~\eqref{assumption:values_f_g}, \eqref{assumption:homogenuity}, and~\eqref{assumption:convexity}. 
Regarding assumption~\eqref{assumption:attraction_infinity}, it holds provided that $Z\geq n-1$. In other words, the condition is met for all neutral atoms and cations, while anions must carry a single negative charge (i.e., $Z=n-1$); note that in this latter case, the gravitational potential cannot be neglected in order to ensure that $Z (1+ G) > (n - 1)(1 - G)$. 
Nevertheless, the class of admissible potentials remains substantially more general.

Assumption \eqref{assumption:homogenuity} is a classical Ambrosetti-Rabinowitz-type condition already present in literature (e.g., see \cite{MinMaxRabinowitz}). Condition \eqref{assumption:attraction_infinity} ensures that, independently on the initial position of the electrons, the attractive force prevails at infinity. Condition \eqref{assumption:convexity} has a physical interpretation too: if we assume electrons to be ordered increasingly, so that $q_i<q_{i+1}$, it states that the nucleus attraction is stronger on the ones closer to the origin. 

Our first result states the existence of frozen planet orbits for generalised $n$-electron problems.
\begin{thm}
	\label{thm:solution_helium}
	For any $T>0$ and any choice of $f_i,g_{ij}$ satisfying assumptions \eqref{assumption:values_f_g}-\eqref{assumption:convexity}, there exists a solution of \eqref{eq:helium_equation} such that:
	\begin{enumerate}
		\item $q_i(t)<q_{i+1}(t)$ for all $t \in[0,T]$ and $i=1,\dots n$,
		\item $\dot{q}_i(0) =\dot{q}_i(T)=0$ for all  $i\ge 2$,
		\item $q_1(0)=0$,  $\dot{q}_1(T)=0$.	
	\end{enumerate} 
\end{thm}

\begin{figure}[t]
    \centering
	\resizebox{.65\textwidth}{!}{
	\tikzset{every picture/.style={line width=0.75pt}} %set default line width to 0.75pt        

\begin{tikzpicture}[x=0.75pt,y=0.75pt,yscale=-1,xscale=1]
%uncomment if require: \path (0,457); %set diagram left start at 0, and has height of 457

%Shape: Circle [id:dp4350477296798687] 
\draw  [draw opacity=0][fill={rgb, 255:red, 206; green, 224; blue, 245 }  ,fill opacity=1 ][dash pattern={on 1.69pt off 2.76pt}][line width=1.5]  (92,86.07) .. controls (92,83.35) and (94.2,81.15) .. (96.93,81.15) .. controls (99.65,81.15) and (101.85,83.35) .. (101.85,86.07) .. controls (101.85,88.8) and (99.65,91) .. (96.93,91) .. controls (94.2,91) and (92,88.8) .. (92,86.07) -- cycle ;
%Straight Lines [id:da17997314652893148] 
\draw    (99,86) -- (505.33,86) ;
%Shape: Circle [id:dp30547536140957243] 
\draw  [color={rgb, 255:red, 48; green, 65; blue, 69 }  ,draw opacity=1 ][fill={rgb, 255:red, 208; green, 2; blue, 27 }  ,fill opacity=1 ] (94,86) .. controls (94,84.62) and (95.12,83.5) .. (96.5,83.5) .. controls (97.88,83.5) and (99,84.62) .. (99,86) .. controls (99,87.38) and (97.88,88.5) .. (96.5,88.5) .. controls (95.12,88.5) and (94,87.38) .. (94,86) -- cycle ;
%Straight Lines [id:da6496088236707023] 
\draw  [dash pattern={on 0.84pt off 2.51pt}]  (505.33,86) -- (529.33,86) ;
%Shape: Circle [id:dp8504098885667012] 
\draw  [draw opacity=0][fill={rgb, 255:red, 206; green, 224; blue, 245 }  ,fill opacity=1 ][dash pattern={on 1.69pt off 2.76pt}][line width=1.5]  (215,86.07) .. controls (215,83.35) and (217.2,81.15) .. (219.93,81.15) .. controls (222.65,81.15) and (224.85,83.35) .. (224.85,86.07) .. controls (224.85,88.8) and (222.65,91) .. (219.93,91) .. controls (217.2,91) and (215,88.8) .. (215,86.07) -- cycle ;
%Shape: Circle [id:dp6939167745223196] 
\draw  [draw opacity=0][fill={rgb, 255:red, 214; green, 229; blue, 247 }  ,fill opacity=1 ][line width=1.5]  (231,86.07) .. controls (231,83.35) and (233.2,81.15) .. (235.93,81.15) .. controls (238.65,81.15) and (240.85,83.35) .. (240.85,86.07) .. controls (240.85,88.8) and (238.65,91) .. (235.93,91) .. controls (233.2,91) and (231,88.8) .. (231,86.07) -- cycle ;
%Shape: Circle [id:dp7180374710615666] 
\draw  [draw opacity=0][fill={rgb, 255:red, 214; green, 229; blue, 247 }  ,fill opacity=1 ][line width=1.5]  (383,86.07) .. controls (383,83.35) and (385.2,81.15) .. (387.93,81.15) .. controls (390.65,81.15) and (392.85,83.35) .. (392.85,86.07) .. controls (392.85,88.8) and (390.65,91) .. (387.93,91) .. controls (385.2,91) and (383,88.8) .. (383,86.07) -- cycle ;
%Shape: Circle [id:dp17330791529117118] 
\draw  [draw opacity=0][fill={rgb, 255:red, 214; green, 229; blue, 247 }  ,fill opacity=1 ][line width=1.5]  (398,86.07) .. controls (398,83.35) and (400.2,81.15) .. (402.93,81.15) .. controls (405.65,81.15) and (407.85,83.35) .. (407.85,86.07) .. controls (407.85,88.8) and (405.65,91) .. (402.93,91) .. controls (400.2,91) and (398,88.8) .. (398,86.07) -- cycle ;
%Straight Lines [id:da35144743437670645] 
\draw [color={rgb, 255:red, 74; green, 144; blue, 226 }  ,draw opacity=1 ] [dash pattern={on 4.5pt off 4.5pt}]  (154,78.77) -- (188.93,80.08) ;
\draw [shift={(190.93,80.15)}, rotate = 182.15] [fill={rgb, 255:red, 74; green, 144; blue, 226 }  ,fill opacity=1 ][line width=0.08]  [draw opacity=0] (12,-3) -- (0,0) -- (12,3) -- cycle    ;
%Shape: Circle [id:dp2673006289559755] 
\draw  [color={rgb, 255:red, 74; green, 144; blue, 226 }  ,draw opacity=1 ][fill={rgb, 255:red, 157; green, 197; blue, 244 }  ,fill opacity=1 ][line width=1.5]  (138,87.07) .. controls (138,84.35) and (140.2,82.15) .. (142.93,82.15) .. controls (145.65,82.15) and (147.85,84.35) .. (147.85,87.07) .. controls (147.85,89.8) and (145.65,92) .. (142.93,92) .. controls (140.2,92) and (138,89.8) .. (138,87.07) -- cycle ;
%Shape: Circle [id:dp4765389433230208] 
\draw  [color={rgb, 255:red, 74; green, 144; blue, 226 }  ,draw opacity=1 ][fill={rgb, 255:red, 157; green, 197; blue, 244 }  ,fill opacity=1 ][line width=1.5]  (299,86.46) .. controls (299,83.74) and (301.2,81.53) .. (303.93,81.53) .. controls (306.65,81.53) and (308.85,83.74) .. (308.85,86.46) .. controls (308.85,89.18) and (306.65,91.38) .. (303.93,91.38) .. controls (301.2,91.38) and (299,89.18) .. (299,86.46) -- cycle ;
%Shape: Circle [id:dp5708009822769293] 
\draw  [color={rgb, 255:red, 74; green, 144; blue, 226 }  ,draw opacity=1 ][fill={rgb, 255:red, 157; green, 197; blue, 244 }  ,fill opacity=1 ][line width=1.5]  (350,85.46) .. controls (350,82.74) and (352.2,80.53) .. (354.93,80.53) .. controls (357.65,80.53) and (359.85,82.74) .. (359.85,85.46) .. controls (359.85,88.18) and (357.65,90.38) .. (354.93,90.38) .. controls (352.2,90.38) and (350,88.18) .. (350,85.46) -- cycle ;
%Shape: Circle [id:dp11318897789159399] 
\draw  [color={rgb, 255:red, 74; green, 144; blue, 226 }  ,draw opacity=1 ][fill={rgb, 255:red, 157; green, 197; blue, 244 }  ,fill opacity=1 ][line width=1.5]  (418,85.46) .. controls (418,82.74) and (420.2,80.53) .. (422.93,80.53) .. controls (425.65,80.53) and (427.85,82.74) .. (427.85,85.46) .. controls (427.85,88.18) and (425.65,90.38) .. (422.93,90.38) .. controls (420.2,90.38) and (418,88.18) .. (418,85.46) -- cycle ;
%Shape: Circle [id:dp2969011975656056] 
\draw  [draw opacity=0][fill={rgb, 255:red, 206; green, 224; blue, 245 }  ,fill opacity=1 ][dash pattern={on 1.69pt off 2.76pt}][line width=1.5]  (316,86.07) .. controls (316,83.35) and (318.2,81.15) .. (320.93,81.15) .. controls (323.65,81.15) and (325.85,83.35) .. (325.85,86.07) .. controls (325.85,88.8) and (323.65,91) .. (320.93,91) .. controls (318.2,91) and (316,88.8) .. (316,86.07) -- cycle ;
%Shape: Circle [id:dp271296153564934] 
\draw  [draw opacity=0][fill={rgb, 255:red, 206; green, 224; blue, 245 }  ,fill opacity=1 ][dash pattern={on 1.69pt off 2.76pt}][line width=1.5]  (331,86.07) .. controls (331,83.35) and (333.2,81.15) .. (335.93,81.15) .. controls (338.65,81.15) and (340.85,83.35) .. (340.85,86.07) .. controls (340.85,88.8) and (338.65,91) .. (335.93,91) .. controls (333.2,91) and (331,88.8) .. (331,86.07) -- cycle ;
%Straight Lines [id:da8050181602151449] 
\draw [color={rgb, 255:red, 74; green, 144; blue, 226 }  ,draw opacity=1 ] [dash pattern={on 4.5pt off 4.5pt}]  (363.33,75.77) -- (390.33,75.98) ;
\draw [shift={(392.33,76)}, rotate = 180.46] [fill={rgb, 255:red, 74; green, 144; blue, 226 }  ,fill opacity=1 ][line width=0.08]  [draw opacity=0] (12,-3) -- (0,0) -- (12,3) -- cycle    ;
%Straight Lines [id:da8640754624835305] 
\draw [color={rgb, 255:red, 74; green, 144; blue, 226 }  ,draw opacity=1 ] [dash pattern={on 4.5pt off 4.5pt}]  (296.33,78.77) -- (262.93,78.77) ;
\draw [shift={(260.93,78.77)}, rotate = 360] [fill={rgb, 255:red, 74; green, 144; blue, 226 }  ,fill opacity=1 ][line width=0.08]  [draw opacity=0] (12,-3) -- (0,0) -- (12,3) -- cycle    ;
%Shape: Circle [id:dp1477730530129613] 
\draw  [draw opacity=0][fill={rgb, 255:red, 214; green, 229; blue, 247 }  ,fill opacity=1 ][line width=1.5]  (432,86.07) .. controls (432,83.35) and (434.2,81.15) .. (436.93,81.15) .. controls (439.65,81.15) and (441.85,83.35) .. (441.85,86.07) .. controls (441.85,88.8) and (439.65,91) .. (436.93,91) .. controls (434.2,91) and (432,88.8) .. (432,86.07) -- cycle ;
%Straight Lines [id:da8351854296319265] 
\draw [color={rgb, 255:red, 74; green, 144; blue, 226 }  ,draw opacity=1 ] [dash pattern={on 4.5pt off 4.5pt}]  (433.93,76.15) -- (401.52,75.79) ;
\draw [shift={(399.52,75.77)}, rotate = 0.64] [fill={rgb, 255:red, 74; green, 144; blue, 226 }  ,fill opacity=1 ][line width=0.08]  [draw opacity=0] (12,-3) -- (0,0) -- (12,3) -- cycle    ;
%Shape: Circle [id:dp3916553949697824] 
\draw  [draw opacity=0][fill={rgb, 255:red, 206; green, 224; blue, 245 }  ,fill opacity=1 ][dash pattern={on 1.69pt off 2.76pt}][line width=1.5]  (91,119.07) .. controls (91,116.35) and (93.2,114.15) .. (95.93,114.15) .. controls (98.65,114.15) and (100.85,116.35) .. (100.85,119.07) .. controls (100.85,121.8) and (98.65,124) .. (95.93,124) .. controls (93.2,124) and (91,121.8) .. (91,119.07) -- cycle ;
%Straight Lines [id:da36867510523748903] 
\draw    (98,119) -- (504.33,119) ;
%Shape: Circle [id:dp5234011113052087] 
\draw  [color={rgb, 255:red, 48; green, 65; blue, 69 }  ,draw opacity=1 ][fill={rgb, 255:red, 208; green, 2; blue, 27 }  ,fill opacity=1 ] (93,119) .. controls (93,117.62) and (94.12,116.5) .. (95.5,116.5) .. controls (96.88,116.5) and (98,117.62) .. (98,119) .. controls (98,120.38) and (96.88,121.5) .. (95.5,121.5) .. controls (94.12,121.5) and (93,120.38) .. (93,119) -- cycle ;
%Straight Lines [id:da24409888958654247] 
\draw  [dash pattern={on 0.84pt off 2.51pt}]  (504.33,119) -- (528.33,119) ;
%Shape: Circle [id:dp2960816548666301] 
\draw  [color={rgb, 255:red, 74; green, 144; blue, 226 }  ,draw opacity=1 ][fill={rgb, 255:red, 157; green, 197; blue, 244 }  ,fill opacity=1 ][line width=1.5]  (213,119.07) .. controls (213,116.35) and (215.2,114.15) .. (217.93,114.15) .. controls (220.65,114.15) and (222.85,116.35) .. (222.85,119.07) .. controls (222.85,121.8) and (220.65,124) .. (217.93,124) .. controls (215.2,124) and (213,121.8) .. (213,119.07) -- cycle ;
%Shape: Circle [id:dp21234205548072105] 
\draw  [color={rgb, 255:red, 74; green, 144; blue, 226 }  ,draw opacity=1 ][fill={rgb, 255:red, 157; green, 197; blue, 244 }  ,fill opacity=1 ][line width=1.5]  (229,119.07) .. controls (229,116.35) and (231.2,114.15) .. (233.93,114.15) .. controls (236.65,114.15) and (238.85,116.35) .. (238.85,119.07) .. controls (238.85,121.8) and (236.65,124) .. (233.93,124) .. controls (231.2,124) and (229,121.8) .. (229,119.07) -- cycle ;
%Shape: Circle [id:dp26498678246230056] 
\draw  [color={rgb, 255:red, 74; green, 144; blue, 226 }  ,draw opacity=1 ][fill={rgb, 255:red, 157; green, 197; blue, 244 }  ,fill opacity=1 ][line width=1.5]  (382,119.07) .. controls (382,116.35) and (384.2,114.15) .. (386.93,114.15) .. controls (389.65,114.15) and (391.85,116.35) .. (391.85,119.07) .. controls (391.85,121.8) and (389.65,124) .. (386.93,124) .. controls (384.2,124) and (382,121.8) .. (382,119.07) -- cycle ;
%Shape: Circle [id:dp6742515382677937] 
\draw  [color={rgb, 255:red, 74; green, 144; blue, 226 }  ,draw opacity=1 ][fill={rgb, 255:red, 157; green, 197; blue, 244 }  ,fill opacity=1 ][line width=1.5]  (397,119.07) .. controls (397,116.35) and (399.2,114.15) .. (401.93,114.15) .. controls (404.65,114.15) and (406.85,116.35) .. (406.85,119.07) .. controls (406.85,121.8) and (404.65,124) .. (401.93,124) .. controls (399.2,124) and (397,121.8) .. (397,119.07) -- cycle ;
%Shape: Circle [id:dp483451435601846] 
\draw  [draw opacity=0][fill={rgb, 255:red, 206; green, 224; blue, 245 }  ,fill opacity=1 ][dash pattern={on 1.69pt off 2.76pt}][line width=1.5]  (313,119.07) .. controls (313,116.35) and (315.2,114.15) .. (317.93,114.15) .. controls (320.65,114.15) and (322.85,116.35) .. (322.85,119.07) .. controls (322.85,121.8) and (320.65,124) .. (317.93,124) .. controls (315.2,124) and (313,121.8) .. (313,119.07) -- cycle ;
%Shape: Circle [id:dp5994863874209223] 
\draw  [draw opacity=0][fill={rgb, 255:red, 206; green, 224; blue, 245 }  ,fill opacity=1 ][dash pattern={on 1.69pt off 2.76pt}][line width=1.5]  (328,119.07) .. controls (328,116.35) and (330.2,114.15) .. (332.93,114.15) .. controls (335.65,114.15) and (337.85,116.35) .. (337.85,119.07) .. controls (337.85,121.8) and (335.65,124) .. (332.93,124) .. controls (330.2,124) and (328,121.8) .. (328,119.07) -- cycle ;
%Shape: Circle [id:dp7981446007757752] 
\draw  [draw opacity=0][fill={rgb, 255:red, 214; green, 229; blue, 247 }  ,fill opacity=1 ][line width=1.5]  (430,119.07) .. controls (430,116.35) and (432.2,114.15) .. (434.93,114.15) .. controls (437.65,114.15) and (439.85,116.35) .. (439.85,119.07) .. controls (439.85,121.8) and (437.65,124) .. (434.93,124) .. controls (432.2,124) and (430,121.8) .. (430,119.07) -- cycle ;
%Shape: Circle [id:dp6813813426050471] 
\draw  [color={rgb, 255:red, 74; green, 144; blue, 226 }  ,draw opacity=1 ][fill={rgb, 255:red, 157; green, 197; blue, 244 }  ,fill opacity=1 ][line width=1.5]  (91,50.07) .. controls (91,47.35) and (93.2,45.15) .. (95.93,45.15) .. controls (98.65,45.15) and (100.85,47.35) .. (100.85,50.07) .. controls (100.85,52.8) and (98.65,55) .. (95.93,55) .. controls (93.2,55) and (91,52.8) .. (91,50.07) -- cycle ;
%Straight Lines [id:da4195690198483003] 
\draw    (99,50) -- (505.33,50) ;
%Shape: Circle [id:dp32439107729584515] 
\draw  [color={rgb, 255:red, 48; green, 65; blue, 69 }  ,draw opacity=1 ][fill={rgb, 255:red, 208; green, 2; blue, 27 }  ,fill opacity=1 ] (94,50) .. controls (94,48.62) and (95.12,47.5) .. (96.5,47.5) .. controls (97.88,47.5) and (99,48.62) .. (99,50) .. controls (99,51.38) and (97.88,52.5) .. (96.5,52.5) .. controls (95.12,52.5) and (94,51.38) .. (94,50) -- cycle ;
%Straight Lines [id:da11322714603695738] 
\draw  [dash pattern={on 0.84pt off 2.51pt}]  (505.33,50) -- (529.33,50) ;
%Shape: Circle [id:dp9563595927991452] 
\draw  [color={rgb, 255:red, 74; green, 144; blue, 226 }  ,draw opacity=1 ][fill={rgb, 255:red, 157; green, 197; blue, 244 }  ,fill opacity=1 ][line width=1.5]  (316,50.07) .. controls (316,47.35) and (318.2,45.15) .. (320.93,45.15) .. controls (323.65,45.15) and (325.85,47.35) .. (325.85,50.07) .. controls (325.85,52.8) and (323.65,55) .. (320.93,55) .. controls (318.2,55) and (316,52.8) .. (316,50.07) -- cycle ;
%Shape: Circle [id:dp23766146086673767] 
\draw  [color={rgb, 255:red, 74; green, 144; blue, 226 }  ,draw opacity=1 ][fill={rgb, 255:red, 157; green, 197; blue, 244 }  ,fill opacity=1 ][line width=1.5]  (332,50.07) .. controls (332,47.35) and (334.2,45.15) .. (336.93,45.15) .. controls (339.65,45.15) and (341.85,47.35) .. (341.85,50.07) .. controls (341.85,52.8) and (339.65,55) .. (336.93,55) .. controls (334.2,55) and (332,52.8) .. (332,50.07) -- cycle ;
%Shape: Circle [id:dp29839891400321394] 
\draw  [color={rgb, 255:red, 74; green, 144; blue, 226 }  ,draw opacity=1 ][fill={rgb, 255:red, 157; green, 197; blue, 244 }  ,fill opacity=1 ][line width=1.5]  (432,50.07) .. controls (432,47.35) and (434.2,45.15) .. (436.93,45.15) .. controls (439.65,45.15) and (441.85,47.35) .. (441.85,50.07) .. controls (441.85,52.8) and (439.65,55) .. (436.93,55) .. controls (434.2,55) and (432,52.8) .. (432,50.07) -- cycle ;
%Shape: Circle [id:dp2956197537009374] 
\draw  [draw opacity=0][fill={rgb, 255:red, 206; green, 224; blue, 245 }  ,fill opacity=1 ][dash pattern={on 1.69pt off 2.76pt}][line width=1.5]  (214,50.07) .. controls (214,47.35) and (216.2,45.15) .. (218.93,45.15) .. controls (221.65,45.15) and (223.85,47.35) .. (223.85,50.07) .. controls (223.85,52.8) and (221.65,55) .. (218.93,55) .. controls (216.2,55) and (214,52.8) .. (214,50.07) -- cycle ;
%Shape: Circle [id:dp265261562610803] 
\draw  [draw opacity=0][fill={rgb, 255:red, 214; green, 229; blue, 247 }  ,fill opacity=1 ][line width=1.5]  (230,50.07) .. controls (230,47.35) and (232.2,45.15) .. (234.93,45.15) .. controls (237.65,45.15) and (239.85,47.35) .. (239.85,50.07) .. controls (239.85,52.8) and (237.65,55) .. (234.93,55) .. controls (232.2,55) and (230,52.8) .. (230,50.07) -- cycle ;
%Shape: Circle [id:dp2559903011661103] 
\draw  [draw opacity=0][fill={rgb, 255:red, 214; green, 229; blue, 247 }  ,fill opacity=1 ][line width=1.5]  (382,50.07) .. controls (382,47.35) and (384.2,45.15) .. (386.93,45.15) .. controls (389.65,45.15) and (391.85,47.35) .. (391.85,50.07) .. controls (391.85,52.8) and (389.65,55) .. (386.93,55) .. controls (384.2,55) and (382,52.8) .. (382,50.07) -- cycle ;
%Shape: Circle [id:dp5228068268444318] 
\draw  [draw opacity=0][fill={rgb, 255:red, 214; green, 229; blue, 247 }  ,fill opacity=1 ][line width=1.5]  (397,50.07) .. controls (397,47.35) and (399.2,45.15) .. (401.93,45.15) .. controls (404.65,45.15) and (406.85,47.35) .. (406.85,50.07) .. controls (406.85,52.8) and (404.65,55) .. (401.93,55) .. controls (399.2,55) and (397,52.8) .. (397,50.07) -- cycle ;
%Shape: Circle [id:dp8443107148942653] 
\draw  [draw opacity=0][fill={rgb, 255:red, 206; green, 224; blue, 245 }  ,fill opacity=1 ][dash pattern={on 1.69pt off 2.76pt}][line width=1.5]  (99,258.07) .. controls (99,255.35) and (101.2,253.15) .. (103.93,253.15) .. controls (106.65,253.15) and (108.85,255.35) .. (108.85,258.07) .. controls (108.85,260.8) and (106.65,263) .. (103.93,263) .. controls (101.2,263) and (99,260.8) .. (99,258.07) -- cycle ;
%Straight Lines [id:da24816283399398453] 
\draw    (106,258) -- (512.33,258) ;
%Shape: Circle [id:dp08324169503616818] 
\draw  [color={rgb, 255:red, 48; green, 65; blue, 69 }  ,draw opacity=1 ][fill={rgb, 255:red, 208; green, 2; blue, 27 }  ,fill opacity=1 ] (101,258) .. controls (101,256.62) and (102.12,255.5) .. (103.5,255.5) .. controls (104.88,255.5) and (106,256.62) .. (106,258) .. controls (106,259.38) and (104.88,260.5) .. (103.5,260.5) .. controls (102.12,260.5) and (101,259.38) .. (101,258) -- cycle ;
%Straight Lines [id:da5282717561132034] 
\draw  [dash pattern={on 0.84pt off 2.51pt}]  (512.33,258) -- (536.33,258) ;
%Shape: Circle [id:dp983405162441581] 
\draw  [draw opacity=0][fill={rgb, 255:red, 206; green, 224; blue, 245 }  ,fill opacity=1 ][dash pattern={on 1.69pt off 2.76pt}][line width=1.5]  (249,258.07) .. controls (249,255.35) and (251.2,253.15) .. (253.93,253.15) .. controls (256.65,253.15) and (258.85,255.35) .. (258.85,258.07) .. controls (258.85,260.8) and (256.65,263) .. (253.93,263) .. controls (251.2,263) and (249,260.8) .. (249,258.07) -- cycle ;
%Shape: Circle [id:dp480271699866604] 
\draw  [draw opacity=0][fill={rgb, 255:red, 214; green, 229; blue, 247 }  ,fill opacity=1 ][line width=1.5]  (265,258.07) .. controls (265,255.35) and (267.2,253.15) .. (269.93,253.15) .. controls (272.65,253.15) and (274.85,255.35) .. (274.85,258.07) .. controls (274.85,260.8) and (272.65,263) .. (269.93,263) .. controls (267.2,263) and (265,260.8) .. (265,258.07) -- cycle ;
%Shape: Circle [id:dp6230716848268002] 
\draw  [draw opacity=0][fill={rgb, 255:red, 214; green, 229; blue, 247 }  ,fill opacity=1 ][line width=1.5]  (424,258.07) .. controls (424,255.35) and (426.2,253.15) .. (428.93,253.15) .. controls (431.65,253.15) and (433.85,255.35) .. (433.85,258.07) .. controls (433.85,260.8) and (431.65,263) .. (428.93,263) .. controls (426.2,263) and (424,260.8) .. (424,258.07) -- cycle ;
%Straight Lines [id:da6203930953648775] 
\draw [color={rgb, 255:red, 74; green, 144; blue, 226 }  ,draw opacity=1 ] [dash pattern={on 4.5pt off 4.5pt}]  (169,250.77) -- (203.93,252.08) ;
\draw [shift={(205.93,252.15)}, rotate = 182.15] [fill={rgb, 255:red, 74; green, 144; blue, 226 }  ,fill opacity=1 ][line width=0.08]  [draw opacity=0] (12,-3) -- (0,0) -- (12,3) -- cycle    ;
%Shape: Circle [id:dp2647039874041013] 
\draw  [color={rgb, 255:red, 74; green, 144; blue, 226 }  ,draw opacity=1 ][fill={rgb, 255:red, 157; green, 197; blue, 244 }  ,fill opacity=1 ][line width=1.5]  (153,259.07) .. controls (153,256.35) and (155.2,254.15) .. (157.93,254.15) .. controls (160.65,254.15) and (162.85,256.35) .. (162.85,259.07) .. controls (162.85,261.8) and (160.65,264) .. (157.93,264) .. controls (155.2,264) and (153,261.8) .. (153,259.07) -- cycle ;
%Shape: Circle [id:dp2249148863301833] 
\draw  [color={rgb, 255:red, 74; green, 144; blue, 226 }  ,draw opacity=1 ][fill={rgb, 255:red, 157; green, 197; blue, 244 }  ,fill opacity=1 ][line width=1.5]  (325,258.46) .. controls (325,255.74) and (327.2,253.53) .. (329.93,253.53) .. controls (332.65,253.53) and (334.85,255.74) .. (334.85,258.46) .. controls (334.85,261.18) and (332.65,263.38) .. (329.93,263.38) .. controls (327.2,263.38) and (325,261.18) .. (325,258.46) -- cycle ;
%Shape: Circle [id:dp9373819332974472] 
\draw  [color={rgb, 255:red, 74; green, 144; blue, 226 }  ,draw opacity=1 ][fill={rgb, 255:red, 157; green, 197; blue, 244 }  ,fill opacity=1 ][line width=1.5]  (383,257.46) .. controls (383,254.74) and (385.2,252.53) .. (387.93,252.53) .. controls (390.65,252.53) and (392.85,254.74) .. (392.85,257.46) .. controls (392.85,260.18) and (390.65,262.38) .. (387.93,262.38) .. controls (385.2,262.38) and (383,260.18) .. (383,257.46) -- cycle ;
%Shape: Circle [id:dp9496520932012501] 
\draw  [draw opacity=0][fill={rgb, 255:red, 206; green, 224; blue, 245 }  ,fill opacity=1 ][dash pattern={on 1.69pt off 2.76pt}][line width=1.5]  (348,258.07) .. controls (348,255.35) and (350.2,253.15) .. (352.93,253.15) .. controls (355.65,253.15) and (357.85,255.35) .. (357.85,258.07) .. controls (357.85,260.8) and (355.65,263) .. (352.93,263) .. controls (350.2,263) and (348,260.8) .. (348,258.07) -- cycle ;
%Shape: Circle [id:dp8109626562783527] 
\draw  [draw opacity=0][fill={rgb, 255:red, 206; green, 224; blue, 245 }  ,fill opacity=1 ][dash pattern={on 1.69pt off 2.76pt}][line width=1.5]  (363,258.07) .. controls (363,255.35) and (365.2,253.15) .. (367.93,253.15) .. controls (370.65,253.15) and (372.85,255.35) .. (372.85,258.07) .. controls (372.85,260.8) and (370.65,263) .. (367.93,263) .. controls (365.2,263) and (363,260.8) .. (363,258.07) -- cycle ;
%Straight Lines [id:da1319174902785082] 
\draw [color={rgb, 255:red, 74; green, 144; blue, 226 }  ,draw opacity=1 ] [dash pattern={on 4.5pt off 4.5pt}]  (396.33,249.77) -- (423.33,249.98) ;
\draw [shift={(425.33,250)}, rotate = 180.46] [fill={rgb, 255:red, 74; green, 144; blue, 226 }  ,fill opacity=1 ][line width=0.08]  [draw opacity=0] (12,-3) -- (0,0) -- (12,3) -- cycle    ;
%Straight Lines [id:da7800654368116888] 
\draw [color={rgb, 255:red, 74; green, 144; blue, 226 }  ,draw opacity=1 ] [dash pattern={on 4.5pt off 4.5pt}]  (319.33,250.77) -- (285.93,250.77) ;
\draw [shift={(283.93,250.77)}, rotate = 360] [fill={rgb, 255:red, 74; green, 144; blue, 226 }  ,fill opacity=1 ][line width=0.08]  [draw opacity=0] (12,-3) -- (0,0) -- (12,3) -- cycle    ;
%Shape: Circle [id:dp4555239169462927] 
\draw  [draw opacity=0][fill={rgb, 255:red, 206; green, 224; blue, 245 }  ,fill opacity=1 ][dash pattern={on 1.69pt off 2.76pt}][line width=1.5]  (98,291.07) .. controls (98,288.35) and (100.2,286.15) .. (102.93,286.15) .. controls (105.65,286.15) and (107.85,288.35) .. (107.85,291.07) .. controls (107.85,293.8) and (105.65,296) .. (102.93,296) .. controls (100.2,296) and (98,293.8) .. (98,291.07) -- cycle ;
%Straight Lines [id:da4933532240744781] 
\draw    (105,291) -- (511.33,291) ;
%Shape: Circle [id:dp13439931381192305] 
\draw  [color={rgb, 255:red, 48; green, 65; blue, 69 }  ,draw opacity=1 ][fill={rgb, 255:red, 208; green, 2; blue, 27 }  ,fill opacity=1 ] (100,291) .. controls (100,289.62) and (101.12,288.5) .. (102.5,288.5) .. controls (103.88,288.5) and (105,289.62) .. (105,291) .. controls (105,292.38) and (103.88,293.5) .. (102.5,293.5) .. controls (101.12,293.5) and (100,292.38) .. (100,291) -- cycle ;
%Straight Lines [id:da9406358808741299] 
\draw  [dash pattern={on 0.84pt off 2.51pt}]  (511.33,291) -- (535.33,291) ;
%Shape: Circle [id:dp43662473439456373] 
\draw  [color={rgb, 255:red, 74; green, 144; blue, 226 }  ,draw opacity=1 ][fill={rgb, 255:red, 157; green, 197; blue, 244 }  ,fill opacity=1 ][line width=1.5]  (249,291.07) .. controls (249,288.35) and (251.2,286.15) .. (253.93,286.15) .. controls (256.65,286.15) and (258.85,288.35) .. (258.85,291.07) .. controls (258.85,293.8) and (256.65,296) .. (253.93,296) .. controls (251.2,296) and (249,293.8) .. (249,291.07) -- cycle ;
%Shape: Circle [id:dp08720523554162207] 
\draw  [color={rgb, 255:red, 74; green, 144; blue, 226 }  ,draw opacity=1 ][fill={rgb, 255:red, 157; green, 197; blue, 244 }  ,fill opacity=1 ][line width=1.5]  (264,291.07) .. controls (264,288.35) and (266.2,286.15) .. (268.93,286.15) .. controls (271.65,286.15) and (273.85,288.35) .. (273.85,291.07) .. controls (273.85,293.8) and (271.65,296) .. (268.93,296) .. controls (266.2,296) and (264,293.8) .. (264,291.07) -- cycle ;
%Shape: Circle [id:dp27518156166438845] 
\draw  [color={rgb, 255:red, 74; green, 144; blue, 226 }  ,draw opacity=1 ][fill={rgb, 255:red, 157; green, 197; blue, 244 }  ,fill opacity=1 ][line width=1.5]  (98,222.07) .. controls (98,219.35) and (100.2,217.15) .. (102.93,217.15) .. controls (105.65,217.15) and (107.85,219.35) .. (107.85,222.07) .. controls (107.85,224.8) and (105.65,227) .. (102.93,227) .. controls (100.2,227) and (98,224.8) .. (98,222.07) -- cycle ;
%Straight Lines [id:da5666751569095746] 
\draw    (106,222) -- (512.33,222) ;
%Shape: Circle [id:dp02722892985140235] 
\draw  [color={rgb, 255:red, 48; green, 65; blue, 69 }  ,draw opacity=1 ][fill={rgb, 255:red, 208; green, 2; blue, 27 }  ,fill opacity=1 ] (101,222) .. controls (101,220.62) and (102.12,219.5) .. (103.5,219.5) .. controls (104.88,219.5) and (106,220.62) .. (106,222) .. controls (106,223.38) and (104.88,224.5) .. (103.5,224.5) .. controls (102.12,224.5) and (101,223.38) .. (101,222) -- cycle ;
%Straight Lines [id:da9474711890303887] 
\draw  [dash pattern={on 0.84pt off 2.51pt}]  (512.33,222) -- (536.33,222) ;
%Shape: Circle [id:dp350751600111154] 
\draw  [color={rgb, 255:red, 74; green, 144; blue, 226 }  ,draw opacity=1 ][fill={rgb, 255:red, 157; green, 197; blue, 244 }  ,fill opacity=1 ][line width=1.5]  (347,222.07) .. controls (347,219.35) and (349.2,217.15) .. (351.93,217.15) .. controls (354.65,217.15) and (356.85,219.35) .. (356.85,222.07) .. controls (356.85,224.8) and (354.65,227) .. (351.93,227) .. controls (349.2,227) and (347,224.8) .. (347,222.07) -- cycle ;
%Shape: Circle [id:dp11795147762659963] 
\draw  [color={rgb, 255:red, 74; green, 144; blue, 226 }  ,draw opacity=1 ][fill={rgb, 255:red, 157; green, 197; blue, 244 }  ,fill opacity=1 ][line width=1.5]  (363,222.07) .. controls (363,219.35) and (365.2,217.15) .. (367.93,217.15) .. controls (370.65,217.15) and (372.85,219.35) .. (372.85,222.07) .. controls (372.85,224.8) and (370.65,227) .. (367.93,227) .. controls (365.2,227) and (363,224.8) .. (363,222.07) -- cycle ;
%Shape: Circle [id:dp26892650273060636] 
\draw  [draw opacity=0][fill={rgb, 255:red, 206; green, 224; blue, 245 }  ,fill opacity=1 ][dash pattern={on 1.69pt off 2.76pt}][line width=1.5]  (249,222.07) .. controls (249,219.35) and (251.2,217.15) .. (253.93,217.15) .. controls (256.65,217.15) and (258.85,219.35) .. (258.85,222.07) .. controls (258.85,224.8) and (256.65,227) .. (253.93,227) .. controls (251.2,227) and (249,224.8) .. (249,222.07) -- cycle ;
%Shape: Circle [id:dp9143445123774641] 
\draw  [draw opacity=0][fill={rgb, 255:red, 214; green, 229; blue, 247 }  ,fill opacity=1 ][line width=1.5]  (265,222.07) .. controls (265,219.35) and (267.2,217.15) .. (269.93,217.15) .. controls (272.65,217.15) and (274.85,219.35) .. (274.85,222.07) .. controls (274.85,224.8) and (272.65,227) .. (269.93,227) .. controls (267.2,227) and (265,224.8) .. (265,222.07) -- cycle ;
%Shape: Circle [id:dp3435190884532615] 
\draw  [draw opacity=0][fill={rgb, 255:red, 214; green, 229; blue, 247 }  ,fill opacity=1 ][line width=1.5]  (422,222.07) .. controls (422,219.35) and (424.2,217.15) .. (426.93,217.15) .. controls (429.65,217.15) and (431.85,219.35) .. (431.85,222.07) .. controls (431.85,224.8) and (429.65,227) .. (426.93,227) .. controls (424.2,227) and (422,224.8) .. (422,222.07) -- cycle ;
%Shape: Circle [id:dp9080968301785849] 
\draw  [draw opacity=0][fill={rgb, 255:red, 206; green, 224; blue, 245 }  ,fill opacity=1 ][dash pattern={on 1.69pt off 2.76pt}][line width=1.5]  (348,291.07) .. controls (348,288.35) and (350.2,286.15) .. (352.93,286.15) .. controls (355.65,286.15) and (357.85,288.35) .. (357.85,291.07) .. controls (357.85,293.8) and (355.65,296) .. (352.93,296) .. controls (350.2,296) and (348,293.8) .. (348,291.07) -- cycle ;
%Shape: Circle [id:dp12070258028660341] 
\draw  [draw opacity=0][fill={rgb, 255:red, 206; green, 224; blue, 245 }  ,fill opacity=1 ][dash pattern={on 1.69pt off 2.76pt}][line width=1.5]  (363,291.07) .. controls (363,288.35) and (365.2,286.15) .. (367.93,286.15) .. controls (370.65,286.15) and (372.85,288.35) .. (372.85,291.07) .. controls (372.85,293.8) and (370.65,296) .. (367.93,296) .. controls (365.2,296) and (363,293.8) .. (363,291.07) -- cycle ;
%Shape: Circle [id:dp8368364724551958] 
\draw  [color={rgb, 255:red, 74; green, 144; blue, 226 }  ,draw opacity=1 ][fill={rgb, 255:red, 157; green, 197; blue, 244 }  ,fill opacity=1 ][line width=1.5]  (424,290.46) .. controls (424,287.74) and (426.2,285.53) .. (428.93,285.53) .. controls (431.65,285.53) and (433.85,287.74) .. (433.85,290.46) .. controls (433.85,293.18) and (431.65,295.38) .. (428.93,295.38) .. controls (426.2,295.38) and (424,293.18) .. (424,290.46) -- cycle ;

% Text Node
\draw (457,63.4) node [anchor=north west][inner sep=0.75pt]    {$t\in ( 0,T)$};
% Text Node
\draw (456,96.4) node [anchor=north west][inner sep=0.75pt]    {$t=T$};
% Text Node
\draw (460,28.4) node [anchor=north west][inner sep=0.75pt]    {$t=0$};
% Text Node
\draw (83,28) node [anchor=north west][inner sep=0.75pt]   [align=left] {$\displaystyle q_{1}$};
% Text Node
\draw (130,65) node [anchor=north west][inner sep=0.75pt]   [align=left] {$\displaystyle q_{1}$};
% Text Node
\draw (202,127) node [anchor=north west][inner sep=0.75pt]   [align=left] {$\displaystyle q_{1}$};
% Text Node
\draw (306,28) node [anchor=north west][inner sep=0.75pt]   [align=left] {$\displaystyle q_{2}$};
% Text Node
\draw (335,28) node [anchor=north west][inner sep=0.75pt]   [align=left] {$\displaystyle q_3$};
% Text Node
\draw (290,65) node [anchor=north west][inner sep=0.75pt]   [align=left] {$\displaystyle q_{2}$};
% Text Node
\draw (352,65) node [anchor=north west][inner sep=0.75pt]   [align=left] {$\displaystyle q_{3}$};
% Text Node
\draw (375,127) node [anchor=north west][inner sep=0.75pt]   [align=left] {$\displaystyle q_{3}$};
% Text Node
\draw (403,127) node [anchor=north west][inner sep=0.75pt]   [align=left] {$\displaystyle q_{4}$};
% Text Node
\draw (235.93,127) node [anchor=north west][inner sep=0.75pt]   [align=left] {$\displaystyle q_{2}$};
% Text Node
\draw (425,28) node [anchor=north west][inner sep=0.75pt]   [align=left] {$\displaystyle q_{4}$};
% Text Node
\draw (420,65) node [anchor=north west][inner sep=0.75pt]   [align=left] {$\displaystyle q_{4}$};
% Text Node
\draw (464,235.4) node [anchor=north west][inner sep=0.75pt]    {$t\in ( 0,T)$};
% Text Node
\draw (463,268.4) node [anchor=north west][inner sep=0.75pt]    {$t=T$};
% Text Node
\draw (467,200.4) node [anchor=north west][inner sep=0.75pt]    {$t=0$};
% Text Node
\draw (83,200) node [anchor=north west][inner sep=0.75pt]   [align=left] {$\displaystyle q_{1}$};
% Text Node
\draw (144,237) node [anchor=north west][inner sep=0.75pt]   [align=left] {$\displaystyle q_{1}$};
% Text Node
\draw (240,300) node [anchor=north west][inner sep=0.75pt]   [align=left] {$\displaystyle q_{1}$};
% Text Node
\draw (339,200) node [anchor=north west][inner sep=0.75pt]   [align=left] {$\displaystyle q_{2}$};
% Text Node
\draw (365,200) node [anchor=north west][inner sep=0.75pt]   [align=left] {$\displaystyle q_3$};
% Text Node
\draw (326,237) node [anchor=north west][inner sep=0.75pt]   [align=left] {$\displaystyle q_{2}$};
% Text Node
\draw (381,237) node [anchor=north west][inner sep=0.75pt]   [align=left] {$\displaystyle q_3$};
% Text Node
\draw (270,300) node [anchor=north west][inner sep=0.75pt]   [align=left] {$\displaystyle q_{2}$};
% Text Node
\draw (430.93,300) node [anchor=north west][inner sep=0.75pt]   [align=left] {$\displaystyle q_3$};
\end{tikzpicture}
	}
	\caption{Two examples of frozen orbits provided by Theorem \ref{thm:solution_helium} for $n=4$ (above) and $n=3$ (below). In particular, this illustrates the qualitative behaviour for a solution of \eqref{eq:helium_mu} when $\mu$ is small enough as in Theorem \ref{thm:mu_to_0}.}\label{fig:orbits}
\end{figure}
As a result, after a time reflection, we obtain a generalized periodic solution of period $2T$ (cf. \cite[Theorem 1.1]{helium_frozen}). In last part of the paper we make the connection with Schubart's orbits more apparent. We introduce a parameter $\mu \in (0,1]$ to dampen the repulsive interaction and investigate the asymptotic behaviour as $\mu\to 0$ of the system 
\begin{equation}
	\label{eq:helium_mu}
	\ddot{q}_i = f_i'(\vert q_i \vert ) - \mu \sum_{j=1}^{i-1} g'_{ij}(\vert q_i-q_j\vert )+\mu \sum_{j = i+1}^n g'_{ij}(\vert q_j-q_i \vert).
\end{equation} 
Theorem \ref{thm:solution_helium} implies that, for any $\mu\in (0,1]$, there exists at least one $\mu-$frozen planet orbit which solves \eqref{eq:helium_mu}. It is worthwhile noticing that, due to a lack of information about uniqueness,  the family of $\mu-$frozen planet orbits needs not be continuous in $\mu$. The limit profile as $\mu \to 0$ is obtained from the unique continuous \emph{brake solution} which solves the non-autonomous boundary value problem
\begin{equation*}
	\begin{cases}
       \ddot{\eta} = F(\eta,t) \\
       \eta(0) =0\\ 
	   \dot \eta(nT) = 0
	\end{cases}, \text{ where } \,F(\eta,t) = \begin{cases}
	f_1(\eta), \text{ if }t \in [0,T] \\
	f_i(\eta), \text{ if }t \in ((i-1)T,iT], i>1
\end{cases}
\end{equation*}
We will call any such curve $\eta$ an $nT-$\emph{brake}. Note that $\eta$ is $C^1$ except at $t=0$; moreover, under our assumptions, we will show that it is unique and monotone increasing. In a similar fashion, we will call a \emph{folded} $nT-$\emph{brake} the curve $x:[0,T]\to \mathbb{R}^n $ defined as:
\begin{equation}
	\label{eq:nT_brake}
	x(t) = (x_1(t),\dots, x_n(t)), \quad x_{i}(t) = \begin{cases}
		\eta(t+(i-1) T) &\text{ if } i \text{ is odd}\\
			\eta(i T-t) &\text{ if } i \text{ is even}.
	\end{cases}
\end{equation}
Note that $x_{i+1}\ge x_{i}$, since $\eta$ is increasing. Moreover, $x$ has exactly $n-1$ double collisions, $\lfloor (n-1)/2 \rfloor$ of which appear at $t=0$, while $\lfloor n/2 \rfloor$ at $t = T$. We have the following result.
\begin{thm}
	\label{thm:mu_to_0}
		For any $T>0$ and for any choice of $f_i,g_{ij}$ satisfying assumptions \eqref{assumption:values_f_g}-\eqref{assumption:convexity}, there exists a family of $\mu-$frozen planet orbits solving \eqref{eq:helium_mu}, converging in $C^2_{loc}(0,T)$ and in $H^1([0,T],\mathbb{R}^n)$ to the corresponding folded $nT-$brake.
\end{thm}

In Chemistry, Theorem \ref{thm:mu_to_0} can apply to highly charged cations in the limit of small $1/Z$ (after scaling $q_i\mapsto Z^{1/3}q_i$) and fixed $n$. Examples of highly ionized heavy-metal cations include $\text{Fe}^{+17}$, $\text{Xe}^{44+}$, or $\text{U}^{91+}$, which may occur in physical environments such as electron beam ion traps, accelerators, or stellar interiors.

The existence of periodic solutions of \eqref{eq:helium_newton} has profound implications in the quantum mechanics description of atoms (see for instance \cite{gutzwiller, rost_tanner}). Periodic orbits form the backbone of semiclassical approximations and are of the utmost importance in the path integral formulation of quantum mechanics, as well as in certain formulas connecting quantization to classical solution of a deterministic system, such as the Gutzwiller trace formula. This connection is highly relevant for interpreting spectral properties and predicting quantum states (see \cite{di2024ground}). In fact, in certain quantum eigenstates of classical Hamiltonians, the spatial probability density is not uniform but instead shows enhanced localization along the paths of classical periodic orbits. This surprising phenomenon, called quantum scarring, demonstrates how classical structures can persist and influence quantum mechanics well beyond the semiclassical limit, and has implications for spectroscopy, quantum transport, and the design of nanoscale systems. Our results have the potential to not only advance the theoretical understanding of the 
$n$-electron problem, but also to open new avenues for computational techniques in quantum chemistry and atomic physics. By combining variational and perturbative methods, they offer a foundation for numerically constructing eigenfunctions that concentrate around the periodic solutions we identify.

The approach of the paper is variational.
We wish to characterise solutions of \eqref{eq:helium_equation} as critical points of the Lagrangian action functional:
\begin{equation*}
	\mathcal{A}(q) = \int_0^T \sum_{i=1}^n \frac12 \vert\dot{q}_i\vert^2+f_i(\vert q_i\vert )-\sum_{j<i}g_{ij}(\vert q_i-q_j\vert)
\end{equation*}
on the Hilbert space $H^1([0,T],\mathbb{R}^n)$. However, the functional $\mathcal{A}$ is undefined on the collision manifold \[
\bigcup_{i<j}\left\{q \in H^1([0,T],\mathbb{R}^n) : \exists t \in [0,T] : q_i(t) = q_j(t)\right\}
\]
 due to the singularities of opposite sign of $f_i$ and $g_{ij}$. Moreover, collisions between the first electron and the nucleus cannot be avoided. Therefore, the notion of critical point needs to be generalised in a suitable way.

To cope with this problem, we introduce a family of $C^{1,1}$ functionals $\mathcal{A}_\ve$  defined on the open set of $H^1([0,T],\mathbb{R}^n) \cap\{q_1(0) =0\}$
\[
	\mathcal{H} = \{(q_1,\dots,q_n)\in H^1([0,T],\mathbb{R}^n): q_i<q_{i+1},\ q_1(0)=0\}.
\]
In any functional $\mathcal{A}_\ve$ the singularities of the functions $f_i$ are dampened, so that the electrons are free to move on the real line. The functionals are defined as
\[
\mathcal{A}_\ve (q)= \int_0^T \sum_{i=1}^n \frac12 \vert\dot{q}_i\vert^2+f^\ve_i(\vert q_i\vert )-\sum_{j<i}g_{ij}(\vert q_i-q_j\vert),
\]
where $f^\ve_i$ is a family of $C^{1,1}(\mathbb{R})$ functions converging pointwise to $f_i$ on $(0,+\infty)$. From a variational perspective, we characterise frozen planet orbits as saddle points of $\mathcal{A}_\ve$, using a Lusternik-Schnirelmann-type theory for manifolds with boundary, developed in \cite{majer95,majer_terracini93}.

The structure of the paper is the following. In Section \ref{sec:critical_point_theorem}, we prove a version of Theorem \ref{thm:solution_helium} (Theorem \ref{thm:general_existence_result}) under a modification of assumption \ref{assumption:attraction_infinity} and some additional assumptions on $f_i$, satisfied by each one of the $\ve-$approximations $f_i^\ve$. In Section \ref{sec:proof_thm1}, we show how this result implies the existence of $\mathcal{A}_\ve$-critical points and prove that they converge to a solution of \eqref{eq:helium_equation}, yielding Theorem \ref{thm:solution_helium}. Finally, Section \ref{sec:zero_charge} is devoted to the proof Theorem \ref{thm:mu_to_0}.

\section{A critical point Theorem}\label{sec:critical_point_theorem}
	In this section we adopt a critical point theory approach to show that there exists at least one solution of \eqref{eq:helium_equation}, satisfying 
	\begin{equation}
    \label{eq:boundary_conditions}
	\begin{cases}
		q_1(0)=0,\ \dot{q}_1(T)=0 \\
		\dot q_i(0) = \dot q_i(T) = 0\quad\text{for}\ i\ge 2.
	\end{cases}
	\end{equation}
	As announced in the Introduction, this will be done for a class of \emph{smoothed} problems, in which the attractive forces $f'_i$ are replaced by bounded interactions. For this reason, we will work under a slight modification of assumptions \eqref{assumption:values_f_g}-\eqref{assumption:attraction_infinity}. In particular, we require that each $f_i$ satisfies additionally:
	\begin{equation}
		\label{assumption:extra_assumptions_f}
		\begin{cases}
			\tag{H5}
		f_i \in C^{1,1}(\mathbb{R}) \\
		f'_i<0 \\
		f_i'(s)<-\nu\ \text{ for some }\nu>0 \text{ and all } s\le0.
		\end{cases} 
	\end{equation}
    Moreover, we will assume that the part of assumption \eqref{assumption:homogenuity} involving $f_i$ and assumption \eqref{assumption:convexity} hold for all $s\in \mathbb
    {R}$. Clearly, these additional hypotheses will be consistent with the choice of the $\ve$-approximating functions $f_i^\ve$ (cf. Section \ref{sec:proof_thm1}).
    In addition, we will work under the following stronger version of \eqref{assumption:attraction_infinity}
    \begin{equation}
    \label{assumption:attraction_infinity_2}
			\tag{H3'}
            \exists s_0>0 : \forall s>s_0, \forall t_i>0 : s-t_i\ge \frac{s}{2n},
        \quad  f'_j(s)-\sum_{i=0}^{j-1}g_{ij}'(s-t_i)< 0.
    \end{equation}
     In Section \ref{sec:proof_thm1} we will prove that assumption \eqref{assumption:attraction_infinity} gives an a priori bound on the $C^0$ norm of any solution of \eqref{eq:helium_newton} with boundary conditions given by \eqref{eq:boundary_conditions}. We will then show how to modify any function satisfying \eqref{assumption:values_f_g}-\eqref{assumption:homogenuity} outside a compact set in order for Assumption \eqref{assumption:attraction_infinity_2}  to hold.
    
    We will prove the following:
	 \begin{thm}
		\label{thm:general_existence_result}
         Let $f_i$ and $g_{ij}$ be functions satisfying the assumptions \eqref{assumption:values_f_g},\eqref{assumption:homogenuity}, \eqref{assumption:attraction_infinity_2} and \eqref{assumption:extra_assumptions_f}, for all $i<j$. Then, for any $T>0$, there exists a collisionless solution $q$ of \eqref{eq:helium_equation} which satisfies the boundary conditions \[
         \dot q_i(0) = \dot q_i(T) = 0 \text{ for } i\ge2 \text{ and } \dot{q}_1(T) = q_1(0)=0.
         \]
         In particular, $q_i(t) < q_j(t)$ for all $t\in [0,T]$ and $i < j$ and $q_1(t)>0$, as soon as $t>0.$
	\end{thm}
    We characterise solutions of Theorem \ref{thm:general_existence_result} as critical points of the action functional $\mathcal{A}$ defined as
      \begin{equation}
    	\label{eq:def_A}
    	\mathcal{A}(q) = \int_0^T\sum_{i=1}^n  \frac{1}{2} \vert\dot{q}_i\vert^2+f_i(q_i)-\sum_{i<j}g_{ij}(q_j-q_i)
    \end{equation}
    in the following open subset $\mathcal{H}$ of the Hilbert space $H^1([0,T],\mathbb{R}^n)\cap\{q_1(0) =0\}$      
    \[
		\mathcal{H} = \{(q_1,\dots,q_n)\in H^1([0,T],\mathbb{R}^n): q_i<q_{i+1},\ q_1(0)=0\}.
    \]
    To show that there exists a critical point of $\mathcal{A}$ in $\mathcal{H}$, we adopt the Lusternik-Schnirelmann theory for manifolds with boundary developed in \cite{majer95,majer_terracini93}. Let us recall the setting and the statement of \cite[Lemma 2.1]{majer_terracini93}.
    Let $\mathcal{G}: \mathcal{H} \to \mathbb{R}$ be a $C^2$ functional. The following lemma holds true.
    
    \begin{lemma}
    	\label{lemma:critical_point}
    	Let us assume that there exists $c^*,\tilde c,b$ and $ \tilde{b}$ in $ \mathbb{R}$ such that $c^*<\tilde{c}$, $\tilde b <b$ and the following conditions hold:
    	\begin{enumerate}[label = \roman*)]
    		\item $\overline{\{\mathcal{A}\le\tilde{c}\}\cap\{\mathcal{G}\le b\}}\footnote{Here, the closure is meant in $H^1([0,T])$ with respect to the strong topology.}\subset \mathcal{H}$  and it is bounded;
    		\item the level sets $b$ of $\mathcal{G}$ are regular, i.e.,
    		\begin{equation}    
    			\label{eq:regularity_boundary_CPTlemma}	 
    			\nabla \mathcal{G}(x) \ne 0,\ \text{ for any }x\in\{\mathcal{A}\le \tilde c\}\cap\{\mathcal{G} \ge \tilde b\};
    		\end{equation}
    		\item any sequence $(x_k) \subseteq \mathcal{H}$ such that
    		\begin{equation}
    			\label{eq:ps_condition1_CPTlemma}
    			\lim\limits_{k\to+\infty} \mathcal{A} (x_k)=c^*,\quad  \limsup\limits_{k\to+\infty} \mathcal{G}(x_k)\le b, \quad \lim\limits_{k\to+\infty}\nabla \mathcal A(x_k)=0
    		\end{equation}
    		has a convergent subsequence;
    		\item any sequence $(x_k)\subset\mathcal{H}$ such that
    		\begin{equation}
    			\label{eq:ps_condition2_CPTlemma}
    			\lim\limits_{k\to+\infty}\mathcal{A} (x_k) = c^*,\quad  \lim\limits_{k\to+\infty}\mathcal{G}(x_k)= b, \quad \lim\limits_{k\to+\infty}\left[\nabla \mathcal A(x_k)-\lambda_k \nabla \mathcal{G}(x_k)\right] = 0,
    		\end{equation}
    		for $\lambda_k\ge0$, has a convergent subsequence;
    		\item for any $\lambda>0$, we have:    
    		\begin{equation}
    			\label{eq:angle_gradients_CPTlemma}
    			\nabla \mathcal{A}(x)\ne \lambda \nabla \mathcal{G}(x),\ \text{ for any } x \in \{\mathcal{A}= c^*\}\cap \{\mathcal{G}= b\};
     		\end{equation}
      		\item for any $\ve\in (0,\tilde c-c^*]$ consider the sets: 
        	\begin{equation}
      			\label{eq:non_deformable_sets_CPTlemma}
      			X_{\pm\ve}:= \{\mathcal{A}\le c^*\pm\ve\}\cup\left(\{\mathcal{A} \le \tilde c \}\cap \{\mathcal{G}\ge b \}\right).
      		\end{equation}
      		The set $X_\ve$ cannot be deformed into $X_{-\ve}$.
        \end{enumerate}
        Then, there exists a critical point of $\mathcal{A}$ lying in $\{\mathcal A = c^*\}\cap \{\mathcal{G}\le b\}.$
    \end{lemma}
    
    The proof of the Lemma relies on a deformation scheme sketched in Appendix \ref{appendix} (see Lemma \ref{lemma:deformation_lemma}). Let us briefly discuss here the assumptions of Lemma \ref{lemma:critical_point} and introduce our choice of $\mathcal{G}$.
        
    Points $ii),v)$ and $vi)$ are crucial for the deformation scheme to work. In fact, the argument core is that, if there are no critical points of $\mathcal{A}$ at level $c^*$, point $v)$ allows to build a $\mathcal{A}$-gradient like deformation leaving $\{\mathcal{G}\ge b\}$ positively invariant. This in turn violates the non deformability condition in $vi)$.
    Point $v)$ can be rephrased in terms of the gradients of the two functionals as follows. We require that $\nabla \mathcal{A}$ and $\nabla \mathcal{G}$ form a non-zero angle, this entails that there is a direction in the positive cone  spanned by $-\nabla \mathcal{A}$ and $\nabla \mathcal{G}$ (i.e. the linear combinations with positive coefficients), along which $\mathcal{A}$ decreases and $\mathcal{G}$ increases. Hypotheses $iii)-iv)$ are Palais-Smale (PS for short) compactness conditions on levels of $\mathcal{A},$ relatively to the sublevel $\{\mathcal{G}\le b\}$. This suggests that a reasonable choice of $\mathcal{G}$ should take into account both the non-compactness sources of the problem: \emph{escape at infinity} and \emph{collisions}. A sensible guess would be the following function 
    	\begin{equation}
    	\label{eq:def_G}
    	\mathcal{G}^\beta(q) = \sum_{i=1}^{n-1}\mathcal{G}^\beta_i(q), \quad 
    	\mathcal{G}^\beta_i(q) = \int_0^T \vert q_{i+1}-q_{i}\vert^2+\frac{1}{\vert q_{i+1}-q_{i}\vert^2}+\beta\sum\limits_{i<j} g_{ij}(q_j-q_i).
    \end{equation}    
	Note that assumption $i)$ of Lemma \ref{lemma:critical_point} is trivially satisfied by $\mathcal{G}^\beta$. Indeed, for any $q\in\{\mathcal{A}\le\tilde{c}\}\cap\{\mathcal{G}^\beta\le b\}$, it is straightforward to check that $\|\dot{q}\|_2^2\le2(\tilde{c}+ b/\beta)$. Since $\mathcal{G}^\beta$ contains a strong force term, we also have $\min_{i, t\in[0,T]}(q_{i+1}(t)-q_i(t))$ is uniformly bounded from below. 
    
	Unfortunately, it is quite hard to check point $v)$ for this choice of $\mathcal{G}^\beta$.  It is more convenient to apply Lemma \ref{lemma:critical_point} to another functional $\mathcal{G}^\beta_{\lambda,c}$ belonging to the family
    	$\Glambda^\beta(q) = \mathcal{G}^\beta(q)+\lambda(c-\mathcal{A}(q))$, for some  $\lambda, c\ge0$. 
		To simplify the notation, we will omit the $\beta$. This amounts to replace, in the discussion below (Proposition \ref{prop:no_solution_eigenvalue}), the repulsion terms $g_{ij}$ inside the action $\mathcal{A}$ with $\left(1+\frac{\beta}{1+\lambda}\right)g_{ij}$ for $\beta$ small enough. We will work with the functional 
		\begin{equation}
    	\label{eq:def_modified_G}
    	\Glambda(q) = \mathcal{G}(q)+\lambda(c-\mathcal{A}(q)),\quad  \lambda, c\ge0. 
    \end{equation}
    The remaining part of this section is devoted to the verification of the assumptions of Lemma \ref{lemma:critical_point} for functionals $\mathcal{A}$ and $\mathcal{G}_{\lambda,c}$.

\subsection{Assumption $ii)$ and $v)$ of Lemma \ref{lemma:critical_point}}
	In this section we show that for any choice of $\tilde{c}\in\mathbb{R}$ and any level $c\le\tilde{c}$, and for a suitable choice of $\lambda$, we can find a direction decreasing the value of $\mathcal A$  and increasing the value of  $\Glambda$ along the set \begin{equation}
		\label{eq:def_domain_eigenvalue_eq}
		\{\mathcal{A} = c\}\cap\{\Glambda = b\}, \text{ for }b\text{ large enough. }
	\end{equation}
	As discussed before, this will indeed guarantee that assumptions $ii)$ and $v)$ of Lemma \ref{lemma:critical_point} holds for $\mathcal{A}$ and $\mathcal{G}_{\lambda,c}$, choosing $c=c^*$ and any $b>b(\tilde{c})$ where $b(\tilde{c})$ is the threshold determined in Proposition \ref{prop:no_solution_eigenvalue} below.  Recall that in this section we will work under hypothesis \eqref{assumption:values_f_g},\eqref{assumption:homogenuity}, \eqref{assumption:attraction_infinity_2} and \eqref{assumption:extra_assumptions_f}. We now prove the following
	\begin{prop}
		\label{prop:no_solution_eigenvalue}
		Let $\tilde{c}\in\mathbb{R}$. There exist $b(\tilde{c}), \lambda(\alpha, T)>0$ such that
		\begin{enumerate}[label=\roman*)]
			\item For any $c\le\tilde{c}$, $b>b(\tilde{c})$ and $\lambda\ge \lambda(\alpha, T)$, there is no solution of the equation
			\[
			\nabla \mathcal{A}(q) = \sigma \nabla\Glambda(q),\ \text{ with }  \sigma>0,\ q\in \{\mathcal{A} \le \tilde{c}\}\cap\{\Glambda = b\}.
			\]
			\item For any $c\le\tilde{c}$, $b>b(\tilde{c})$ and $\lambda \ge \lambda(\alpha, T)$ we have
			\[
			\nabla \Glambda(q)\ne0\ \text{for any}\ q\in\{\mathcal{A}\le \tilde{c}\}\cap \{\mathcal{G}_{\lambda,c}=b\}.
			\]
		\end{enumerate}
	\end{prop}
	The following Lemma is needed to prove Proposition \ref{prop:no_solution_eigenvalue}.

	\begin{lemma}
		\label{lemma:proof_proportionality_gradients}
		Let $\tilde{c}\in\mathbb{R}$. For any $c\le\tilde{c}$, the following assertions hold true.
		\begin{enumerate}[label=\roman*)]
			\item Any solution $q$ of \begin{equation}
				\label{eq:proportional_gradients_Glambda}
				\nabla \mathcal{A}(q) = \sigma \nabla\Glambda(q),\ \text{ with } \sigma>0 \text{ and }\ \mathcal{A}(q) \le \tilde c
			\end{equation}
			satisfies $\nabla \mathcal{A}(q) = \sigma' \nabla \mathcal{G}(q)$ for some $0<\sigma'\le 1/\lambda.$
			\item There exist $\sigma_0, R_0, C >0$, depending only on $T$, $\tilde{c}$ and $\alpha$ such that for any $R\ge R_0$ and any solution $q$ of 
			\begin{equation}
				\label{eq:proportional_gradients_G}
				\nabla \mathcal{A}(q) = \sigma\nabla \mathcal{G}(q),\ \text{ with } 0<\sigma<\sigma_0 \text{ and }\ \mathcal{A}(q) \le \tilde c
			\end{equation} 
			having $\sum_{i=1}^{n-1}\int_0^T \vert q_i-q_{i+1}\vert^2 =R$, there exists a $j \in \{1,\dots,n-1\}$ such that the following inequality holds:
			\[ 
			\frac{C R}{n-1}\le C \int_0^T \vert q_j-q_{j+1}\vert^2\le \min_{t \in[0,T]}\vert q_j(t)-q_{j+1}(t)\vert^2.
			\]
            Moreover, for any $1 \le i \le n$ and $t \in [0,T]$ we have
            \[
             \vert q_i(t) \vert  \le 2 n \min_{t \in[0,T]} \vert q_j(t)-q_{j+1}(t)\vert.
            \]
		\end{enumerate}
		\begin{proof}
			Let us prove $i)$. A straightforward computation shows that $\nabla \mathcal{A} = \sigma \nabla\Glambda$ implies 
			\[
			\nabla \mathcal{A} = \frac{\sigma}{1+\sigma \lambda}\nabla \mathcal{G}.
			\]
			Since $\sigma>0$, then $0<\sigma /(1+\sigma \lambda)< 1/\lambda$ and the claim is proved.
			
			For point $ii)$ we argue as follows. We can assume, without loss of generality, that $\tilde{c}\ge0$. Suppose that $q$ is a solution of  \eqref{eq:proportional_gradients_G}. Thanks to \eqref{assumption:homogenuity} and since $\mathcal{A}(q)\le\tilde{c}$, we have that:
			\[
			\begin{aligned}
				\langle\nabla \mathcal A(q) ,q\rangle &= \int_0^T\sum\limits_{i=1}^n |\dot{q}_i|^2+ f_i'(q_i)q_i -\sum_{j<i}g_{ij}'(q_i-q_j)(q_i-q_j)  \\        	  
				&\ge \int_0^T\sum\limits_{i=1}^n \vert \dot{q}_i\vert^2-\alpha f(q_i)+\alpha \sum_{j<i}g_{ij}(q_i-q_j)\\
				&\ge \frac{\alpha+2}{2}\int_0^T\sum\limits_{i=1}^n \vert \dot{q}_i\vert^2-\alpha \tilde{c}
			\end{aligned}
			\]
			Similarly we obtain that:
			\begin{equation*}
				\langle\nabla \mathcal{G}(q),q\rangle = \int_0^T\sum_{i=1}^{n-1} 2\vert q_i-q_{i+1} \vert^2-\frac{2}{ \vert q_i-q_{i+1} \vert^2} \le 2 \int_0^T\sum_{i=1}^{n-1} \vert q_i-q_{i+1} \vert^2. 
			\end{equation*}
			Thus, we have showed that:
			\begin{equation}\label{eq:inequality_alphac}
			\begin{aligned}
				0 &= \langle \nabla \mathcal{A}(q)-\sigma \nabla \mathcal{G}(q),q\rangle \\
				&\quad\ge \frac{\alpha+2}{2}\int_0^T\left(\sum_{i=1}^n \vert \dot{q}_i\vert^2-\frac{4\sigma}{\alpha+2} \sum_{i=1}^{n-1}\vert q_i-q_{i+1}\vert^2\right)-\alpha \tilde{c}
			\end{aligned}
			\end{equation}
			Moreover, since the following inequality holds for any $i=1,\ldots,n-1$
			\begin{equation*}
				\int_0^T\vert \dot q_i-\dot q_{i+1}\vert^2\le  2 \int_0^T  \vert\dot q_i\vert^2+\vert \dot q_{i+1}\vert^2,
			\end{equation*}
			an easy computation shows that 
			\[
				\frac{\alpha+2}{2}\int_0^T\sum\limits_{i=1}^n\vert\dot{q}_i\vert^2\ge\frac{\alpha+2}{8}\int_0^T\sum\limits_{i=1}^{n-1}\vert\dot{q}_i-\dot{q}_{i+1}\vert^2.
			\]
			This fact, combined with \eqref{eq:inequality_alphac}, gives the following inequality
			\begin{equation}
				\label{eq:inequality_qf}
				\frac{8\alpha \tilde{c}}{\alpha+2} \ge \int_0^T\sum_{i=1}^{n-1}\vert \dot{q}_i-\dot{q}_{i+1}\vert^2-\frac{16\sigma}{\alpha+2} \vert q_i-q_{i+1}\vert^2.
			\end{equation}
			To simplify notation, let us set $\bar{c} = 8\tilde{c}\alpha/(\alpha+2)$ and $\bar\sigma = 16 \sigma/(\alpha+2).$ Moreover let $j$ be such that 
			\[
			\int_0^T \vert q_{j}-q_{j+1} \vert^2 \ge \int_0^T \vert q_{i}-q_{i+1} \vert^2\quad \forall \, i=1,\ldots,n. 
			\]
			Let us define $v := q_{j+1}-q_j \ge0$. In particular we see that
			\begin{equation*}
				\int_0^T \sum_{i=1}^{n-1} \vert q_i-q_{i+1}\vert^2 \le (n-1)\int_0^T \vert v\vert^2, \quad \int_0^T \sum_{i=1}^{n-1} \vert \dot q_i-\dot q_{i+1}\vert^2 \ge \int_0^T \vert \dot v \vert^2.
			\end{equation*}
			We can thus rewrite inequality \eqref{eq:inequality_qf} for the function $v$ as follows:
			\begin{equation*}
				\bar{c} \ge \int_0^T\vert \dot{v}\vert^2-\bar \sigma(n-1)\int_0^T\vert v\vert^2.
			\end{equation*}
			By means of the fundamental theorem of calculus, we infer that, for any $t_1,t_2\in[0,T]$
			\begin{equation*}
				\vert v(t_1)-v(t_2)\vert^2 =\left\vert \int_{t_1}^{t_2} \dot{v}\,\, \right\vert^2\le T  \int_0^T \vert\dot{v}\vert^2 \le T\left(\bar c+\bar \sigma(n-1) \int_0^T\vert v\vert^2 \right).
			\end{equation*}
			Recall that $v$ is continuous. The mean value theorem guarantees the existence of $t^*\in[0,T]$ such that 
			\[
			\int_0^T\vert v\vert^2 = T\vert v(t^*)\vert^2\ge T \min_{t \in[0,T]}\vert v(t)\vert^2.
			\]
			Let us denote by $s^*$ a point in $[0,T]$ where the minimum in the right-hand side is achieved. Combining the previous fact with the last inequality we obtain that:
			\[
			\begin{aligned}
				\sqrt{T} \left(\bar c+\bar \sigma (n-1) \int_0^T\vert v\vert^2 \right)^{\frac12}&\ge\vert v(t^*)- v(s^*)\vert \\
				&= v(t^*)- v(s^*) \\ 
				&= \frac{1}{\sqrt{T}}\left(\int_0^T\vert v\vert^2 \right)^{\frac12}- \min_{t \in[0,T]}\vert v(t)\vert
			\end{aligned} 
			\]
			By standard computations, it follows that:
			\begin{align}            \label{eq:inequality_minimum_proof_lemma2}
            \nonumber
				\min_{t \in[0,T]}\vert v(t)\vert &\ge \frac{1}{\sqrt{T}}\left(\int_0^T\vert v\vert^2 \right)^{\frac12}- \sqrt{T} \left(\bar c+\bar \sigma (n-1) \int_0^T\vert v\vert^2 \right)^{\frac12} \\
				&= \frac{1}{\sqrt{T}}\left(1-T \sqrt{n-1}\sqrt{\bar \sigma  +\frac{\bar c}{(n-1)\Vert v \Vert_2^2}}\right)\Vert v \Vert_2 \nonumber\\
                &\ge  \frac{1}{\sqrt{T}}\left(1-T \sqrt{n-1}\sqrt{\bar \sigma  +\frac{\bar c}{R}}\right)\Vert v \Vert_2
			\end{align}
			
             since $\Vert v\Vert_2^2\ge R/(n-1)$.
            Let us write
            \[
                q_i = q_1+\sum_{k=1}^{i-1}q_{k+1}-q_k
             \]  
  which implies that 
  \[
    q_i\le  \vert q_i\vert \le \sqrt{n} \left(\sum_{k=1}^{n-1} (q_{k+1}-q_{k})^2+q_1^2\right)^{\frac{1}{2}}.
  \]
  Integrating over $[0,T]$ and applying H\"older inequality we obtain the following estimate for the mean value of $ q_i$.
  \[
     \frac1T\int_0^T q_i\le  \frac1T\int_0^T \vert q_i\vert \le\sqrt{\frac{n}{T}}\left(\sum_{k=1}^{n-1}\int_0^T (q_{k+1}-q_{k})^2+\int_0^Tq_1^2\right)^\frac12.
  \]
  
  Moreover we have a Poincaré inequality on $q_1$ yielding
  \[
  \int_0^Tq_1^2  = \int_0^Tt^2\left(\frac{1}{t} \int_0^t \dot{q}_1\right)^2\le \int_0^Tt\left(\int_0^t\dot{q}_1^2\right)\le \frac{T^2}{2} \int_0^T\dot{q}_1^2.
  \]
  On the other hand, we can apply the same argument involving the mean value theorem as above to show that $\exists t^*\in[0,T]$ such that for all $t \in [0,T]$
  \[
  \vert q_i(t)-q_i(t^*)\vert =
  \left\vert q_i(t)-\frac{1}{T}\int_0^Tq_i\right\vert \le \sqrt{T} \left(\int_0^T \dot{q}_i^2\right)^\frac{1}{2}
  \]
  which, combining the previous inequalities together with the fact that \[
  \sqrt{a+b}\le \sqrt{a}+\sqrt{b} \le \sqrt{2}\sqrt{a+b}\]
  and $n\ge2$ implies that
  \begin{align*}
   \vert q_i \vert &\le \sqrt{T} \left(\int_0^T \dot{q}_i^2\right)^\frac{1}{2}+\frac{1}{T}\int_0^T \vert q_i \vert \\&\le  \sqrt{\frac{n}{T}}\left(\sum_{k=1}^{n-1}\int_0^T (q_{k+1}-q_{k})^2+ \frac{T^2}{2} \int_0^T\dot{q}_1^2\right)^\frac12+\sqrt{T}\left(\int_0^T \dot{q}_i^2\right)^\frac{1}{2}\\
   &\le  \sqrt{\frac{n}{T}}\left(\sum_{k=1}^{n-1}\int_0^T (q_{k+1}-q_{k})^2\right)^{\frac12}+\sqrt{ \frac{nT}{2}} \left( \int_0^T\dot{q}_1^2\right)^\frac12+\sqrt{T}\left(\int_0^T \dot{q}_i^2\right)^\frac{1}{2}\\
   &\le  \sqrt{\frac{n}{T}}\left(\sum_{k=1}^{n-1}\int_0^T (q_{k+1}-q_{k})^2\right)^\frac12+\sqrt
   {n T}\left(\int_0^T\dot{q}_1^2+\int_0^T \dot{q}_i^2\right)^\frac{1}{2}\\
  \end{align*}
  Recall that $R = \sum_{k=1}^{n-1}\int_0^T (q_{k+1}-q_{k})^2$. Thanks to the homogeneity condition  in equation \eqref{eq:inequality_alphac} we obtain
  \begin{align*}
       \vert q_i \vert &\le  \sqrt{\frac{n}{T}}\sqrt{R}+\sqrt
  {n T}\left(\frac{2 \alpha \tilde c}{\alpha+2} + \frac{4\sigma}{\alpha +2} R\right)^\frac{1}{2}\\
  &=\sqrt{ nR}\left(\sqrt{\frac{1}{T}}+\sqrt
  {T}\left(\frac{2 \alpha \tilde c}{(\alpha+2)R} + \frac{4\sigma}{\alpha +2} \right)^\frac{1}{2}\right)\\
  & = \sqrt{ \frac{nR}{T}}\left(1+\frac{T}{2}\left(\frac{\bar c}{R} + \bar\sigma \right)^\frac{1}{2}\right).
  \end{align*}
 Combining this inequality with \eqref{eq:inequality_minimum_proof_lemma2} we obtain that
 \[
   \min_{t \in[0,T]} \vert v(t)\vert \ge \frac{1}{n}\left( \frac{1-T\sqrt{n-1}\sqrt{\bar\sigma+ \frac{\bar c}{R}}}{1+\frac{T}{2} \sqrt{\bar\sigma+ \frac{\bar c}{R}}}\right) q_i
 \]
 Let us observe that the function $h(x)= \frac{1}{n}\frac{1-x}{1+x/(2\sqrt{n-1})}$ is monotone decreasing, negative at infinity and has value $1/n$ for $x=0$. Thus, there exists $x_0 = x_0(n)>0$ such that for any $x \in [0,x_0]$, $h(x)\ge 1/2n$.
It follows that, for 
\[
\bar \sigma \le \frac{x_0^2}{2T^2(n-1)}, \quad R_0= \frac{2T^2\bar c(n-1)}{x_0^2}
\]
and for all $i\le n$, we have
\[
q_{j+1}-q_j \ge \frac{1}{2 n}  \vert q_i \vert , \quad  \min_{t \in [0,T]} q_{j+1}-q_j \ge C(T, \alpha, n) \sqrt{R},
\]
for some positive constant $C(T,\alpha,n)$ depending on $T,\alpha$ and $n$.

    \end{proof}
	\end{lemma}

	\begin{proof}[Proof of Proposition \ref{prop:no_solution_eigenvalue}]
		Without loss of generality we can assume that $b\gg R_0$ and choose  $\lambda$ in \eqref{eq:def_modified_G} so that $\lambda \ge \frac{1}{\sigma_0}$ (compare with the proof of Lemma \ref{lemma:proof_proportionality_gradients} for the value of these constants).
		
		First we prove $i)$ and we argue by contradiction. Let us assume that there exists indeed a solution $q$ of \eqref{eq:proportional_gradients_Glambda} satisfying $\mathcal{A}(q) \le \tilde{c}$ and $\Glambda(q)=b$. First, assume by contradiction that 
		\[
		\int_0^T\sum\limits_{i=1}^{n-1} \vert q_i-q_{i+1}\vert^2\ge R> R_0
		\] 
		for some $R$ (which will be uniquely determined along the proof). From $ii)$ in Lemma \ref{lemma:proof_proportionality_gradients}, there exists $j$ such that  $\min_{t \in[0,T]}\vert q_j(t)-q_{j+1}(t)\vert^2\ge \frac{C R}{(n-1)}$. 
		
		Let us define a vector $w_j\in\mathbb{R}^n$ such that
		\begin{equation}
			\label{eq:def_w_j}
			\begin{cases}
			(w_j)_i = 0 \text{ if } i\le j \\
			(w_j)_i = 1 \text{ if } i>j
			\end{cases}
		\end{equation}
		For $\ve>0$, consider the variation  $q+\ve w_j$. A straightforward computation implies that:
		\begin{equation*}
			\langle \nabla \mathcal{A}(q), w_j \rangle = \int_0^T \sum_{i=j+1}^{n} \left( f_i'(q_i)-\sum_{k=1}^jg_{ik}'(q_i-q_k) \right).
		\end{equation*}

        Thanks to point ii) of Lemma \ref{lemma:proof_proportionality_gradients} we know that for $k<j<i$ and $t \in [0,T]$
        \[
        \vert q_i (t) \vert \le 2 n \min_{t \in [0,T]} (q_{j+1}(t)-q_j(t)) \le 2 n (q_i(t)-q_k(t)).
        \]
        We have two possibilities. Either $q_i(t) \in (-\infty ,s_0]$ or $q_i(t)\in (s_0,\infty)$, here $s_0$ is the one appearing in \eqref{assumption:attraction_infinity_2}. In the first case we have
        \[
        f_i'(q_i)\le\max_{s \in (-\infty,s_0]}f_i'(s)<0
        \]
        thanks to \eqref{assumption:extra_assumptions_f}, whereas for any $k<j$
        \[
        -g_{ik}'(q_i-q_k)\le \max_{s\ge\sqrt{CR/(n-1)}} -g_{ik}'(s).
        \]
        Since $\lim_{s\to \infty} g_{ik}'(s) =0$, we can assume that $\bar{R}_0^{ik}$ is big enough so that whenever $R\ge \bar{R}_0^{ik}$ it holds
        \[ 
          \max_{s\ge\sqrt{CR_0^{ik}/(n-1)}} -g_{ik}'(s)<\min_{s \in (-\infty,s_0]}-f_i'(s), \text{ for all }k <i.
        \]
        In turn, this implies that
        \[
         f_i'(q_i)-\sum_{k=1}^jg_{ik}'(q_i-q_k)<0
         \]
         whenever $q_i\le s_0$. On the other hand, if $q_i>s_0$ we have 
         \[
           q_i-q_k\ge \frac{q_i}{2n} \text{ and } q_i> s_0
         \]
         and so \eqref{assumption:attraction_infinity_2} implies that 
          \[
         f_i'(q_i)-\sum_{k=1}^jg_{ik}'(q_i-q_k)<0
         \]
         is negative on $[0,T]$.
		
		Let us compute now $\langle \nabla \mathcal{G}(q), w_j\rangle$. A straightforward calculation shows:
		\[
		\langle \nabla \mathcal{G}(q), w_j \rangle = 2\int_0^T (q_{j+1}-q_j)-\frac{1}{(q_{j+1}-q_j)^3}.
		\]
		Again from Lemma \ref{lemma:proof_proportionality_gradients} we know that $q_{j+1}-q_j\ge\sqrt{C R}/\sqrt{n-1}$ and thus 
		\[
		-\frac{1}{(q_{j+1}-q_j)^3}\ge-\left(\frac{\sqrt{n-1}}{\sqrt{CR}}\right)^3.
		\]
		It follows that
		\[
		\langle \nabla \mathcal{G}(q), w_j \rangle \ge 2 T \frac{(CR)^2-(n-1)^2}{\sqrt{n-1}(CR)^{3/2}}>0
		\]
		provided that $R>(n-1)/C$. Thus, as long as $R> \bar{R}:=\max\{R_0,\frac{n-1}{C},{\bar R_0^{ik}}\}$, we have $ \sigma<0 $ and thus a contradiction arises.
		
		Therefore, we can assume that $\sum_{i=1}^{n-1}\int_0^T \vert q_i-q_{i+1}\vert^2<\bar R+1$; note that $\bar{R}$ does not depend on $b$. 
		Let us observe that in this case, \eqref{eq:inequality_alphac} gives a bound on the $L^2$ norm of the derivatives $\dot q_i$ which is $b-$independent. Indeed we easily have:
		\begin{equation*}
			C_1 = \frac{2\alpha\tilde{c}}{\alpha+2}+ \frac{4 \sigma{(\bar R+1)}}{\alpha+2} \ge \Vert \dot q\Vert_2^2.
		\end{equation*} 
		In turn, using the fundamental theorem of calculus, this implies that the first particle $q_1$ satisfies $q_1(t)\ge - \sqrt{TC_1}$ for all $t \in[0,T]$. 	
		Any solution $q$ of \eqref{eq:proportional_gradients_G}  solves the equation $\ddot{q}=\nabla U_\sigma(q)$, where $U_\sigma$ is the potential
		\[U_{\sigma}(q) = \sum_{i=1}^nf_i(q_i)-\sum_{1\le i<j\le n}g_{ij}(q_j-q_i)-\left(\sum_{i=1}^{n-1} \frac{\sigma}{(q_{i+1}-q_{i})^2}+\sigma (q_{i+1}-q_{i})^2\right),\]
        where $\sigma>0$ is the coefficient appearing in \eqref{eq:proportional_gradients_G}. 
		Hence, integrating the energy $h$ of  a solution of \eqref{eq:proportional_gradients_G} gives
		\begin{equation*}
			T h  = \frac12\Vert \dot{q}\Vert_2^2 - \int_0^T U_\sigma(q).
		\end{equation*}
		Moreover, thanks to assumption \eqref{assumption:homogenuity} we have the following inequality:
		\[
		\begin{aligned}
			0 &=\langle \nabla\mathcal{A}(q)-\sigma\nabla\mathcal{G}(q),q\rangle \\
			&=\int_0^T\sum\limits_{i=1}^n\vert\dot{q}_i\vert^2 + f_i'(q_i)q_i - \sum\limits_{j<i}g_{ij}'(q_i-q_j)(q_i-q_j) \\
			&\quad-2\sigma\int_0^T\sum\limits_{i=1}^{n-1}(q_i-q_{i+1})^2-\frac{1}{(q_i-q_{i+1})^2} \\
			&\ge\|\dot{q}\|_2^2 - \alpha\int_0^TU_\sigma(q) + \int_0^T\sum\limits_{i=1}^{n-1}\frac{\sigma(2-\alpha)}{(q_i-q_{i+1})^2}-(2+\alpha)\sigma(q_i-q_{i+1})^2 \\
			&\ge \Vert \dot{q}\Vert_2^2 -\alpha \int_0^T U_{\sigma}(q)-\sigma(2+\alpha) (\bar R+1)\\
			&=\frac{2-\alpha}{2} \Vert \dot{q}\Vert_2^2-\sigma(2+\alpha) (\bar R+1)+\alpha T h.
		\end{aligned}
		\]
		Therefore, the energy $h$ is bounded from above and in particular
		\[
			h\le \frac{\sigma(2+\alpha)(\bar{R}+1)}{\alpha T}+\frac{(\alpha-2)C_1}{2\alpha T}.
		\]
		We readily see that the contribution of the $f_i$ is bounded since
		\begin{equation}
		\label{eq:esitmate_going_wrong_limit}
		0\ge  -f_i(q_i)\ge-f_i(-\sqrt{TC_1}).
		\end{equation}
		However, since we are assuming that $\Glambda(q) = b$, we see that
		\[
			\bar{R}+1>\int_0^T\sum\limits_{i=1}^{n-1}(q_{i+1}-q_i)^2 = b-\int_0^T\sum\limits_{i=1}^{n-1}\frac{1}{(q_{i+1}-q_i)^2}
		\]
		and so, for at least one $j$, we have 
		\begin{equation*}
		\int_0^T \frac{1}{(q_{j+1}-q_j)^2}\ge \frac{b-(\bar R+1)}{n-1},
		\end{equation*}
		hence $q_{j+1}(t)-q_j(t)\le \sqrt{\frac{T(n-1)}{b-(\bar R+1)}}$ at some instant $t$. In particular, there is always an instant $t$ in which the value of $g_{j+1,j}(q_{j+1}-q_j)$ can be made arbitrarily large as $b$ grows. This follows directly from \eqref{assumption:homogenuity} which implies that $g_{ij}(s)\to +\infty$ as $s \to 0^+$.
		We finally obtain a contradiction observing that
		\begin{equation*}
			\frac{\sigma (\alpha+2)(\bar{R}+1)}{\alpha T}+\frac{C_1(\alpha-2)}{2\alpha T} \ge h \ge - \sum_{i=1}^nf_i\left(-\sqrt{TC_1}\right) + g_{j+1,j}\left (\sqrt{\frac{T(n-1)}{b-(\bar R+1)}}\right).
		\end{equation*}

		To prove $ii)$, let us observe that $\nabla \mathcal{G}_{\lambda,c}=0$ if and only if:
		\[
		\nabla\mathcal{A} = \frac{1}{\lambda} \nabla \mathcal{G} .
		\]
		Thanks to the previous point and Lemma \ref{lemma:proof_proportionality_gradients}, this is not possible whenever $\lambda \ge\frac{1}{\sigma_0}$.
	\end{proof}

\subsection{Palais-Smale conditions}
	In this section we prove two versions of Palais-Smale compactness relative to the sublevels of $\mathcal{G}_{\lambda,c}$.

	\begin{prop}
		The following assertions hold true:
		\begin{enumerate}[label = \roman*)]
			\item for any $c, b \in \mathbb{R}$, any sequence $(q^k)\subseteq \mathcal{H}$ having  
			\begin{equation}
				\label{eq:palais_condition_sublevels_g}
				\lim\limits_{k\to+\infty} \mathcal{A} (q^k) = c, \quad \limsup\limits_{k\to+\infty} \mathcal{G}_{\lambda,c}(q^k) \le b, \quad \lim\limits_{k\to+\infty}\nabla \mathcal{A}({q^k})=0
			\end{equation}
			admits a strongly convergent subsequence.
			\item For any $c,b \in \mathbb{R}$, any sequence $(q^k)\subseteq \mathcal{H}$ having  \begin{equation}
				\label{eq:palais_condition_constrained}
				\lim\limits_{k\to+\infty} \mathcal{A} (q^k) = c, \quad  \lim\limits_{k\to+\infty} \mathcal{G}_{\lambda,c}(q^k) = b, \quad \lim\limits_{k\to+\infty}\nabla \mathcal{A}({q^k})-\lambda_k \nabla \mathcal{G}_{\lambda,c}(q^k)=0
			\end{equation}
			for some sequence $(\lambda_k)$ with $\lambda_k \ge 0$, admits a strongly convergent subsequence.
		\end{enumerate}
		\begin{proof}
			We prove $i)$. Since $q_1^k(0)=0$ for any $k$, a Poincar\'e inequality holds for $q_1^k$, namely $\|q_1^k\|_2\le T\|\dot{q}_1^k\|_2/\sqrt{2}$. Moreover, by Cauchy-Schwartz and triangular inequality, we have that there exist $c_1,c_2 >0$ such that
			\[
				c_1\|q^k\|_2^2\le \|q_1^k\|_2^2 + \int_0^T\sum\limits_{i=1}^{n-1}(q_{i+1}^k-q_i^k)^2\le c_2\|q^k\|_2^2.
			\]
			By assumption, we also know that
			\[
				\limsup\limits_{k\to+\infty} \mathcal{G}_{\lambda,c}(q^k) = \int_0^T\sum\limits_{i=1}^{n-1}(q_{i+1}^k-q_i^k)^2 + \frac{1}{(q_{i+1}^k-q_i^k)^2} \le b
			\]
			and so we see that, for $k$ big enough and some constant $C>0$
			\[
				\|q^k\|_{H^1}^2 =\sum\limits_{i=1}^n(\|q_i^k\|_2^2+\|\dot{q}_i^k\|_2^2)\le (b+1)/c_1+C\|\dot{q}_1^k\|_2^2 + \sum\limits_{i=2}^n\|\dot{q}_i^k\|_2^2.
			\]
			Therefore, we only need to prove that $\|\dot{q}^k\|_2$ is bounded. By contradiction, let us assume that, up to sub-sequence, $\|\dot{q}^k\|_2 \to +\infty$. Thanks to assumption \eqref{assumption:homogenuity} (see also the proof of $ii)$, Lemma \ref{lemma:proof_proportionality_gradients}) we have that:
			\begin{equation*}                        
				0 = \lim_{k \to \infty} \frac{\langle\nabla \mathcal{A}(q^k),q^k\rangle}{\|\dot{q}^k\|_2} \ge \limsup_{k \to \infty} \frac{\alpha+2}{2}\| \dot{q}^k\|_2- \frac{c}{\| \dot{q}^k\|_2},
			\end{equation*}
			 which is a contradiction. Thus, $(\dot{q}^k)$ is bounded in $L^2$. 
			
			So $(q^k)$ is bounded in $H^1$ and admits a weakly convergent subsequence with its weak limit $q$. Moreover, again by assumption, note that $\int_0^T(q^k_i-q^k_{i+1})^{-2}<b$ for each $i$. This implies that
			\[
			\liminf_{k\to+\infty} \min_{ t \in[0, T],i}\vert q^k_{i+1}(t)-q^k_i(t) \vert>0.
			\]
			Testing $\nabla \mathcal A (q^k)$ against $q-q^k$ yields:
			\[
			\begin{aligned}
				\langle \nabla \mathcal{A}(q^k),q-q^k\rangle &= \langle \dot q-\dot q^k,\dot{q}^k\rangle +\int_0^T\sum_i f_i'(q^k_i)(q_i-q^k_i)\\
				&\quad-\sum_{j<i}g_{ij}'(q^k_i-q^k_j)\big(q_i-q_j-(q^k_i-q^k_j)\big). 
			\end{aligned}
			\]
			Let us observe that here $f_i'(q_i^k)$ and $g_{ij}'(q^k_i-q_j^k)$ are bounded functions. Indeed the $f_i'$ satisfy assumption \eqref{assumption:extra_assumptions_f} and the sequence $(q^k)$ is far from collision configurations. Since $q-q^k$ tends to zero uniformly, we conclude that $(q^k)$ converges strongly in $H^1$ since
			\[
			\lim\limits_{k\to+\infty} \langle \dot q-\dot q^k,\dot{q}^k\rangle = 0.
			\]
			
			The proof of $ii)$ follows the same lines. We recover again $H^1$ boundedness using assumption \eqref{assumption:homogenuity} and the bound on $\mathcal{G}_{\lambda,c}.$ Let us observe that
			\[
			\nabla\mathcal{A}({q^k})-\lambda_k \nabla \mathcal{G}_{\lambda,c}(q^k) =   (1+\lambda_k\lambda) \nabla \mathcal{A}({q^k}) -\lambda_k\nabla \mathcal{G}({q^k}).
			\]
			Thus, considering again the weak limit $q$ as in the previous point, we know by assumption that
			\[
				\langle\nabla\mathcal{A}(q^k),q-q^k\rangle = \frac{\lambda_k}{1+\lambda\lambda_k}\langle\nabla\mathcal{G}(q^k), q-q^k\rangle + o(1),
			\]
			but also
			\begin{align*}
				\langle \nabla \mathcal A(q^k) , q-q^k\rangle 
				= 	&\langle \dot {q}^k , \dot{q}^k-\dot{q}\rangle + \int_0^T\sum_i f_i'(q^k_i)(q_i-q^k_i)+ \\ &-\sum_{j<i}g_{ij}'(q^k_i-q^k_j)\big(q_i-q_j-(q^k_i-q^k_j)\big).
			\end{align*}
			Thus, we deduce strong convergence following the same arguments as before.
		\end{proof}
	\end{prop}

\subsection{Assumption $vi)$ of Lemma \ref{lemma:critical_point}}\label{sec:boundary_disk}
	The aim of this section is to show that not all sublevels of the functional $\mathcal{A}$ below some threshold $\tilde{c}\in\mathbb{R}$ can be deformed into each other relatively to the set $\{\mathcal{G}_{\lambda,\tilde{c}}\ge b\}$. Let us start with some preliminary definitions. 

	\begin{defn}
    For any two closed subset $A,B \subseteq \mathcal{H}$ we say that $A$ is deformable into $B$ if there exists a continuous map $h:[0,1]\times \mathcal{H} \to \mathcal{H}$ such that
		\begin{enumerate}[label = \roman*)]
			\item $h(0,a) = a, $ for all $a \in A$,
			\item $h(1,A)\subseteq B$.
            \item $h(t,b) = b$ for all $b \in B$ and $t \in [0,1]$.
		\end{enumerate}
	\end{defn}
	\noindent Let $\tilde c >c \in \mathbb{R}$ and consider the sets:
	\begin{equation}
		\label{eq:def_Acb}
		A_{c,b} := \{\mathcal{A} \le c\}\cup \left(\{\mathcal{A}\le \tilde c\}\cap\{\mathcal{G}_{\lambda,\tilde{c}}\ge b\}\right).
	\end{equation}
	Moreover, let us define the following value:
	\begin{equation}
		\label{eq:min_max_value}
		c^*= \inf_{c \le\tilde c } \{A_{\tilde c,b} \text{ is deformable into } A_{c,b}\}.
	\end{equation}
	The rest of the section is devoted to proving that $c^* >-\infty$. This indeed guarantees that assumption $vi)$ of Lemma \ref{lemma:critical_point} holds (see Proposition \ref{prop:non_deformability} below). The strategy of the proof is to construct an embedding of a $(n-2)-$dimensional sphere in the set $\{\mathcal{G}_{\lambda,\tilde{c}}\ge b\}$ and show that any capping with a disk must have a point on which  $\mathcal{G}_{\lambda,\tilde c}$ is bounded from above and  $\mathcal{A}$ from below by a constant depending only on $\tilde c, T$ and $n.$

	We start constructing the desired embedding. Let us consider the following minimization problem. Let us introduce the following functional
	\begin{equation}		\label{eq:minimization_problem_kepler}
		\mathcal{F}(y) = \int_0^T \frac12 \sum_{i=0}^{n-1}\vert \dot y_i \vert^2+f_{i+1}(y_i),
	\end{equation} 
	on the set \[
	\mathcal{C} := \{ y \in H^1([0,T],\mathbb R^n) :y_0(0)=0,\, y_{2k}(T) = y_{2k+1}(0), \, y_{2k+1}(T) = y_{2k}(0)\}
	\]
	for $0\le k\le\lfloor (n-1)/2 \rfloor$. Note that, any $n$ potentials $f_1,\dots,f_n$ satisfying assumptions \eqref{assumption:values_f_g}-\eqref{assumption:convexity} (but not necessarily \eqref{assumption:extra_assumptions_f}) can be extended to functions on $\mathbb{R}$ with values in $\mathbb{R}\cup \{+\infty\}$ and can be used to defined the functional in \eqref{eq:minimization_problem_kepler}. Indeed, setting, for $s<0$
	\[
	f_j(s) = \begin{cases}
		\lim_{s\to 0^+}f_j(s) \quad &\text{ if } \lim_{s\to 0^+}f_j(s)<\infty\\
		+\infty &\text{ otherwise}
	\end{cases}
	\]
	we obtain the extensions which are in particular monotone non-increasing.
	By standard arguments, $\mathcal{F}$ is coercive on $\mathcal{C}$ and admits a minimum  Let $x = (x_0,x_1,\dots,x_{n-1}) \in\mathcal{C}$ be a minimizer and let us set
	\begin{equation}\label{eq:c_0}
		c_0 = \min_{y \in \mathcal{C}} \mathcal{F}(y).
	\end{equation}

	\begin{lemma}
		\label{lemma:minimizers_monotone}
     Any minimizer $x$ of $\mathcal{F}$ in $\mathcal{C}$ satisfies the following 
     \begin{enumerate}[label=\roman*)]
     	\item $x_j(t)\ge0$ for all $t\in [0,T]$ and for all $j=1, \dots n$,
     	\item $x_i$ is monotone increasing if $i$ is even and decreasing if $i$ is odd.
     \end{enumerate} 
			\begin{proof}
				Let us observe that reversing time when $i$ is odd, i.e. considering 
				\[
				z_i(t) = \begin{cases}
					x_i(t) &\text{ if } i \text{ is even},\\
					x_i(T-t) &\text{ if } i \text{ is odd}					
				\end{cases}
				\]
				gives a minimizer of $\mathcal{F}$ on the set
				\[
				\mathcal{C}' := \{ y \in H^1([0,T],\mathbb R^n) :y_0(0)=0,\, y_{k}(T) = y_{k+1}(0), k = 0 \dots, n-1\}
				\]
				and viceversa. Clearly, if $z_j(0) = a$ then $z_j(t)\ge a$ for all $t\in [0,T]$ since otherwise, replacing $z_j$  and $z_i$ for $i>j$ with
				\[
				\tilde{z}_j(t) = \max\{a,z_j(t)\}, \quad \tilde{z}_i(t) = z_i(t) + \tilde{z}_j(T)-z_j(T) 
				\] 
				we can reduce the value of $\mathcal{F}$, since all the $f_j$ are decreasing and  point $i)$ follows. 
				
				Point $ii)$ is equivalent to showing that the $z_j$ are monotone increasing. Let us observe that the competitors 
				\[
				z_j(t,s) = \begin{cases}
					z_j(t) &\text{ if } t\le s\\
					\max\{z_j(s), z_j(t)\} &\text{ if }t\ge s
				\end{cases}, \quad z_i(t,s) = z_i(t) + z_j(T,s)-z_j(T) 
				\]
				imply that the $z_j$ are non decreasing. In particular, whenever they satisfy the Euler-Lagrange equation, which is either on  $[0,T]$ if $f_j(0^+)<\infty$ is finite or on a full measure open set otherwise, they are concave and strictly increasing. This in particular implies that they are always strictly positive except for $j=0$ and $t=0$.
			\end{proof}
	\end{lemma}
	
	Let us consider the following open set of $\mathbb{R}^{n-1}$ defined as:
	\[
		S=\{s=(s_1,\dots,s_{n-1}) \in \mathbb{R}^{n-1}:\,0<s_1<\dots<s_{n-1}\}
	\] 
	and let us define the following maps:
	\begin{align}
		\Theta &: S \to \mathcal 
		H, \quad \Theta(s_1,s_2,\dots,s_{n-1}) = (x_0,x_1+s_1,\dots,x_{n-1}+s_{n-1}), \nonumber
		\\
		\nonumber 
		\Phi_i &: \mathcal{H} \to \mathbb{R}, \quad \Phi_i(q) = \min_{t \in [0,T]} q_{i+1}(t)-q_i(t),
		\\
		\label{eq:def_Phi}
		\Phi &: \mathcal{H} \to \mathbb{R}^{n-1}, \quad 
		\Phi(q) = \left(\log \Phi_1(q), \ldots, \log\Phi_{n-1}(q) \right).
	\end{align}
	If $r>0$ we denote by $B_r$ the ball of radius $r$ centred at $0$ in $\mathbb{R}^{n-1}$ and $B'_r = (\Phi \circ \Theta)^{-1}(B_r)\subset S$.
	We have the following Lemma.

	\begin{lemma}
		\label{lemma:construction_boundary}
		The following statements hold true.
		\begin{enumerate}[label=\roman*)]
		\item The map $\Phi \circ \Theta : S \to \mathbb{R}^{n-1}$ is a diffeomorphism and so $B'_r$ is diffeomorphic to a $(n-1)$-dimensional disk. 
		\item 	The map $\Phi \circ \Theta: \partial B'_r \to \partial B_r$ has degree $1$. 
		\item  We have that $ \sup_{s \in S} \mathcal{A}(\Theta(s))\le c_0$, where $c_0$ is defined in \eqref{eq:c_0}. Moreover
		\[
		\lim_{r\to +\infty} \min_{s \in \partial B'_r} \mathcal{G}(\Theta(s)) =\lim_{r\to +\infty} \min_{s \in \partial B'_r} \mathcal{G}_{\lambda,\tilde c}(\Theta(s))=+\infty.
		\]
		\end{enumerate}
		\noindent In other words, for any $\tilde c$, $b>0$ and $\lambda>0$, there exists $r_0>0$ such that, for any $r>r_0$,  $S_r:=\Theta(\partial B'_r)$ is topologically a $(n-2)-$dimensional sphere and is contained in the set $\{\mathcal{G}_{\lambda,\tilde c}>b\}$. 
		\begin{proof}
		To prove point $i)$ let us observe that $S$ is a cone and the change of coordinates
		\[
		v_1 = s_1, \quad v_i=s_i-s_{i-1},    
		\] 
		maps it to the standard one defined by $V = \{(v_1,\dots,v_{n-1}): v_i>0\}.$ 
		
		Using these coordinates we immediately see that:
		\begin{equation*}
		\Phi_1(\Theta(s)) = s_1+\min_{t\in[0,T]}x_1-x_0 = v_1.
		\end{equation*}
        Similarly, if $i>2$ and $q=\Theta(s)$, then  $q_{i+1}-q_i =  v_{i}+ x_{i}-x_{i-1}$ and thus $\Phi_i(\Theta(s)) = v_i$.
		Indeed, the minimum of the function $x_{i+1}(t)-x_i(t)$ is always zero and is attained  at $t=0$ if $i$ is odd and at $t=T$ if $i$ is even. This is a consequence of the  fact that the components $x_j$ of a minimizers $x$ are monotone increasing/decreasing depending on the parity of $i$, in particular they are folded $nT$-brake orbits and solve $\ddot{x}_i= f_{i+1}'(x_i)$ for any $i=0,\ldots, n-1$ (compare with Lemma \ref{lemma:minimizers_monotone} and Lemma \ref{lemma:action_level}).  Thus, the map $\Phi\circ \Theta$ is a diffeomorphism.
			
		Point $ii)$ follows immediately from point $i)$ since $\Phi\circ\Theta\vert _{B'}$ is an orientation preserving diffeomorphism.

		Let us prove $iii)$. Composing $\mathcal{A}$ and $\Theta$, setting $s_0 = 0$ and recalling the functional $\mathcal{F}$ introduced in \eqref{eq:minimization_problem_kepler}, we obtain:
		\begin{align*}
			\mathcal{A}(\Theta(s)) &= \int_0^T  \sum_{i=0}^{n-1} \frac12 \vert \dot{x}_i\vert^2 + f_{i+1}(x_i+s_i)-\sum_{ 0\le i<j\le n-1}g_{i+1,j+1}(s_{j}-s_{i}+ x_{j}-x_{i})\\
			&\le \int_0^T\sum_{i=0}^{n-1} \frac12 \vert \dot{x}_i\vert^2+f_{i+1}(s_i+x_i)< \mathcal{F}(x)=c_0,
		\end{align*}
		and thus $\mathcal{A}$ is bounded on the image of $\Theta.$
			
		For the second part, recall that $(\Phi\circ\Theta)(B_r') = B_r$. Therefore, if a point $s$ belongs to $\partial B'_r$, then there exists at least one index $j$ such that $\vert\log(s_{j+1}-s_j)\vert \ge \frac{r}{\sqrt{n-1}}$, where we set $s_0 =0$ (recall the definition of $\Phi$ above). Thus either 
		\[
			s_{j+1}-s_j\ge  e^{\frac{r}{\sqrt{n-1}}} \text{ or }  s_{j+1}-s_j\le  e^{-\frac{r}{\sqrt{n-1}}}.
		\]
		Note also that
		\[
		\begin{aligned}
			\mathcal{G}(\Theta(s)) &=  \int_0^T\sum\limits_{i=0}^{n-2}\left[(x_{i+1}+s_{i+1}-x_i-s_i)^2 + \frac{1}{(x_{i+1}+s_{i+1}-x_i-s_i)^2}\right] \\
			&\ge \int_0^T(s_{j+1}-s_j)^2+\frac{1}{(s_{j+1}-s_j+x_{j+1}-x_j)^2},
		\end{aligned}
		\]
        If $s_{j+1}-s_j\ge e^{\frac{r}{\sqrt{n-1}}}$, we easily see that $\mathcal{G}(\Theta(s))\ge   T e^{2\frac{r}{\sqrt{n-1}}}$. In the other case we observe that
        \begin{align*}
          \mathcal{G}(\Theta(s)) &\ge  \int_0^T\frac{1}{(s_{j+1}-s_j+x_{j+1}-x_j)^2} \\
          &\ge\frac{1}{\Vert\dot{x}_{j+1}-\dot{x}_j \Vert_2^2} \left(\int_0^T\frac{\dot{x}_{j+1}-\dot{x}_j}{x_{j+1}-x_j+s_{j+1}-s_j}\right)^2\\
          &=\frac{1}{\Vert\dot{x}_{j+1}-\dot{x}_j \Vert_2^2} \left( \log( x_{j+1}-x_j+s_{j+1}-s_j)\vert_0^{T}\right)^2.         
        \end{align*}
		Since $x_{j+1}-x_j$ vanishes at $0$ or $T$, we obtain that $\mathcal{G}(\Theta(s))\ge (C_1 r-C_2)^2$ for some positive constants independent of $r$.
		\end{proof}
	\end{lemma}

	Recalling the notations of the previous Lemma, let us assume that $r>r_0$, consider the sphere $S_r$ and let $B_1$ be the $(n-1)-$dimensional ball of radius 1. We introduce the following set:
    \begin{equation}\label{def:gamma_zero}
	\begin{aligned}
		\Gamma_r =\{ \gamma\colon B_1\to \mathcal{H}\,\vert\,\gamma\ \text{is continuous and } \gamma\vert_{\partial B_1} = \gamma_0\},\\
		\text{ where }\gamma_0 : \partial B_1 \to \mathcal{H}, \quad \gamma_0(x) = \Theta \circ  \left(\Phi\circ \Theta \right)^{-1}(r x).
	\end{aligned} 
    \end{equation}

        \begin{figure}[t]
    \centering
	\resizebox{.65\textwidth}{!}{

\tikzset{every picture/.style={line width=0.75pt}} %set default line width to 0.75pt        

\begin{tikzpicture}[x=0.75pt,y=0.75pt,yscale=-1,xscale=1]

%Curve Lines [id:da5608374919477868] 
\draw [color={rgb, 255:red, 198; green, 192; blue, 192 }  ,draw opacity=1 ][line width=1.5]  [dash pattern={on 1.69pt off 2.76pt}]  (107.33,132.17) .. controls (186.33,129.15) and (241.33,133.88) .. (267.33,139.17) ;
%Curve Lines [id:da9813714520469629] 
\draw [color={rgb, 255:red, 47; green, 118; blue, 200 }  ,draw opacity=1 ]   (47.33,25.33) .. controls (88.33,42.33) and (101.33,66.33) .. (106.33,105.33) .. controls (111.33,144.33) and (106.33,180.42) .. (39.33,255.33) ;
%Curve Lines [id:da43874085181104083] 
\draw [color={rgb, 255:red, 47; green, 118; blue, 200 }  ,draw opacity=1 ]   (329.33,22.33) .. controls (276.33,49.33) and (260.33,62.25) .. (261.33,108.33) .. controls (262.33,154.42) and (306.33,203.42) .. (334.33,257.33) ;
%Curve Lines [id:da5671579960535046] 
\draw    (22.33,45.9) .. controls (74.67,42.8) and (153.33,64.9) .. (168.33,124.07) .. controls (183.33,183.23) and (107.33,231.68) .. (17.33,236.9) ;
%Curve Lines [id:da958024272756024] 
\draw    (352.33,234.23) .. controls (290.33,233.23) and (166.33,225.95) .. (170.33,136.95) .. controls (174.33,47.95) and (315.33,38.95) .. (347.33,45.95) ;
%Curve Lines [id:da9840707330084594] 
\draw [color={rgb, 255:red, 185; green, 181; blue, 181 }  ,draw opacity=1 ]   (149,91) .. controls (140.33,113.95) and (145.66,173.18) .. (157.33,177.95) ;
%Curve Lines [id:da9407252066271319] 
\draw [color={rgb, 255:red, 155; green, 155; blue, 155 }  ,draw opacity=1 ]   (160,103) .. controls (151.33,125.95) and (152.66,159.18) .. (164.33,163.95) ;
%Curve Lines [id:da23861363072250474] 
\draw [color={rgb, 255:red, 155; green, 155; blue, 155 }  ,draw opacity=1 ]   (168.33,124.07) .. controls (161.33,134.95) and (163.33,146.95) .. (167.33,154.95) ;
%Curve Lines [id:da1253936206325693] 
\draw [color={rgb, 255:red, 155; green, 155; blue, 155 }  ,draw opacity=1 ]   (178.33,106.07) .. controls (171.33,116.95) and (177.33,169.95) .. (181.33,177.95) ;
%Curve Lines [id:da8306368336353718] 
\draw [color={rgb, 255:red, 200; green, 198; blue, 198 }  ,draw opacity=1 ]   (187.33,93.07) .. controls (180.33,103.95) and (186.33,175.95) .. (196.33,195.95) ;
%Curve Lines [id:da3159808685762997] 
\draw [color={rgb, 255:red, 214; green, 210; blue, 210 }  ,draw opacity=1 ]   (131,75) .. controls (122.33,97.95) and (128.66,185.18) .. (140.33,189.95) ;
%Curve Lines [id:da0663506485227271] 
\draw [color={rgb, 255:red, 217; green, 213; blue, 213 }  ,draw opacity=1 ]   (203.33,77.07) .. controls (196.33,87.95) and (202.33,191.95) .. (219.33,212.95) ;
%Curve Lines [id:da6570147164938598] 
\draw [color={rgb, 255:red, 47; green, 118; blue, 200 }  ,draw opacity=1 ]   (39.33,255.33) .. controls (9.33,289.42) and (294.33,287.33) .. (334.33,257.33) ;
%Curve Lines [id:da5232938813142937] 
\draw [color={rgb, 255:red, 47; green, 118; blue, 200 }  ,draw opacity=1 ]   (47.33,25.33) .. controls (87.33,52.42) and (289.33,52.33) .. (329.33,22.33) ;
%Curve Lines [id:da40552333466801815] 
\draw [color={rgb, 255:red, 47; green, 118; blue, 200 }  ,draw opacity=1 ]   (47.33,25.33) .. controls (35.33,18.55) and (76.33,20.55) .. (80.33,21.55) ;
%Curve Lines [id:da11431145639838358] 
\draw    (65.33,10.25) .. controls (86.33,27.25) and (136.33,63.57) .. (148.33,76.57) .. controls (160.33,89.57) and (165.21,98.94) .. (178.33,88.57) .. controls (191.46,78.19) and (249.33,34.57) .. (285.33,6.57) ;
%Curve Lines [id:da1850032013890125] 
\draw [color={rgb, 255:red, 198; green, 196; blue, 196 }  ,draw opacity=1 ]   (150.33,76.57) .. controls (162.33,84.18) and (178.33,85.18) .. (192.33,77.18) ;
%Curve Lines [id:da5894145541753566] 
\draw [color={rgb, 255:red, 155; green, 155; blue, 155 }  ,draw opacity=1 ]   (158.33,87.57) .. controls (168.33,89.42) and (166.33,90.42) .. (178.33,88.57) ;
%Curve Lines [id:da04975863779536693] 
\draw [color={rgb, 255:red, 231; green, 227; blue, 227 }  ,draw opacity=1 ]   (134.33,62.47) .. controls (146.33,70.08) and (173.33,79.47) .. (213.33,62.47) ;
%Curve Lines [id:da9477849495549567] 
\draw    (321.33,275.22) .. controls (284.33,271.22) and (200.56,213.75) .. (187.19,202.17) .. controls (173.81,190.58) and (167.92,181.81) .. (156.04,193.58) .. controls (144.15,205.35) and (108.33,242.57) .. (33.33,271.57) ;
%Curve Lines [id:da8870415822814599] 
\draw [color={rgb, 255:red, 198; green, 196; blue, 196 }  ,draw opacity=1 ]   (185.2,202.39) .. controls (172.42,196.16) and (156.41,196.94) .. (143.39,206.45) ;
%Curve Lines [id:da21724762309647094] 
\draw [color={rgb, 255:red, 155; green, 155; blue, 155 }  ,draw opacity=1 ]   (176.02,192.35) .. controls (165.88,191.62) and (167.76,190.41) .. (156.04,193.58) ;
%Curve Lines [id:da9256201053462426] 
\draw [color={rgb, 255:red, 231; green, 227; blue, 227 }  ,draw opacity=1 ]   (202.67,214.62) .. controls (189.89,208.39) and (162.02,202.07) .. (124.16,223.42) ;
%Curve Lines [id:da48224730794772297] 
\draw [color={rgb, 255:red, 47; green, 118; blue, 200 }  ,draw opacity=1 ]   (279.33,11.55) .. controls (304.33,11.55) and (335.33,10.55) .. (329.33,22.33) ;
%Curve Lines [id:da5705565240028455] 
\draw [color={rgb, 255:red, 47; green, 118; blue, 200 }  ,draw opacity=1 ] [dash pattern={on 0.84pt off 2.51pt}]  (80.33,21.55) .. controls (121.33,20.55) and (251.33,12.55) .. (279.33,11.55) ;
%Curve Lines [id:da6853387681937403] 
\draw [line width=1.5]    (107.33,132.17) .. controls (15.33,145.17) and (55.79,162.95) .. (118.33,175.07) .. controls (180.87,187.18) and (296.33,177.97) .. (307.33,172.07) .. controls (318.33,166.17) and (323.33,152.9) .. (267.33,139.17) ;
\draw [shift={(62.66,159.37)}, rotate = 210.83] [fill={rgb, 255:red, 0; green, 0; blue, 0 }  ][line width=0.08]  [draw opacity=0] (13.4,-6.43) -- (0,0) -- (13.4,6.44) -- (8.9,0) -- cycle    ;
\draw [shift={(220.57,180.71)}, rotate = 178.67] [fill={rgb, 255:red, 0; green, 0; blue, 0 }  ][line width=0.08]  [draw opacity=0] (13.4,-6.43) -- (0,0) -- (13.4,6.44) -- (8.9,0) -- cycle    ;
\draw [shift={(292.5,146.7)}, rotate = 22.09] [fill={rgb, 255:red, 0; green, 0; blue, 0 }  ][line width=0.08]  [draw opacity=0] (13.4,-6.43) -- (0,0) -- (13.4,6.44) -- (8.9,0) -- cycle    ;

% Text Node
\draw (219.72,233.71) node [anchor=north west][inner sep=0.75pt]  [rotate=-24.01]  {$\mathcal{A}  >\ c^{*}$};
% Text Node
\draw (188,20.4) node [anchor=north west][inner sep=0.75pt]    {$\mathcal{A}  >\ c^{*}$};
% Text Node
\draw (22.33,51.3) node [anchor=north west][inner sep=0.75pt]    {$\mathcal{A} =\ c^{*}$};
% Text Node
\draw (244.76,200.16) node [anchor=north west][inner sep=0.75pt]  [rotate=-12.52]  {$\mathcal{A} =\ c^{*}$};
% Text Node
\draw (111.33,258.53) node [anchor=north west][inner sep=0.75pt]  [color={rgb, 255:red, 74; green, 144; blue, 226 }  ,opacity=1 ]  {$\mathcal{G} =\ b$};
% Text Node
\draw (67,111.4) node [anchor=north west][inner sep=0.75pt]    {$\gamma _{0}$};

\end{tikzpicture}
}
	\caption{An illustration of the proof of assumption $vi)$ of Lemma \ref{lemma:critical_point}. Any capping of the $(n-2)-$dimensional sphere $\gamma_0$ defined in \eqref{def:gamma_zero} must intersect the region $\{\mathcal{G}\le b\}$ in a point in which the action is bounded from below.}
    \label{fig:tube}
    \end{figure}
    
	Let us observe that there is always an element in $\Gamma_r$ whose image is contained in $\{\mathcal{A} \le \tilde{c}\}$, as soon as $\tilde{c}\ge c_0$ given in point $iii)$ of Lemma \ref{lemma:construction_boundary}. This element is simply given by the map
	\[
	\gamma: B_1 \to \mathcal{H}, \quad \gamma(x) = \Theta \circ (\Phi \circ \Theta)^{-1}(r x).
	\]
	\begin{lemma}
    \label{lemma:bound_G_A_internal_point}
		Let $r_0>0$ be the one defined in Lemma \ref{lemma:construction_boundary}. There exist constants $C_1,C_2$ depending only on $T,$ $n$, $\lambda>0$ and $\tilde{c} \ge c_0$ such that, for any $r>r_0$ and any $\gamma \in \Gamma_r$  whose image is contained in $\{\mathcal{A}\le \tilde c\}$, there exists $x \in B_1$ satisfying:
		\begin{equation*}
			\mathcal{G}_{\lambda, \tilde c} (\gamma(x))\le C_1, \quad \mathcal{A}(\gamma(x)) \ge C_2.
		\end{equation*}
		\begin{proof}
		The function $\Phi \circ\gamma_0$ has degree 1 being a homothety. It follows that, for any $\gamma\in\Gamma_r$ there exists at least $x\in B_1$ such that $\Phi(\gamma(x)) =0.$ Setting $q=\gamma(x)$, this implies that $\log(\Phi_i(q))=0$ for any $i$ and thus:
		\begin{equation}
			\label{eq:minimum_distances}
			\min_{t\in[0,T]} q_{i+1}(t)-q_i(t) =1, \quad   \forall i=1,\dots, n-1.
		\end{equation}
		So, $q_{k}-q_i\ge1$ for all $k>i$ and thus we see that
		\begin{equation}
		    \tilde c \ge \mathcal{A} (q) \ge \frac12 \Vert \dot{q}\Vert_2^2-T \sum\limits_{i=1}^{n-1}\sum_{j=i+1}^n g_{ij}(1) \ge -T \sum\limits_{i=1}^{n-1}\sum_{j=i+1}^n g_{ij}(1).
		\end{equation} 
	    Thus $\Vert \dot{q}\Vert_2$ and $\mathcal{A}(q)$ are bounded. Let us show that $\mathcal{G}$ is bounded as well. Indeed, let us set  $v_i = q_{i+1}-q_i$ and pick $t_i$ such that $v_i(t_i)=1$.
        Since $\|\dot{q}\|_2$ is bounded,  we have that $\Vert \dot v_i\Vert_2 \le C$ for a constant $C$ depending only on $\tilde c$ and $T$, thus \[
        \int_0^T v_i^2(t) dt  = \int_0^T \left(v(t_i)+\int_{t_i}^T \dot v_i\right)^2 dt\le T(1+ \sqrt{T}C)^2
        \]
        by H\"older inequality. On the other hand we immediately see that 
	    \[
	        \int_0^T \frac{1}{v_i^2}\le T.
	    \]     
        Combining the previous observations with the definition \eqref{eq:def_G} we obtain that $\mathcal{G}$ is bounded at $q$ by a constant depending only on $\tilde 
        c, n$ and $T$.
    
       	Finally, let us recall that $\mathcal{G}_{\lambda, \tilde c}$, as defined in \eqref{eq:def_modified_G}, differs from $\mathcal{G}$ by a factor $\lambda(\tilde c-\mathcal{A})$.  It follows that:
    	\[
		\mathcal{G}_{\lambda,\tilde c}(q) \le \mathcal{G}(q)+\lambda \left(\tilde c+T \sum_{i<j} g_{ij}(1)\right). 
		\]
	\end{proof}
	\end{lemma}
	As announced at the beginning of this section, the following proposition finally provides the non-deformability property $vi)$ of Lemma \ref{lemma:critical_point}.
	\begin{prop}[Non-deformability]\label{prop:non_deformability}
		For any $\tilde c\ge c_0$, there exists $b(\tilde{c})$ such that for any $b>b(\tilde{c})$ the corresponding value $c^*$ defined in \eqref{eq:def_Acb}-\eqref{eq:min_max_value} satisfies $c^*>-\infty$. In other words, $A_{\tilde{c},b}$ and any $A_{c,b}$ with $c \in (c^*,\tilde c]$ cannot be deformed into $A_{c^*-\ve,b}$ for any $\ve>0$.
		\begin{proof}
		We argue by contradiction. Let us assume that there exists $\tilde c \ge c_0$ such that, for arbitrarily large values of $b$, the value $c^*$ is equal to $-\infty$. In other words, let us assume that $A_{\tilde c,b}$ can be deformed in $A_{c,b}$, for any $c<\tilde c$. 
		
		Let $r_0>0$ be the one defined in Lemma \ref{lemma:construction_boundary}. Since $\tilde c \ge c_0$ the set   $\Gamma_r$ is non-empty for any $r>r_0$ and there exist at least a map $\gamma$ whose image is completely contained in $\{\mathcal{A}\le \tilde c\}$.        
		Let us consider a sequence $c_m \to -\infty$ and deformations $h_m$ sending $ A_{\tilde c,b}$ in $A_{c_m,b}$.        
		Pick any $\gamma' \in \Gamma_r$ with $\gamma'(B_1)\subset A_{\tilde c, b}$ and define $\gamma_m(x) = h_m (1, \gamma'(x))$. It follows that $\gamma_m(B_1)\subseteq A_{c_m,b}$ and $\gamma_m\vert_{\partial B_1}=\gamma_0$, so that $\gamma_m\in \Gamma_r$. However, thanks to Lemma \ref{lemma:bound_G_A_internal_point}, there is always a point $x_m$ having \[
		\mathcal{G}_{\lambda,\tilde c}(\gamma_m(x_m))\le C_1\text{ and } \mathcal{A}(\gamma_m(x_m))\ge C_2.
		\]
		Since $\tilde{c}$ is fixed and $C_1, C_2$ only depend on $\tilde{c}, T$ and $n$, whilst $b$ can be chosen arbitrarily large, a contradiction arises. 
		\end{proof}
	\end{prop}
	As a consequence, whenever $f_i$ and $g_{ij}$ satisfy conditions \eqref{assumption:values_f_g},\eqref{assumption:homogenuity},\eqref{assumption:attraction_infinity_2} and \eqref{assumption:extra_assumptions_f}, the assumptions of Lemma \ref{lemma:critical_point} are fulfilled. Notice that 
	\[
	\mathcal{G}_{\lambda,\tilde{c}}(q) = \mathcal{G}_{\lambda, c^*}(q) + \lambda(\tilde{c}-c^*).
	\] 
	Therefore, we conclude that there exists at
	least a critical point of $\mathcal{A}$ in the set $\{\mathcal{A}=c^*\}\cap\{\mathcal{G}_{\lambda, c^*}< b\}$ with $c^*\le c_0$.

\section{Proof of Theorem \ref{thm:solution_helium}}\label{sec:proof_thm1}
	In this section we prove Theorem \ref{thm:solution_helium}. First, we build a family of approximating functionals satisfying the assumptions of Theorem \ref{thm:general_existence_result}. Then, we show that their critical points converge to a collisionless solution of \eqref{eq:helium_equation}.
\subsection{Existence of $\ve-$solutions}
	We start with building an approximating family of functionals $\mathcal{A}_\ve$, in which we \emph{smooth out} the singularity of  $f_j$ at the origin  replacing $f_j$ with a $C^{1,1}(\mathbb R)$ function $f_j^\ve$ . Then, we modify its behaviour at infinity, toobtainher with the one of the $g_{ij}$. We consider a family of $\tilde{f}^\ve_j$ and $\tilde{g}_{ij}$ satisfying assumptions \eqref{assumption:values_f_g},\eqref{assumption:homogenuity},\eqref{assumption:attraction_infinity_2} and \eqref{assumption:extra_assumptions_f}, whose derivatives agree in a neighbourhood of the origin  with the ones of the original functions, namely:
	\begin{align*}
	(\tilde{f}^\ve_j)'(s)&= (f_j^\ve)'(s) \text{ for all }   s<\tilde{s},\\
    \tilde g_{ij}'(s)&= g_{ij}'(s) \text{ for all } 0<s<\tilde{s},
	\end{align*}
    and such that the differences with the $f_j^\ve$ and $g_{ij}$ are controlled arbitrarily well.
	An explicit construction of such a family is given in the next section.
    The modification at infinity is essentially irrelevant in the forthcoming discussion and it is only needed to apply Theorem \ref{thm:general_existence_result}. Indeed, the solutions of \eqref{eq:helium_newton} we will consider will lie in a compact set where the derivatives of these functions coincide, thanks to some a priori estimates in the spirit of Lemma \ref{lemma:uniform_bounds} below.
    We are thus brought to consider the following functional on $\mathcal{H}$
	\begin{equation*}
		\mathcal{A}_\ve(q) =  \sum_{i=1}^n \int_0^T \frac{1}{2} \vert \dot{q}\vert^2+f_i^\ve(q_i)-\sum_{i<j}g_{ij}(q_j-q_i).
	\end{equation*}
	A straightforward application of Theorem \ref{thm:general_existence_result}, given in Section \ref{sec:smoothing}, yields the following result.
	\begin{thm}
		\label{thm:ve_critical_points}
		Let us assume that \eqref{assumption:values_f_g}-\eqref{assumption:attraction_infinity} hold and that the $f_j$ are convex.
		For any $\ve>0$ there exists a critical point $q^\ve$ of $\mathcal{A}_\ve$ in $\mathcal{H} \cap \{\mathcal{A}_\ve\le c^\ve\}$ where $c^\ve$ is the value defined in \eqref{eq:c_0} and \eqref{eq:minimization_problem_kepler} choosing $f_j^\ve$ as potentials. 
	\end{thm}
    Clearly, for every $q^\ve$, there exists $b_\ve>0$ such that  $q^\ve \in \{\mathcal{G}\le b_\ve\}$.
	Unfortunately, it is possible that $b_\ve \to +\infty$ as $\ve\to 0$, since the original $f_j$ could be unbounded (compare with \eqref{eq:esitmate_going_wrong_limit}). Thus, we know that each critical point $q^\ve$ has no collisions, but, a priori, the same property is not guaranteed for the accumulation points.  The rest of the section is devoted to proving that the sequence $(q^\ve)$ converges, up to subsequence, to a collisionless solution of \eqref{eq:helium_equation}.

\subsection{Construction of $f_j^\ve$, $\tilde{f}_j^\ve$ and $\tilde{g}_{ij}$}
\label{sec:smoothing}

    The first step is to modify any $n-$ple of $f_j$ outside a small neighbourhood of the origin. 
    \begin{lemma}
    	\label{lemma:smoothing_f}
    	For any $f_j$ satisfying assumption \eqref{assumption:values_f_g}-\eqref{assumption:attraction_infinity} and any $\ve>0$ small enough there exists a function $f_j^\ve: \mathbb{R}\to (0,+\infty)$ satisfying \eqref{assumption:values_f_g}-\eqref{assumption:attraction_infinity} and \eqref{assumption:extra_assumptions_f} which coincide with $f_j$ outside an $\ve-$neighbourhood  of the origin. Moreover, if $f_j$ is convex so is $f_j^\ve$.
    	\begin{proof}
    		Let us consider the line tangent to $f_j$ at $\ve$. It reads:
    		\begin{equation*}
    			p_{j,\ve}(s) = f_j(\ve)+ f'_j(\ve)(s-\ve).
    		\end{equation*}
    		Let us define:
    		\begin{equation*}
    			f_{j}^\ve(s) = \begin{cases}
    				p_{j,\ve}(s) \text{ if } s \le \ve,\\
    				f_j(s) \text{ otherwise }
    			\end{cases}
    		\end{equation*}
    		All assumptions but \eqref{assumption:homogenuity} are clearly satisfied. Let us check the homogeneity condition for $s\le \ve$. Since
    		\[
    		\frac{d}{d s}\left[ s (f_j^\ve(s))'+\alpha f_j^\ve(s) \right]= (1+\alpha)  f_j'(\ve) <0
    		\] 
    		and, since the condition holds for $s=\ve$, we are done.  
    	\end{proof}
    \end{lemma}

     \begin{lemma}
     	\label{lemma:change_infinity_fg}
    	For any choice of functions $f^\ve_j$ and $g_{ij}$ satisfying assumptions \eqref{assumption:values_f_g}-\eqref{assumption:attraction_infinity} and $\eqref{assumption:extra_assumptions_f}$ there exists $k_0$,$\delta_0$ and functions $\tilde{f}_j^{\ve,\delta}$, $\tilde{g}_{ij}^k$ such that for any  $\delta<\delta_0$ and $k>k_0$ $\tilde{f}_j^{\ve,\delta}$, $\tilde{g}_{ij}^k$ satisfy simultaneously  \eqref{assumption:attraction_infinity} with the same choice of $s_0$ and $t_0$ as for  $f_{j}$, $g_{ij}$ and \eqref{assumption:attraction_infinity_2} (with possibly a non-uniform choice of $s_0$) as well as  \eqref{assumption:values_f_g},\eqref{assumption:homogenuity} and \eqref{assumption:extra_assumptions_f}.  
    	
    	Moreover, $(\tilde{f}_j^{\ve,\delta})$, $(\tilde{g}_{ij}^k)$ coincide with $(f_{j}^\ve)'$ and $g_{ij}'$ on arbitrarily big neighbourhoods of the origin as $\delta\to0$ and $k \to \infty$. 
    	
    	If $f_j^\ve$ is the approximation given in Lemma \ref{lemma:smoothing_f} and $f_j$ is convex, we can further assume that
    	\begin{equation}
    		\label{eq:difference_f_f_delta}
    		\tilde{f}_j^{\ve, \delta}(s)-\frac{\delta}{\gamma} \le f_j^{\ve}(s)\le f_j(s).
    	\end{equation}
    
    	\begin{proof}
    		First, we modify $f_j^\ve$ and $g_{ij}$ at infinity.
   			Let $0<\gamma<\alpha$, $\delta>0$ and define
			 \[
			      u_{\delta,j}(s)=\begin{cases}
				-(f_j^\ve)'(s), &\text{ if } s<1,\\
				\max\{-(f_j^\ve)'(s),\frac{\delta}{s^{\gamma+1}}\} &\text{ if } s\ge1.
					\end{cases}
			\]

Clearly, $u_\delta$ is continuous provided that $\delta<-(f_j^\ve)'(1)$ and agrees with $-(f_j^\ve)'$ on arbitrarily large neighborhoods of the origin up to choosing a $\delta$ small enough. We then define
\[
\tilde{f}_j^{\ve, \delta}(s) =\int_s^\infty u_{\delta,j}(\tau) d \tau.
\]
Let us observe that, by definition,  $f_j^\ve(s)+ \frac{\delta}{\gamma s^\gamma}\ge \tilde{f}^{\ve, \delta}_j(s)\ge \max\{f_j^\ve, \frac{\delta}{\gamma s^\gamma}\}$ for $s\ge1$ and $\tilde{f}_j^{\ve,\delta}(s)\ge f_j^{\ve}(s)$ for all $s$. So, naturally, $\tilde{f}_j^{\ve,\delta}$ satisfies  \eqref{assumption:values_f_g}-\eqref{assumption:homogenuity}. 

We now consider the $g_{ij}$. Let us observe that $\liminf_{s\to \infty} -s^{\alpha+1} g_{ij}'(s) = \ell < +\infty$. Indeed, let us assume by contradiction  that $\ell$ is equal to $+\infty$, applying De l'Hopital rule we find
\[
+\infty = \lim_{s\to+\infty} - \alpha \frac{g_{ij}'(s) }{1/s^{\alpha+1}} = \lim_{s\to+\infty} \frac{g_{ij}(s)}{1/s^\alpha} = \lim_{s\to+\infty}s^\alpha g_{ij}(s).
\]
However, \eqref{assumption:homogenuity} implies that $s^{\alpha}g_{ij}(s) \le g_{ij}(1)$ for $s>1$ and thus $\ell<g_{ij}(1)$, a contradiction.
Thus, there exists a sequence $\{s_k\}_k$ converging to $+\infty$ satisfying
\[
-g_{ij}'(s_k) < \frac{\ell +1}{ s_k^{\alpha+1}}.
\]

This prompts us to define
\[
w^k_{ij}(s) = \begin{cases}
	-g_{ij}'(s), &\text{ if } s<s_k,\\
	\min\{-g_{ij}'(s),\frac{\ell+1}{ s^{\alpha+1}}\} &\text{ if } s\ge s_k.
\end{cases}, \quad \tilde{g}_{ij}^k(s)  = \int_s^\infty w^k_{ij}(\tau) d \tau.
\]

Note that for $s\ge s_k$, $\tilde{g}_{ij}^k(s)\le \min \left\{ g_{ij}(s), \frac{ \ell+1}{\alpha s^\alpha}\right\}
$
and for $0\le s \le s_k$ we have $\tilde{g}_{ij}^k(s)\le g_{ij}(s)$ and thus $g_{ij}^k$ satisfy \eqref{assumption:values_f_g}-\eqref{assumption:homogenuity}.

 As far \eqref{assumption:attraction_infinity} is concerned, let us observe that for any $s>t>0$ we have
\[
-(f_{j}^\ve) '(s)+\sum_{i<j}g_{ij}'(s-t)\le -(\tilde{f}_j^{\ve, \delta})'(s)+\sum_{i<j}(\tilde{g}^k_{ij})'(s-t).
\]
Thus, if \eqref{assumption:attraction_infinity} is satisfied for $f_j^\ve$ and $g_{ij}$ for some $t_0,s_0$ it is satisfied for $\tilde{f}^{\ve, \delta}_j$ and $\tilde{g}_{ij}^k$ as well.

Finally, we turn to  \eqref{assumption:attraction_infinity_2}. 
	Recall that, thanks to \eqref{assumption:homogenuity}, the function $	s^{\alpha}g(s)$ is decreasing and so there exists an $L>0$ such that 
	\[
	\sup_{k} g(s_k)s^\alpha_k< L.
	\]
	On the other hand, \eqref{assumption:homogenuity} implies that for $0<s \le s_k$
	\[
	g_{ij}(s) \ge \frac{s_k^\alpha g(s_k)}{s^\alpha}\ge \frac{L}{s^\alpha} \text{ and so } -g_{ij}'(s) \ge \frac{\alpha L}{s^{\alpha+1}}.
	\]
Thus, setting $C = \min\left\{\ell+1,\alpha L\right\}$ we have that $-(\tilde{g}_{ij}^k)'\ge \frac{C}{s^{\alpha+1}}$. Moreover, if $s\ge 1$ $-(\tilde{f}_j^{\ve, \delta})'(s)\ge \frac{\delta}{s^{\gamma+1}}$.
Summing up we have 
\[
\frac{\delta}{s^{\gamma+1}}-\sum_{i<j}\frac{C}{(s-t_i)^{\alpha+1}} \le  -(\tilde{f}_j^{\ve, \delta})'(s)+\sum_{i<j}(\tilde{g}^k_{ij})'(s-t_i).
\]
Imposing $s-t_i\ge s/2n$ we find that the left-hand side is positive as soon as 
\[
s\ge \max\left\{\left(\frac{j C (2n)^{\alpha+1}}{\delta}\right)^{\frac{1}{\alpha-\gamma}},1\right\}.
\]

The last observation we make concerns the quantities $\vert f_j^\ve-\tilde{f}_j^\ve\vert$ and $\vert g_{ij}-\tilde{g}_{ij}^k\vert$. For $s \in \mathbb R$  (and, respectively, $s>0$) we have
\[
\vert f_j^\ve(s)-\tilde{f}_j^{\ve,\delta}(s) \vert \le \int_{\max \{s,1\}}^\infty \frac{\delta}{\tau^{\gamma+1}} \le \frac{\delta}{\gamma}, \quad \vert g_{ij}(s) -\tilde{g}_{ij}^k(s) \vert \le \int_{\max\{s_k,s\}}^\infty \frac{\ell +1}{\tau^{\alpha+1}} \le\frac{(\ell+1)}{\alpha s_k^{\alpha}}
\]

Combining this with the convexity assumption in \eqref{assumption:convexity} it follows that
\begin{equation*}
\tilde{f}_j^{\ve, \delta}(s)-\frac{\delta}{\gamma} \le f_j^{\ve}(s)\le f_j(s).
\end{equation*}

\end{proof}
\end{lemma}

\begin{proof}[Proof of Theorem \ref{thm:ve_critical_points}]

Let us now denote by $\tilde{q}_{\delta,k}^\ve$ any critical point of the functional
\[
\tilde{\mathcal{A}}(q) =  \int_0^T  \sum_{i=1}^n\frac{1}{2} \vert \dot{q}_i\vert^2+\tilde{f}_i^{\ve,\delta}(q_i)-\sum_{i<j}\tilde{g}^k_{ij}(q_j-q_i)
\]
in the set $\mathcal{H}$ whose existence is guaranteed by Theorem \ref{thm:general_existence_result}.

Now we show that, for $1/k$ and $\delta$  small enough, the solutions $\tilde q^{\ve}_{\delta,k}$ do not depend on $\delta$ or $k$ and are critical points of $\mathcal{A}_\ve$ as well. 

From Section \ref{sec:boundary_disk} we know that the action value of the $\q$
can be bound by 
\[
c_{\delta}^\ve  =  \inf_{y \in \mathcal{C}} \mathcal{F}_{\delta}^\ve(y),
\]
where the $ \mathcal{F}_{\delta}^\ve$ is the functional defined in \eqref{eq:minimization_problem_kepler}, using the $\tilde f_j^{\ve,\delta}$. We define $c^\ve$ to be the minimum level of $\mathcal{F}_0^\ve$, i.e. the functional defined in \eqref{eq:minimization_problem_kepler} using the $f_j^\ve$. Let $c_0 = \inf_{y \in \mathcal{C}} \mathcal{F}(y)$ where $\mathcal{F}$ stands for the one defined using the $f_j$, namely \[
   \mathcal{F}(y) =  \int_0^T\frac12 \sum_{i=0}^{n-1}\vert \dot y_i \vert^2+f_{i+1}(y_i).
\] Thanks to \eqref{eq:difference_f_f_delta} we have
\[
    c_{\delta}^\ve \le \frac{n \delta T}\gamma + c^\ve \le \frac{n \delta T}\gamma  + c_0.
\]
Assumption \eqref{assumption:homogenuity} then implies that any critical point $x$ of $\mathcal{F}_\delta^\ve$
 satisfy the following bound on the $L^2$ norm of the velocity
\[
    \Vert \dot{x}\Vert_2^2 \le \frac{2 \alpha (c_0 \gamma+n \delta  T)}{\gamma(\alpha+2)}.
\]
Thus, since on $\mathcal{C}$ a Poincaré inequality holds, minimizers are contained in a compact neighbourhood of the origin. Thus, for $\delta$ small enough $\mathcal{F}^\ve$ and $\mathcal{F}_{\delta}^\ve$ have the same minimizers and all the $\q$ satisfy
\[
\tilde{\mathcal{A}}(\q) \le c_0 +\frac{n\delta T}{\gamma}.
\]
Using \eqref{assumption:homogenuity}, this can be rephrased as
\[
\Vert \dot{\tilde{q}}_{\delta,k}^\ve \Vert_2^2 \le \frac{2 \alpha (c_0 \gamma+ n\delta  T)}{\gamma(\alpha+2)}.
\]
To simplify notation we now fix  $\ve,\delta$ and $k$ and denote by $q_i$ the $i-$th component of $\q$.
The above inequality implies that the absolute value of $q_1$ is controlled by $\sqrt{\frac{2T \alpha (c_0 \gamma+n \delta  T)}{\gamma(\alpha+2)}}$. Thanks to Assumption \eqref{assumption:attraction_infinity}, taking $t_0 = \frac{2 T \alpha (c_0 \gamma+n \delta  T)}{\gamma(\alpha+2)}$ there exists $s_0$ such that 
$q_2(t)<s_0$ at some $t\in [0,T]$. Indeed, if that were not the case the same variation used in the proof of Proposition \ref{prop:no_solution_eigenvalue}, $\q+\sigma w_2$, where $w_2$ is defined in \eqref{eq:def_w_j}, would give
\[
0 =  \langle\nabla\tilde{\mathcal{A}} (\q),w_2\rangle= \int_0^T \sum_{j=2}^{n} \left( (f_j^{\ve,\delta})'(q_j)- (\tilde{g}^k_{j1})'(q_j-q_1) \right)<0,
\]
Thus, the absolute value of $q_2$ is controlled by $s_0+\sqrt{T}t_0$. Iterating the argument we easily see that there is a constant $C$, independent on $\ve,k$ and $\delta$ such that
\[
\vert q_i(t)\vert \le C,
\]
for all $i$ and all $t \in [0,T]$. Thus, for $\delta$ and $1/k$ sufficiently small all the $\q$ are contained in a compact set in which $(f_i^{\ve,\delta})' = (f_i^\ve)'$ and $(\tilde{g}^k_{ij})' = g_{ij}'$. This also implies that the value of $\mathcal{G}$ on $\q$ is independent of $\delta$ and $k$ and that
\[
\mathcal{A}_\ve(\q) = \tilde{\mathcal{A}}(\q) \le\inf_{\delta} (c^\ve + n\delta T \gamma^{-1})= c^\ve \le   \inf_{\delta} (c_0 + n\delta T \gamma^{-1})= c_0,
\]
concluding the proof.
\end{proof}

\subsection{A priori estimates}
	In this section we collect some a priori estimates and qualitative properties of the sequence $(q^\ve)$ appearing in Theorem \ref{thm:ve_critical_points}.

	\begin{lemma}
		\label{lemma:uniform_bounds}
		Let us assume that $\eqref{assumption:values_f_g}-\eqref{assumption:attraction_infinity}$ hold and that the $f_j$ are convex.
        Let us consider the functional $\mathcal{F}(y) = \int_0^T\frac12 \sum_{i=0}^{n-1}\vert \dot y_i \vert^2+f_{i+1}(y_i)$ and
            \[
                c_0 = \inf_{y \in \mathcal{C}} \mathcal{F}(y)
            \]
            as defined in Section \ref{sec:boundary_disk}.
		There exists $\bar\ve>0$ such that, for any $\ve\in(0,\bar\ve)$ the following assertions hold true:
		\begin{enumerate}[label=\roman*)]
			\item $\mathcal{A}_\ve(q^\ve)\le c_0$ and $\Vert \dot{q}^\ve\Vert_2\le \frac{2 \alpha c_0}{\alpha+2}$,
			\item  the total energy is bounded, independently of $\ve$, in fact 
			\[
			-\frac{c_0}{T}<h_\ve=\sum_{i=1}^n\frac12(\dot{q}_i^\ve)^2 -  f_i^\ve(q_i^\ve)+\sum_{j<i}g_{ij}(q^\ve_i-q^\ve_j)< \frac{\alpha-2}{\alpha+2}c_0<0.
			\]
			\item There exists a constant $C>0$, not depending on $\ve$, such that $\Vert q^\ve \Vert_{H^1}\le C$. In particular, there exists a subsequence $q^\ve$ weakly converging to a function $\bar{q} \in \bar{\mathcal{H}},$ where $\bar{\mathcal{H}}$ is the weak $H^1$-closure of $\mathcal{H}$.    
		\end{enumerate}
		\begin{proof}
		Point $i)$ follows directly from point $iii)$ of Lemma \ref{lemma:construction_boundary} and the last part of the proof of Theorem \ref{thm:ve_critical_points}. 
		
		The bounds on the energy $h_\ve$ are a straightforward consequence of the following inequalities
		\[
		\frac{2\alpha}{\alpha+2}\mathcal{A}_\ve(q^\ve) \ge \Vert \dot{q}^\ve \Vert_2^2, \quad 
		T h_\ve = \Vert \dot{q}^\ve \Vert_2^2-\mathcal{A}_\ve(q^\ve) \le \frac{\alpha-2}{\alpha+2} \mathcal{A}_\ve(q^\ve) .
		\]
		Point $iii)$ follows from the same argument used in the  proof of Theorem \ref{thm:ve_critical_points}. Indeed, we proved that there exists a constant $C>0$ such that $\vert q^\ve(t)\vert\le C$ for all $t \in[0,T]$. Thus, $\Vert q \Vert_2\le \sqrt{T}C$. %Indeed, we have established in point $i)$ that $\Vert \dot{q}^\ve\Vert_2$ is bounded. Let us assume by contradiction that there exists a sequence $\ve_m\to 0$ such that $\Vert q^{\ve_m}\Vert_2 \to +\infty$. Notice that, since $q_1^{\ve_m}(0) = 0,$ $\int_0^T(q_1^{\ve_m})^2$ is uniformly bounded. For this reason, we can assume that there exists a $j>0$  such that, up to subsequence
%		\[
%		\int_0^T (q^{\ve_m}_{j+1}-q^{\ve_m}_{j})^2  \to +\infty,
%		\]
%        whereas all $\int_0^T(q^{\ve_m}_{i+1}-q^{\ve_m}_{i})^2$ remain bounded for all $i<j$. It follows that there exists a constant $C_1$ such that $\vert q_i^{\ve_m}\vert \le C_1$. As in the proof of $ii)$, Lemma \ref{lemma:proof_proportionality_gradients}, by the mean value theorem, there is a point $t^*\in[0,T]$ such that $\int_0^T (q^{\ve_m}_{j+1}-q^{\ve_m}_{j})^2 =T (q^{\ve_m}_{j+1}-q^{\ve_m}_{j})^2(t^*) $. Since $\Vert\dot{q}^{\ve_m} \Vert^2_2$ is bounded, it follows that $\min_{[0,T]} (q^{\ve_m}_{j+1}(t)-q^{\ve_m}_{j}(t))\to +\infty$ as $m \to \infty$ as well. Let us now consider the same variation considered in the proof of Proposition \ref{prop:no_solution_eigenvalue}.  Let $w_j$ be the vector defined in \eqref{eq:def_w_j}. Then
%		\[
%		\langle\nabla\mathcal{A}_\ve (q^{\ve_m}),w_j\rangle= \int_0^T \sum_{i=j+1}^{n} \left( (f_i^\ve)'(q^{\ve_m}_i)-\sum_{k=1}^jg_{ik}'(q^{\ve_m}_i-q^{\ve_m}_k) \right)<0,
%		\]
%		thanks to assumption \eqref{assumption:attraction_infinity}. Thus $\nabla \mathcal{A}_\ve \ne 0$ at $q^{\ve_m},$ a contradiction.
%		
\end{proof}
	\end{lemma}

	   Let us denote by $\bar{q} $ a weak limit as in point $iii)$ of Lemma \ref{lemma:uniform_bounds}. 
	\begin{lemma}
		\label{lemma:concavity_q1}
		For every $\ve>0$, $q_1^\ve$ is strictly concave and increasing. Moreover, for $\ve$ small enough and $\delta\in(0,T)$, $q^\ve_1(t)\vert_{[\delta,T]}\ge C_\delta$, for some $C_\delta>0$  independent of $\ve$. In particular, for any $\delta\in(0,T)$ and $t\ge \delta$, $\bar{q}_1(t)\ge\bar{q}_1(\delta)>0$.
        \begin{proof}
           Thanks to the motion equations $q_1^\ve$ is strictly concave. Since $\dot q_1^\ve(T) = 0$ it is also monotone increasing.  
           Following the proof of \cite[Proposition 3.3]{helium_frozen} and using the motion equation, for any $t \in (0,T)$ we write
           \[
           q_1^\ve(t) = \int_0^t \dot{q}_1^\ve(\tau) d\tau = -\int_0^t\int_\tau^T (f^\ve_1)'(q^\ve_1(s))+\sum_{j=2}^ng_{1j}'(q^\ve_j(s)-q^\ve_1(s))dsd\tau.
           \]
           By contradiction, let us assume that there exists a sequence $\ve_m\to 0$ for which $q_1^{\ve_m}(t)\to 0$. Then, $q_1^{\ve_m}$ converges to $0$ uniformly on $[0,t]$ since it is monotone. Without loss of generality, we can assume that it converges weakly to some function $\bar{q}$. Thanks to Fatou Lemma, we have
           \begin{align*}
           0&=\liminf_m q_1^{\ve_m}(t) \ge\int_0^t \liminf_m \left(\int_\tau^T-(f^{\ve_m}_1)'(q^{\ve_m}_1(s))-\sum_{j=2}^ng_{1j}'(q^{\ve_m}_j(s)-q^{\ve_m}_1(s))ds\right)d\tau
            \\ &\ge\int_0^t\int_\tau^t\liminf_{n\to \infty}-(f^{\ve_n}_1)'(q_1^{\ve_n})- \int_0^t\int_t^T (f_1)'(\bar q_1)
           -\int_0^t\int_\tau^T \sum_{j=2}^n g_{1j}'(\bar{q}_j-\bar{q}_1) \\
           &>\int_0^t\int_\tau^t\liminf_{n\to \infty}-(f^{\ve_n}_1)'(q_1^{\ve_n}).
           \end{align*}
            However, the right-hand side is strictly positive  (possibly $+\infty$). So, we reached a contradiction.
        \end{proof}
	\end{lemma}

\subsection{Convergence to a solution of \eqref{eq:helium_equation}}

	Now, we are ready to prove Theorem \ref{thm:solution_helium} and to show that an approximating sequence of minimizers $(q^\ve)$ approaches a solution of the original problem \eqref{eq:helium_equation}.

	\begin{proof}[Proof of Theorem \ref{thm:solution_helium}]
	We prove the statement in several steps. The strategy is to show that there exists a subsequence of the sequence $(q^\ve)$ given in Theorem \ref{thm:ve_critical_points}, which converges to a solution of \eqref{eq:helium_equation}. 
		
	As in the previous section, let us denote by $\bar q$ the weak limit of $q^{\ve}$. As a first step, we show that, if $\bar q_j(t^*) =\bar q_{j+1}(t^*)$ for some $j$ and $t^*\in [0,T]$, then $t^*=0$. This will allow us to show that $(q^\ve)$ has a subsequence converging in $C^2([\delta,T])$ to a solution of \eqref{eq:helium_equation}, for any $\delta>0$ (see Lemma \ref{lemma:concavity_q1}). Then, we show that any such solution cannot have a collision at $t=0$. In turn, this will imply that the convergence is in $C^2([0,T])$ for $q_i^\ve$ and $i\ge 2.$
		
	Let us assume that there exists indeed an instant $t^*>0$ for which $\bar q_j(t^*) = \bar q_{j+1}(t^*)$ for some $j$. Thanks to Lemma \ref{lemma:concavity_q1}, we know that there exists a constant $C>0$ such that  
	\[
		\liminf\limits_{\ve\to 0} \min_{t\in [t^*/2,T]}q^\ve_1(t)>C.
	\]
	Let us consider the energy  $h_\ve$ of $q^\ve$.  Thanks to point $ii)$ of Lemma \ref{lemma:uniform_bounds}, we know it is uniformly bounded. However, evaluating $h_\ve$ at $t^*$ we obtain:
	\[
		\liminf\limits_{\ve\to 0} h_\ve \ge \lim\limits_{\ve\to 0}\left[ - \sum_{i=1}^n f_i^\ve(q_i^\ve(t^*))+\sum_{j<i}g_{ij}(q^\ve_i(t^*)-q^\ve_j(t^*))\right] = +\infty,
	\]
	yielding a contradiction. Thus, there are no collisions on $(0,T].$
		
	We now show that, up to subsequence, $(q^\ve)$ converges in the $C^2$ norm to $\bar q$ on any interval $[\delta, T]$. Thanks to Lemma \ref{lemma:concavity_q1} and what we have just proved, we know that $\min_{i, t \in[\delta, T]} \{q_i^\ve(t),q^\ve_{i+1}(t)-q_i^\ve(t)\} \ge c(\delta)$ for some strictly positive constant $c(\delta)$ independent on $\ve$. Combining this observation with point $ii)$ of Lemma \ref{lemma:uniform_bounds}, we know that the sequence of derivatives $\dot{q}_i^\ve$ is  uniformly bounded in $C^0([\delta,T])$. Using the equation for $\ddot{q}_i^\ve$, we easily see that they are also equicontinuous. The same argument applies to $\ddot{q}_i^\ve$ as well. The statement now follows from an application of Ascoli-Arzelà Theorem (see \cite[Theorem 3.4]{helium_frozen}). Let us remark that, after proving $C^2$ convergence, we can assume that the sequence of energies $h_\ve$ converges to some constant $\bar h$ and that $\bar q$ satisfies the energy equation, with energy $\bar{h}$, for $t\in(0,T]$. 
		
	Now, we show that collisions are possible only if they are collisions with the origin. Let us assume now that there exists a $j$ such that $\bar q_j(0) = \bar q_{j-1}(0) = q_0 >0$. Without loss of generality, we can assume that $j$ is the largest with this property. Similarly, define $l$ to be the smallest  index such that $\bar q_{l+1}(0)=\bar q_l(0) = q_0$. Clearly $j>l$. Let us consider the following \emph{cluster energy}
	\[
		h_{jl} = \sum_{i=l}^j\frac12 \dot{\bar{q}}^2_i -f_i(\bar q_i)+\sum_{l\le i<k\le j} g_{ik} (\bar q_k-\bar q_i).
	\]
	A straightforward computations shows that:
	\begin{equation}\label{eq:derivative_cluser_energy}
		\begin{aligned}
	  	\dot{h}_{jl} &= \sum_{i=l}^j  \dot{\bar{q}}_i \left( \ddot{\bar{q}}_i -f'_i(\bar q_i)+\sum_{l\le k< i} g'_{ik} (\bar q_k-\bar q_i)-\sum_{i<k\le j} g'_{ik} (\bar q_k-\bar q_i)\right) \\&= \sum_{i=l}^j \, \dot{\bar q}_i \left(\sum_{k>j } g'_{ik} (\bar q_k-\bar q_i) -\sum_{k<l } g'_{ik} (\bar q_k-\bar q_i)\right) .
		\end{aligned}
	\end{equation}
    Moreover, there exists some positive constant $C_1$ such that 
    \begin{equation*}
      	\max \left\lbrace \Vert g'_{ik}(\bar q_i-\bar q_k)\Vert_{\infty} \,:\, k<l \text{ or }k>j,\, l\le i\le j \, \right\rbrace\le C_1,
    \end{equation*}
    since we are assuming that $\bar q_{j+1}(0)> \bar q_{j}(0)$, $\bar q_{l-1}(0)<\bar q_l(0)$ and we have already proved that there is no collision at $t>0$. Integrating \eqref{eq:derivative_cluser_energy} over $(\delta,t]$, for $0<\delta<t$, and using that $\dot{\bar q}$ is in $L^2$, we obtain:
    \begin{equation*}
    	\limsup\limits_{\delta\to 0}\vert h_{jl}(t)-h_{jl} (\delta) \vert \le C_2,
	\end{equation*}
    for some positive constant $C_2$.
	However, since $\bar q_l$ remains bounded away from zero, this implies that $\lim_{\delta\to 0^+}h_{jl}(\delta) =+\infty$, a contradiction. 
	
	Thus, we can assume that, if a collision occurs, it is with the origin and there exists a $j>1$ such that $\bar q_{j+1}(0)>0$ and $\bar q_j(0)=0$. Let us consider the motion equation for $\bar q_{j}- \bar q_1$, holding for $t>0$. It reads:
	\begin{equation}
		\label{eq:differenza_q_j_q_1}
		\begin{aligned}
		\ddot {\bar q}_{j}- \ddot{\bar q}_1 &= f_{j}'(\bar q_{j})-f_1'(\bar q_1) -\left( \sum_{k=1}^{j-1}g'_{jk}(\bar q_{j}-\bar q_k)+\sum_{k=2}^{j}g'_{1k}(\bar q_{1}-\bar q_k)\right) \\
		&\quad + \sum_{k=j+1}^{n}g'_{kj}(\bar q_k-\bar q_j)-g'_{k1}(\bar q_k-\bar q_{1}).
		\end{aligned}
	\end{equation}
	Let us observe that, thanks to assumption \eqref{assumption:convexity}, the term $f_{j}'(\bar q_{j})-f_1'(\bar q_1)$ is bounded by below by a constant $-C_3$. Moreover, there exists some positive constant $C_4$ such that 
	\begin{equation*}	  
		\left\Vert\sum_{k=j+1}^{n}g'_{kj}(\bar q_k-\bar q_j)-g'_{k1}(\bar q_k-\bar q_1) \right\Vert_{\infty} \le C_4.
	\end{equation*}
	Indeed, we are assuming that $\bar q_{j+1}(0)> \bar q_{j}(0)$ and we have already proved that there are no collisions in the limit $\bar{q}$ for $t>0$. Hence, $\bar q_{j+1}-\bar q_1$ is positive on $[0,T]$.
	
	Let us now integrate $\ddot {\bar q}_{j}- \ddot{\bar q}_1 $ over $[\delta,T]$. Using that $\dot {\bar q}_{j}(T)- \dot{\bar q}_1(T) =0$, we obtain:
	\begin{equation*}
		\dot {\bar q}_{j}(\delta)- \dot{\bar q}_1(\delta)\le 2(C_3+C_4) T +2 \left( \int_\delta^T \sum_{k=1}^{j}g'_{jk}( \bar q_{j}- q_k)+g'_{1k}(\bar q_{1}-\bar q_k) \right).
	\end{equation*} 
	Let us observe that $\bar{q}_j-\bar{q}_1$ is convex as soon as $g'_{j1}(q_j-q_1)\le- (C_3+C_4)$. Since $\bar{q}_j-\bar q_1 \ge0$ and $\bar{q}_j(0)-\bar q_1(0) = 0$, there exists a small interval $[0,\tau]$ on which $\dot{\bar{q}}_j-\dot{\bar{q}}_1 \ge0.$
	Finally let us observe that, thanks to convexity, $\lim_{\delta\to 0^+}\dot{\bar{q}}_j(\delta)-\dot{\bar{q}}_1(\delta) = C_5$ exists finite and is non negative. Thus $\bar q_j-\bar q_1 \le C_6 t$ for a positive constant $C_6$ and $t \in [0,\tau]$. However, assumption \eqref{assumption:homogenuity} implies that:
	\begin{equation*}
		g_{1,j}'(s)\le -\alpha g_{1j}(1)s^{-(\alpha+1)} \text{ for all } s \in (0,1].
	\end{equation*}
	Thus the integral $\int_\delta^Tg'_{1j}(\bar{q}_j-\bar{q}_1)\to -\infty$ as $\delta \to 0^+$ and therefore  $\dot{\bar{q}}_j(\delta)-\dot{\bar{q}}_1(\delta) \to -\infty$ as $\delta\to0^+$, a  contradiction.
	\end{proof}
	  
 \section{Zero-charge case}
 \label{sec:zero_charge}
 Let $f_i$ and $g_{ij}$ satisfy assumptions \eqref{assumption:values_f_g}-\eqref{assumption:convexity} and let  $\mu \in(0,1]$. In this section we consider the asymptotic behaviour of frozen planet orbits of
 \begin{equation}
 	\label{eq:helium_equation_mu}
 	\ddot{q}_i = f_i'(\vert q_i \vert ) - \mu\sum_{j=1}^{i-1} g'_{ij}(\vert q_i-q_j\vert )+\mu\sum_{j = i+1}^n g'_{ij}(\vert q_j-q_i \vert),
 \end{equation} 
 as the positive parameter $\mu$ tends to zero. Before diving into this analysis, we need to make some preliminary observations concerning the \emph{action value} of the frozen planet orbits obtained in Theorem \ref{thm:solution_helium}. To this extent, we need to take a closer look at the auxiliary minimization problem introduced in Section \ref{sec:boundary_disk} in  \eqref{eq:minimization_problem_kepler}-\eqref{eq:c_0}. Let us introduce two small parameters $0\le\ve_1\le\ve_2$ and consider the functional
\begin{equation*}
	\mathcal{F}_{\ve_1,\ve_2}(y) = \int_0^T	 \frac12 \vert \dot y_0 \vert^2+f^{\ve_1}_{1}(y_0)+ \frac12 \sum_{i=1}^{n-1}\vert \dot y_i \vert^2+f^{\ve_2}_{i+1}(y_i),
\end{equation*}
and its minimum level
\[
  c_0^{\ve_1,\ve_2} = \min_{y \in \mathcal{C}} \mathcal{F}_{\ve_1,\ve_2}(y),
\]
where \[
	\mathcal{C} := \{ y \in H^1([0,T],\mathbb R^n) :y_0(0)=0,\, y_{2k}(T) = y_{2k+1}(0), \, y_{2k+1}(T) = y_{2k}(0)\}
	\]
	for $0\le k\le\lfloor (n-1)/2 \rfloor$. Let us denote by $x_{\ve_1,\ve_2}$ a minimizer of $\mathcal{F}_{\ve_1,\ve_2}$ in $\mathcal{C}$. 
Similarly, we consider the functional	\[\mathcal{F}(y) = \int_0^T \frac12 \sum_{i=0}^{n-1}\vert \dot y_i \vert^2+f_{i+1}(y_i)
\]
defined using the unregularized potential.  Its minimum is denoted by
	\[
		c_0 = \min_{y \in \mathcal{C}} \mathcal{F}(y).
    \]
We will call $x$ any minimizer of $\mathcal{F}$ in $\mathcal{C}$.
Let us point out that the potential functions involved in the definition of $\mathcal{F}_{\ve_1,\ve_2}$ satisfy additionally assumption \eqref{assumption:extra_assumptions_f} and this definition agrees with the one given in \eqref{eq:minimization_problem_kepler}. Moreover, if $\ve_1=\ve_2$, we recover the approximating problems studied in Section \ref{sec:proof_thm1}.
\begin{lemma}
	\label{lemma:action_level}
	The following facts hold true:
	\begin{enumerate}[label=\roman*)]
		\item For any $\ve_2\ge \ve_1 \ge 0$, if  $x_{\ve_1,\ve_2} = (x_{\ve_1,\ve_2}^0,\dots,x_{\ve_1,\ve_2}^{n-1})$ is a minimizer of $\mathcal{F}_{\ve_1,\ve_2}$ in $\mathcal{
        C}$, then $x_{\ve_1,\ve_2}^i$ is increasing when $i$ is even and decreasing when $i$ is odd.
		\item For any $\ve_2\ge \ve_1\ge 0,$ the minimizers $x_{\ve_1,\ve_2}$ are unique and are given by folded $nT-$brake orbits, as defined in \eqref{eq:nT_brake}.
		\item For any $\ve_2\ge\ve_1\ge 0$, $c^{\ve_1,\ve_2}_0 \le c^{\ve_1,\ve_1}_0\le c_0$ and $c_0^{\ve_1,\ve_2} \to c_0$ as $\ve_1,\ve_2 \to0$. Moreover, there exists $\ve_2^0>0$ such that, for any $0<\ve_1\le\ve_2\le\ve_2^0$, $c^{\ve_1,\ve_2}_0 = c_0^{\ve_1,\ve_1}$, i.e., it does not depend on $\ve_2$.
	\end{enumerate} 	
\begin{proof}
	Point $i)$ follows from the monotonicity of $f_1^{\ve_1}$ and $f_j^{\ve_2}$. Indeed, critical points of $\mathcal{F}_{\ve_1,\ve_2}$ solve the equations $\ddot{q}_j = (f_j^{\ve})'(q_j)$, so each $q_j$ has at most one maximum point. Let us assume that $\dot{q}_j(t^*) =0$ for  $t^*\in(0,T)$ and define:
	\begin{equation*}
		\tilde q_j(t) = 
		\begin{cases}
            q_j(t) \text{ if } t\le t^* \\
           2 q_j(t^*)- q_j(t) \text{ otherwise}
 		\end{cases}.
	\end{equation*}
    It is straightforward to check that, defining  $\tilde q_k(t) =  q_k(t)$ for $k<j$ and $\tilde q_k(t) = q_k(t) +2( q_j(t^*)- q_j(T))$ for $k\ge j+1$, we end up with a curve with lower action.

   Point $ii)$ is a consequence of assumption \eqref{assumption:convexity}. Any $x_{\ve_1,\ve_2} \in \mathcal{C}$ minimizer determines a $C^1$ curve $\eta$ on $[0,nT]$ setting 
   \[
   \eta(t) = \begin{cases}
   x_{\ve_1,\ve_2}^k\left (t-k T\right), &\text{ if }t \in [kT,(k+1)T] \text{ and $k$ even},\\
    x_{\ve_1,\ve_2}^k\left ((k+1) T-t\right), &\text{ if }t \in [kT,(k+1)T] \text{ and $k$ odd}
   \end{cases}
   \]
   for $k=0,\dots,n-1$, satisfying $\eta(0) =0$ and $\dot\eta(T) =0$. Assume that two minimizers exist and let $\eta_1$ and $\eta_2$ be the corresponding $C^1$ curves. We may assume that either $\eta_2\ge \eta_1$ or $\eta_1\ge \eta_2$, otherwise we would be able to produce minimizers which are not $C^1$. Thus, assume that $\eta_2\ge \eta_1$. It follows that $\eta_2-\eta_1$ is convex, equal to zero at $t=0$ and positive with a minimum at $nT$. Thus $\eta_1 = \eta_2$. 
   
   We prove $iii)$. Since $f_j^{\ve_i} \le f_j$, then $c_0^{\ve_1,\ve_2}\le c_0^{\ve_1,\ve_1}\le c_0$. An application of Fatou Lemma shows that $c_0^{\ve_1,\ve_2}\to c_0$ as $\ve_2,\ve_1\to 0$ (see the proof of \cite[Lemma A.1]{helium_frozen}).
   The last assertion follows from the concavity of minimizers (compare with Lemma \ref{lemma:concavity_q1}). An application of the fundamental theorem of calculus yields 
   \[
      x_{\ve_1,\ve_2}^0(t) = \int_0^t \dot{x}_{\ve_1,\ve_2}^0(\tau) d\tau =\dot{x}_{\ve_1,\ve_2}^0 (T) t - \int_0^t\int_\tau^T (f^{\ve_1}_1)'( x_{\ve_1,\ve_2}^0(s))dsd\tau.
   \]
   Observe that $\dot{x}_{\ve_1,\ve_2}^0(T)\ge0$ since $\eta$ is increasing. 
   If ${x}_{\ve_1,\ve_2}^0(T)$ were to converge to zero as $\ve_2\to 0$, an application of Fatou Lemma would yield a contradiction since 
   \[
         0\ge-\frac{T^2(f_1^{\ve_1})'(0)}{2}>0.
   \]
   Thus $ x_{\ve_1,\ve_2}^j(t)\ge  x_{\ve_1,\ve_2}^0(T)>\ve_2$ for any $t\in[0,T]$ and $j\ge1$, provided that $\ve_2$ is small enough. So $f_j^{\ve_2}( x_{\ve_1,\ve_2}^j) = f_j( x_{\ve_1,\ve_2}^j)$ for all $j\ge1$.
\end{proof}
\end{lemma}

In particular, if $c_\mu$ denotes the action level of a $\mu-$frozen planet orbit obtained in Theorem \ref{thm:solution_helium}, we have $c_\mu\le c_0$ (see point $iii)$ of Lemma \ref{lemma:construction_boundary} and Lemma \ref{lemma:uniform_bounds}).

\begin{lemma}\label{lem:minimisers_D} 
Let $\bar{\mathcal{H}}$ be the weak $H^1$-closure of $\mathcal{H}$ and consider the set
\[
	\mathcal{D} =\left \{q\in\bar{\mathcal{H}} : \min_{t\in[0,T]} q_i(t)-q_{i-1}(t) = 0 \text{ for all }i\right \}.
\]
Then, any minimizer of $\mathcal{F}_{\ve_1,\ve_2}$ in $\mathcal{D}$ is also a minimizer in the set $\mathcal{C}$. 
\begin{proof}
	Let us first show that minimizers of $\mathcal{F}_{\ve_1,\ve_2}$ in $\mathcal{D}$ have exactly one collision instant between $q_j$ and $q_{j+1}$ for each $j$. 
	
	Let $t^* \in (0,T]$ be a collision instant between $q_j$ and $q_{j+1}$. Now, we show that $\sum_{k =j+1}^n \dot{q}_k$ is non-positive a.e. on $[0,t^*]$. Let us define:
	\[
	A_+ = \left\lbrace t<t^*:   \sum_{k =j+1}^n \dot{q}_k(t)> 0\right\rbrace.
	\]
	Assume by contradiction that $A_+$ has positive measure $m(A_+)$ and define the function:
	\[
	\varphi(t) = m(A_+)-\int_0^t \chi_{A_+}(\tau)d\tau.       	
	\]
	By construction $\varphi(t^*) =0$, $\varphi\ge0$ and it is not identically zero. Consider the variation $q +\ve \varphi w_j$, where $w_j$ is defined in \eqref{eq:def_w_j} and $\ve\ge0$. The right partial derivative of $\mathcal{F}_{\ve_1,\ve_2}$ at zero reads:
	\[
	\partial_\ve \mathcal{F}_{\ve_1,\ve_2}(q+\ve \varphi w_j)\vert_{\ve=0} = \int_0^T \dot{\varphi} \sum_{k =j+1}^n \dot{q}_k+ \varphi  \sum_{k =j+1}^n (f^{\ve_2}_k)'(q_k) <0,
	\]
	and so $q$ cannot be a minimizer. A completely analogous statement holds for collision instants $t^*\in[0,T)$: $\sum_{k =j+1}^n \dot{q}_k$ is non-negative a.e. on $[t^*,T]$ .
	
	Let us now assume that $t^*<t^{**}$ are two collision instants between $q_j$ and $q_{j+1}$. By the previous argument, $\sum_{k =j+1}^n \dot{q}_k=0$ a.e. on $[t^*,t^{**}]$. Thus, for any positive function $\varphi$ with compact support on $[t^*,t^{**}]$ we have
	\[
	\partial_\ve \mathcal{F}_{\ve_1,\ve_2}(q+\ve \varphi w_j)\vert_{\ve=0} = \int_0^T \varphi  \sum_{k =j+1}^n (f^{\ve_2}_k)'(q_k) <0,
	\]
	which is again a contradiction. Thus, collision instants are unique.
	
	The next step is to prove that only double collisions are admissible and determine how the derivatives change after collision.	Assume that $t^* \in (0,T)$ is a collision instant which is not a double collision. Then, $q_{j+1}(t^*)=q_{j}(t^*)=q_{j-1}(t^*)$ for some $j$. Since $q_j$ has no other collisions, $\dot{q}_j(0) =\dot{q}_j(T) =0$. However, at least one between  $\dot{q}_{j-1}(0) $ and $\dot{q}_{j-1}(T)$ is equal to zero since $q_{j-1}$ has exactly one collision with $q_{j-2}$. Without loss of generality assume that $\dot{q}_{j-1}(0) =0 $. Then, $q_j-q_{j-1}$ is convex, positive and with a critical point at $0$, but vanishes at $t^*$, a contradiction. Thus, collisions at $t^*\in (0,T)$ are necessarily double.    
	
	Let us observe that derivatives at an internal collision point $t^*$ must satisfy an elastic reflection law. Indeed, let us fix $\ve>0$ small enough so that $[t^*-\ve,t^*+\ve]\subseteq(0,T)$ and there is no further collision. It is straightforward to see that the curve $\eta$ defined on $[-\ve,\ve]$ with value in $\mathbb{R}^2$
	\[
	\eta(s) = \begin{cases}
		(q_{j+1}(s+t^*),q_j(s+t^*)) \text{ if }s\le0\\
		(q_{j}(s+t^*),q_{j+1}(s+t^*)) \text{ if }s>0
	\end{cases}
	\] 
	minimizes the action $\tilde{\mathcal{A}}(\eta) = \int_{-\ve}^\ve\frac12\Vert \dot{\eta}\Vert^2+V(\eta)$, where $V$ is the function:
	\begin{equation*}
		V(\eta) = V(x,y) =\begin{cases}
			f_j^{\ve_j}(x)+f_{j+1}^{\ve_j}(y), \text{ if } x\le y\\
			f_j^{\ve_j}(y)+f_{j+1}^{\ve_j}(x), \text{ if }x>y
		\end{cases}
	\end{equation*}
    where $\ve_j$ is equal to $\ve_1$ if $j=1$ and $\ve_2$ if $j\ge2$.
	Since $V$ is Lipschitz continuous, $\eta$ must be $C^1$ and the assertion about the right/left derivatives follows.
	
	Let us observe that the collisions at the endpoints must be double too. Indeed, assume by contradiction that $q_{j+1}(T) = q_j(T) = q_{j-1}(T)$ for some $j$. Let us observe that necessarily $\dot{q}_j(0)=0$ and so, since $q_j$ is concave, ${q_j}(0)>q_j(T)$. The curve $\tilde q=(\tilde{q}_1,\ldots,\tilde{q}_n)$ defined as 
	\[
		\tilde{q}_k(t) = 
		\begin{cases}
			\begin{aligned}
			&q_k(t) &\text{if}\ k\le j \\
			&q_k(T-t)+q_j(0)-q_j(T) &\text{if}\ k>j
			\end{aligned}
		\end{cases}
	\]
	has lower action thanks to the monotonicity of the function $f_k^{\ve_j}.$ Therefore, collisions at $T$ are always double as well. In particular, we obtain that $\dot{q}_j(T) = -\dot{q}_{j+1}(T)$, for any collision at $T$. An analogous statement holds for collisions at $t=0.$ 
	
	We have thus proved that there exist exactly $n-1$ double collisions satisfying an elastic reflection's law. 
	
	Now, we show that there are no collision instants in $(0,T)$.  Let us consider the case of a collision between $q_n$ and $q_{n-1}$ at an internal point $t^*$. This implies that $\dot {q}_n(T) = \dot{q}_n(0) =0$. Without loss of generality, let us assume that $q_{n-1}$ has no collisions with $q_{n-2}$ on $[t^*,T]$. Then, the function
	\[
	\eta(t) = \begin{cases}
		q_n(t) \text{ if } t\le t^* \\
		q_{n-1}(t)\text{ if } t> t^*
	\end{cases}
	\]
	is concave and has two critical points, a contradiction. Thus $q_n$ and $q_{n-1}$ meet at one of the endpoints of the interval. Without loss of generality, let us assume that they meet at $T$. Let $t^*_1$ be the collision instant between $q_{n-1}$ and $q_{n-2}$, and assume by contradiction that it is an internal point. Hence, both the functions
	\[
	\eta_1 (t) = 
	\begin{cases}
		\begin{aligned}
		&q_{n-1}(t) &\text{if}\ t \in [0,t_1^*]\\
		&q_{n-2}(t) &\text{if}\ t \in [t_1^*,T]
		\end{aligned}
	\end{cases}
	\]
	\[
	\eta_2(t) = 
	\begin{cases}
		\begin{aligned}
		&q_n(t) &\text{if}\ t \in [0,T] \\
		&q_{n-1}(2T-t) &\text{if}\ t \in [T,2T-t_1^*] \\
		&q_{n-2}(2T-t) &\text{if}\ t \in [2T-t_1^*,2T] \\
		\end{aligned}
	\end{cases}
	\]
	are $C^1$ and concave in their respective domains. However, one of them has two critical points, and this is a contradiction.  Arguing iteratively, we see that there are no internal collisions and so the minimizer is in $\mathcal{C}$. 

    Let us observe that $q_1$ can only have a collision at $t=T$. Indeed, if $q_2(0) =0$, the function $\tilde{q}= (q_1,\tilde q_2, \dots, \tilde q_n)$ with 
    \[
       \tilde{q}_k(t) = q_k(T-t)+q_1(T), \quad  k \ge2    
    \]
    has lower action. This implies that the collision between $q_j$ and $q_{j+1}$ happens at $t=T$  if $j$ is odd and at $t=0$ if $j$ is even.
\end{proof}
\end{lemma}
We are now in the position to prove Theorem \ref{thm:mu_to_0}
\begin{proof}[Proof of Theorem \ref{thm:mu_to_0}]
Let us begin by observing that, for fixed $\mu$, any solution $q^\mu$ of \eqref{eq:helium_equation_mu} solves the following system as well
	 \begin{equation}
	 	\label{eq:mu_frozen_auxiliary}
	 	\begin{cases}
	 		\ddot{q}_1 = f_1'(\vert q_1 \vert ) +\mu\sum_{j = 2}^n g'_{ij}(\vert q_j-q_i \vert) \\
	 		\ddot{q}_j =( f^{\ve_2}_j)'(\vert q_j \vert ) - \mu\sum_{k=1}^{j-1} g'_{jk}(\vert q_j-q_k\vert )+\mu\sum_{k = j+1}^n g'_{jk}(\vert q_k-q_j \vert)
	 	\end{cases}
	 \end{equation} 
	 for $j =2 , \dots, n$ provided $\ve_2$ is smaller than some $\ve_2(\mu)> 0$. This is so, since $q^\mu$ are collisionless and $q_1^\mu$ is concave.  On the other hand, thanks to Theorem \ref{thm:solution_helium}, for any fixed $\ve_2>0$ small enough and for any $\mu \in (0,1]$, there exists a solution $q^{\mu,\ve_2}$ of \eqref{eq:mu_frozen_auxiliary}. Indeed, the functions $f_1,f_2^{\ve_2},\dots f_n^{\ve_2}$ satisfy all the assumptions of Theorem \ref{thm:solution_helium} and so there exists a $2-$parameter family of $(\mu,\ve_2)-$frozen planet orbits $q^{\mu,\ve_2}$ solving \eqref{eq:mu_frozen_auxiliary}.
	 
	 For $j\ge 2$, let us integrate \eqref{eq:mu_frozen_auxiliary} over $[0,T]$. Taking into account the boundary conditions $\dot{q}_j^{\mu,\ve_2}(T) =\dot{q}_j^{\mu,\ve_2}(0)$, we obtain:
	 \begin{equation}	 	\label{eq:critical_point_eq_integrated}
	 	\int_0^T (f^{\ve_2}_j)'(q_j^{\mu,\ve_2}) +\mu \sum_{k=j+1}^n g_{kj}'(q_k^{\mu,\ve_2}-q^{\mu,\ve_2}_j) = \int_0^T\mu \sum_{k=1}^{j-1}  g_{kj}'(q_j^{\mu,\ve_2}-q^{\mu,\ve_2}_k).
	 \end{equation}
	 In particular, for any fixed $\ve_2>0$, the derivative $(f_j^{\ve_2})'(q_j^{\mu,\ve_2})$ is bounded on $[0,T]$ with respect to $\mu$ and thus, for $j=n$, $\int_0^T\mu \sum_{k=1}^{n-1}  g_{kn}'(q_n^{\mu,\ve_2}-q^{\mu,\ve_2}_k)$ is bounded as well. Taking instead $j=n-1$, we obtain the same conclusion for $\int_0^T\mu \sum_{k=1}^{n-2}  g_{k\,n-1}'(q_{n-1}^{\mu,\ve_2}-q^{\mu,\ve_2}_k)$. Iterating this procedure, we obtain that, for fixed $\ve_2$, the integrals $\mu\int_0^Tg_{jk}'(q^{\mu,\ve_2}_j-q^{\mu,\ve_2}_k)$ are bounded for all $j<k$ with respect to $\mu$. This implies that $\int_0^T \mu g_{ij}(q^{\mu,\ve_2}_j-q^{\mu,\ve_2}_i)$ converges to zero as $\mu\to0$. Indeed, let $\gamma = \gamma(\mu)$ be such that $ g_{ij} (\gamma) =1/ \sqrt{\mu}$. We have $\gamma\to 0 $ as $\mu \to 0$. Denote by $x = q^{\mu,\ve_2}_j-q^{\mu,\ve_2}_i.$ Thanks to assumption \eqref{assumption:homogenuity}, we have:
	\[
	\begin{aligned}
	 	\int_0^T \mu g_{ij}(x)&= \int_{x\le \gamma} \mu g_{ij}(x)+\int_{x> \gamma} \mu g_{ij}(x) \\
		&\le\int_{x\le \gamma} -\frac\mu\alpha g'_{ij}(x) x+\sqrt{\mu} T \\
	 	&\le -\frac{\mu \gamma }{\alpha }\int_0^Tg'_{ij}(x) +\sqrt{\mu} T \to 0\ \text{ as } \mu\to 0.
	\end{aligned} 
	\]
	Notice also that the functions $q^{\mu,\ve_2}$ are uniformly bounded in $H^1[0,T]$. This follows from an application of Lemma \ref{lemma:uniform_bounds}, since the bound on the action does not depend on $\mu$, neither on $\ve_2$. Let $\bar{q}^{\ve_2}$ denote a weak limit as $\mu$ tends to zero. We now prove that $\bar{q}^{\ve_2}$ lies in the set $\mathcal{D}$ given in Lemma \ref{lem:minimisers_D}. Indeed, assume by contradiction that $\bar{q}_j^{\ve_2}(t)>\bar{q}_{j-1}^{\ve_2}(t)$ for some $j$ and all $t\in[0,T]$, then the functions $\mu g_{ik}'(q^{\mu,\ve_2}_k-q^{\mu,\ve_2}_i)$ with $i<j\le k$ uniformly converge to 0 as $\mu\to 0$. Let us consider $y^{\mu,\ve_2} =\sum_{k=j}^n q_k^{\mu,\ve_2}$. From \eqref{eq:mu_frozen_auxiliary} and the boundary conditions we have $\dot y^{\mu,\ve_2}(0) = \dot y^{\mu,\ve_2}(T) = 0$ and 
    \[
    \ddot{y}^{\mu,\ve_2} = \sum_{k=j}^n\left(( f^{\ve_2}_k)'(q_k^{\mu,\ve_2}) - \mu\sum_{l=1}^{j-1} g'_{kl}( q_k^{\mu,\ve_2}-q_l^{\mu,\ve_2} )\right).
    \]
    Since the family $\{y^{\mu,\ve_2}\}_\mu$ is bounded in $H^1[0,T]$ it is pre-compact in the $C^0$ topology. Thanks to the equation, $\ddot{y}^{\mu,\ve_2}$ are bounded in the $C^0$ norm. Since $    \dot{y}^{\mu,\ve_2}(0) =0$ for all $\mu$, the family $\dot{y}^{\mu,\ve_2}$ is pre-compact in the $C^0$ topology as well. By an application of Ascoli-Arzelà Theorem, we thus find that up to subsequence the $y^{\mu,\ve_2}$ converge to a strictly concave function, having two critical points (in $t=0$ and $t =T$). A contradiction. Therefore, $\bar{q}^{\ve_2}\in\mathcal{D}$, for any $\ve_2$. 

	 Let $c_\mu^{\ve_2}$ denote the action level of $q^{\mu,\ve_2}$. Thanks to point $iii)$ of Lemma \ref{lemma:construction_boundary} and Lemma \ref{lemma:action_level}, we know that
	 \[
	     c_{\mu}^{\ve_2} \le c_0^{\ve_1,\ve_2}\le c_0^{0,\ve_2}\le c_0.
	 \]
	 Moreover, since we have proved that the repulsion terms converge to zero, we obtain that
		\begin{align*}
		 \limsup_\mu c_\mu^{\ve_2}  &= \limsup_\mu
		\int_0^T\frac12(\dot{q}_1^{\mu,\ve_2})^2+f_1(q_1^{\mu,\ve_2})+ \sum_{j=2}^n\frac12(\dot{q}_j^{\mu,\ve_2})^2+f_j^{\ve_2}(q_j^{\mu,\ve_2}) \\&= \limsup_\mu\mathcal{F}_{0,\ve_2}(q^{\mu,\ve_2})
		%\liminf_\mu c_\mu^{\ve_2} &= \liminf_\mu \mathcal{F}_{0,\ve_2}(q^{\mu,\ve_2})
		\end{align*}
	and the analogous equality for the $\liminf_\mu c_\mu^{\ve_2}$. Combining with the previous inequality we obtain
	\[
     c_0^{0,\ve_2}\ge\limsup_\mu c_\mu^{\ve_2}   \quad 
	\liminf_\mu \mathcal{F}_{0,\ve_2}(q^{\mu,\ve_2})\ge c_0^{0,\ve_2}.
	\]
     Thus, $\lim_\mu\mathcal{F}_{0,\ve_2}(q^{\mu,\ve_2})= c^{0,\ve_2}_0$. Recall that $\liminf_n( a_n+b_n) \ge \liminf a_n+\liminf_n b_n$ and thanks to Fatou Lemma and lower semi-continuity of the $L^2$ norm, we know that for all $j= 2, \dots, n$
     \begin{align*}
     \liminf_\mu \int_0^T f_j^{\ve_2}(q_j^{\mu,\ve_2}) &\ge \int_0^T f^{\ve_2}_j(\bar{q}_j^{\ve_2}), \\ \liminf_\mu \int_0^T f_1(q_1^{\mu,\ve_2}) &\ge \int_0^T f_1(\bar{q}_1^{\ve_2}),\\ \lim\inf_\mu \Vert \dot{q}^{\mu,\ve_2}\Vert_2 &\ge \Vert \dot{\bar{q}}^{\ve_2}\Vert_2.
     \end{align*}
     Since we have established  $\lim_\mu\mathcal{F}_{0,\ve_2}(q^{\mu,\ve_2})= c^{0,\ve_2}_0$, it follows that all inequalities are equalities. Recall that $\limsup(a_n+b_n)\ge \limsup_n a_n +\liminf b_n $ and thus
     \begin{align*}
     c_0^{0,\ve_2} &\ge \limsup_\mu  \frac12\Vert \dot{q}^{\mu,\ve_2}\Vert_2^2 + \liminf_\mu \int_0^T f_1(q_1^{\mu,\ve_2})+ \sum_{j=2}^n \liminf_\mu \int_0^T f_j^{\ve_2}(q_j^{\mu,\ve_2})\\
     &=\limsup_\mu \frac12 \Vert \dot{q}^{\mu,\ve_2}\Vert_2^2 +  \int_0^T f_1(\bar{q}_1^{\ve_2})+ \sum_{j=2}^n f_j^{\ve_2}(\bar{q}_j^{\ve_2})
     \end{align*}
     Thus, we can conclude that  $\lim_\mu  \Vert \dot{q}^{\mu,\ve_2}\Vert_2 = \Vert \dot{\bar{q}}^{\ve_2}\Vert_2$ and so
     \[
     \Vert q^{\mu,\ve_2}-\bar{q}^{\ve_2}\Vert_{H^1}^2 = \Vert \bar{q}^{\ve_2}\Vert_{H^1}^2+\Vert {q}^{\mu,\ve_2}\Vert_{H^1}^2-2 \langle  \bar{q}^{\ve_2}, q^{\ve_2,\mu}\rangle_{H^1} \to 0
     \]
     and hence the sequence converges strongly to a minimizer of $\mathcal{F}_{0,\ve_2}$ in $\mathcal{D}$. Note that, since the minimizer $x_{\ve_2}$ of $\mathcal{F}_{0,\ve_2}$ is unique, we have that for any sequence $\mu_m\to 0$ there exists a subsequence strongly converging to it. Thus, the whole family $q^{\mu,\ve_2}$ converges strongly to $x_{\ve_2}$ as $\mu\to 0$. Recall that, for $\ve_2$ small enough, the minimum value $c_0^{0,\ve_2}$ and the minimizer do not depend on $\ve_2$ (see point $iii)$ of Lemma  \ref{lemma:action_level}). Moreover, since the minimizers of $\mathcal{F}_{0,\ve_2}$ are $nT-$brake orbits, $q_2^{\ve_2}(t)\ge q_1^{\ve_2}(T)$ for all $t\in[0,T]$. However, $q_1^{\ve_2}(T)$ is uniformly bounded from below in $\ve_2$, again due to concavity (see the proof of $iii)$ in Lemma \ref{lemma:concavity_q1}). This means that we can assume that $q_j^{\mu,\ve_2}>\ve_2$ for $\mu<\mu_0(\ve_2)$ small enough and $j\ge2$. Thus, for this range of parameters, the functions $q^{\mu,\ve_2}$ solve both \eqref{eq:helium_equation_mu} and \eqref{eq:mu_frozen_auxiliary}. 
     
      To summarize,
      for any $\mu_m\to 0$ we can find a subsequence $\mu_{m_k}$ of $\mu_{m_k}$-frozen planet orbits converging  strongly to a minimizer of $\mathcal{F}$, a folded $nT-$brake. Thanks to Lemma \ref{lemma:action_level} folded $nT-$brakes are unique. This implies that the whole family of $\mu-$frozen planet orbits converges to the folded $nT-$brake since if that were not the case it would be possible to construct a sequence not converging to the folded $nT-$brake, contradicting what we have just proved. This concludes the proof. 
\end{proof}

\appendix
\section{Deformation Lemma}
\label{appendix}
In this section we prove the deformation Lemma needed in the proof of Theorem \ref{thm:general_existence_result}. Lemma \ref{lemma:critical_point} is a straightforward consequence of the following result.
\begin{lemma}[Deformation Lemma]
	\label{lemma:deformation_lemma}
	Let $\mathcal{U}$ be a open subset of a Hilbert space and let $\mathcal{A}\in C^{1,1}(\mathcal{U})$, $\mathcal{G}\in C^2(\mathcal{U})$. Assume that there exists real numbers $c>c^*$ and $b>b^*$ such that 
	\begin{enumerate}[label = \roman*)]
		\item $\overline{\{\mathcal{A}\le c\}\cap\{\mathcal{G}\le b\}}\subset \mathcal{U}$ and it is bounded.
		\item The level sets $b$ of $\mathcal{G}$ are regular, i.e.
		\begin{equation*}    
			\nabla \mathcal{G}(x) \ne 0, \text{ for any }x\in\{\mathcal{A}\le c\}\cap\{\mathcal{G}\ge b^*\}.
		\end{equation*}
		\item Any sequence $(x_n) \subseteq \mathcal{U}$ such that
		\begin{equation*}
	 			\mathcal{A} (x_n) \to  c^*,\quad  \limsup\limits_{n\to+\infty} \mathcal{G}(x_n)\le b, \quad \nabla \mathcal A(x_n) \to 0
		\end{equation*}
		has a convergent subsequence.
		\item Any sequence $(x_n) \subseteq \mathcal{U}$ such that
		\begin{equation*}
			\mathcal{A} (x_n) \to  c^*,\quad  \mathcal{G}(x_n) \to b, \quad \nabla \mathcal A(x_n)-\lambda_n \nabla  \mathcal{G}(x_n) \to 0
		\end{equation*}
		for $\lambda_n\ge0$ has a convergent subsequence.
		\item For any $\lambda>0$ we have:    
		\begin{equation}\label{eq:lambda}
			\nabla \mathcal{A}(x)\ne \lambda \nabla \mathcal{G}(x) \text{ for any } x \in \{\mathcal{A}= c^*\}\cap \{\mathcal{G}= b\}.
		\end{equation}
		\item For any $\ve\in (0, c-c^*]$ consider the sets: 
		\begin{equation*}
			X_{\pm\ve}:= \{\mathcal{A}\le c^*\pm\ve\}\cup\left(\{\mathcal
			A \le c \}\cap \{\mathcal{G}\ge b \}\right),
		\end{equation*}
		the set $X_\ve$ cannot be deformed into $X_{-\ve}$.
	\end{enumerate}
	Then, there exists a critical point of $\mathcal{A}$ lying in $\{\mathcal A = c^*\}\cap \{\mathcal{G}\le b\}.$
	\begin{proof}
		Let us assume by contradiction that the set
		\[
		K = \{q \in \mathcal{U} : \nabla \mathcal{A}(q) =0, \mathcal{A}(q)=c^*, \mathcal{G}(q)\le b\}=\emptyset.
		\]
		Since the Palais-Smale condition $iii)$ holds at level $c^*$ on $\{\mathcal{G}\le b\}$, we can assume that $\exists \ve>0$ such that:
		\begin{equation}\label{eq:bound_on_nabla_A}
		\Vert\nabla \mathcal{A}\Vert\ge \ve \text{ whenever } \vert \mathcal{A}-c^*\vert\le \ve \text { and }\mathcal{G}
		\le b+\ve.
		\end{equation}
		The Palais-Smale condition $iv)$ holds as well. We can thus assume that $\ve$ is chosen so small that there exists some $\ve_1>0$ such that the following condition holds too
		\begin{equation}
			\label{eq:proportionality_gradients_proof_deformation}
			\left\Vert\nabla\mathcal{A}-\frac{\Vert\nabla \mathcal{A}\Vert}{\Vert\nabla\mathcal{G} \Vert}\nabla\mathcal{G} \right\Vert\ge \ve_1, \text{ whenever } \vert \mathcal{A}-c^*\vert\le \ve_1 \text{ and }  \vert \mathcal{G}-b \vert\le\ve_1 .
		\end{equation}
		Indeed, if this were not the case, there would exist a sequence $(x_n)$ satisfying 
		\[
		\left\Vert\nabla\mathcal{A}(x_n)-\frac{\Vert\nabla \mathcal{A}(x_n)\Vert}{\Vert\nabla\mathcal{G}(x_n) \Vert}\nabla\mathcal{G}(x_n) \right\Vert\to 0.
		\]
		Condition $iv)$ then would imply that, up to subsequence, $(x_n)$ converges either to an element of $K$ or to a solution of \eqref{eq:lambda}. However, we are assuming the former to be empty and we know the latter is not possible thanks to $v)$.
		
		Let us observe that \eqref{eq:proportionality_gradients_proof_deformation} has also the following consequence on the value of the cosine of the angle  between $\nabla \mathcal{A}$ and $\nabla \mathcal{G}$.
		Indeed:
		\[	        
			\ve_1^2\le\left\Vert\nabla\mathcal{A}-\frac{\Vert\nabla \mathcal{A}\Vert}{\Vert\nabla\mathcal{G} \Vert}\nabla\mathcal{G} \right\Vert^2 = 2 \Vert \nabla \mathcal{A}\Vert^2\left( 1-\frac{\langle  \nabla \mathcal{A}, \nabla \mathcal{G}\rangle }{\Vert \nabla \mathcal{A}\Vert \Vert \nabla \mathcal{G}\Vert}\right).
		\]
		Since $\mathcal{U}\supset\overline{\{\mathcal{A}\le c\}\cap\{\mathcal{G}\le b\}}$ and it is bounded, then $\nabla \mathcal{A}$ is bounded on subsets of the form $\{\mathcal{G}\le b+\ve\}\cap \{c^*+\ve \ge \mathcal{A}\ge c^*-\ve\}$ for $\ve>0$ small enough, we conclude that there exists an $\ve_2>0$ such that:
		\begin{equation}\label{eq:cosine}
		1-\ve_2 \ge \frac{\langle  \nabla \mathcal{A}, \nabla \mathcal{G}\rangle }{\Vert \nabla \mathcal{A}\Vert \Vert \nabla \mathcal{G}\Vert}
		\end{equation}
		whenever $\vert \mathcal{A}-c^*\vert\le \ve_1$ and $\vert \mathcal{G}-b \vert\le\ve_1$ .
		At this point we want to show that, if $K=\emptyset$, then we can deform $X_{\ve}$ into $X_{-\ve}$, thanks to a suitable gradient flow. 
		Let us consider a smooth monotone function $\varphi : \mathbb{R}\to [0,1]$ satisfying $\varphi(t) =1$ for $t\le \ve/2$ and $\varphi(t) = 0$ for $t\ge \ve$.
		Let us define:
		\begin{equation*}
			 Z(q) = \varphi\left(\vert \mathcal{A}(q)-c^*\vert\right) \varphi\left( \mathcal{G}(q)
			-b\right) \left( \frac{-\nabla \mathcal{A}(q) }{\Vert \nabla \mathcal{A}(q) \Vert}+\varphi(\vert \mathcal{G}(q)-b\vert) \frac{\nabla \mathcal{G}(q) }{\Vert \nabla \mathcal{G}(q) \Vert} \right),
		\end{equation*}
        and consider the flow $\eta_t(q)$ generated by $Z(q)$.
		Let us observe that $\mathcal{A}$ is decreasing along flow lines of $\eta_t$. Indeed, thanks to \eqref{eq:cosine}:
		\begin{align*}
			\frac{d}{dt}\mathcal{A}(\eta_t(q)) = \varphi\left(\vert \mathcal{A}-c^*\vert\right) \varphi\left( \mathcal{G}-b\right) \Vert \nabla \mathcal{A} \Vert\bigg(-1+ \varphi(\vert \mathcal{G}(q)-b\vert)\frac{\langle  \nabla \mathcal{A}, \nabla \mathcal{G}\rangle }{\Vert \nabla \mathcal{A}\Vert \Vert \nabla \mathcal{G}\Vert}  \bigg) \le 0.
		\end{align*}
		Moreover, recalling \eqref{eq:bound_on_nabla_A}, on the set $\{\vert \mathcal{A}-c^*\vert\le\ve/2\}\cap \{\mathcal{G}\le b+\ve/2\}$ it is strictly decreasing since
		\begin{equation}\label{eq:linear_growth}
			\frac{d}{dt}\mathcal{A}(\eta_t(q)) \le -\ve \, \ve_2<0.
		\end{equation} 
		Furthermore, $\mathcal{G}$ is non-decreasing whenever $\vert\mathcal{G}- b \vert\le \ve/2$ since
		\begin{align*}
			\frac{d}{dt}\mathcal{G}(\eta_t(q)) &=  \varphi\left(\vert \mathcal{A}-c^*\vert\right) \varphi\left( \mathcal{G}-b\right) \Vert \nabla \mathcal{G} \Vert\bigg(\varphi(\vert \mathcal{G}(q)-b\vert ) - \frac{\langle  \nabla \mathcal{A}, \nabla \mathcal{G}\rangle }{\Vert \nabla \mathcal{A}\Vert \Vert \nabla \mathcal{G}\Vert} \bigg)\\
			&=\varphi\left(\vert \mathcal{A}-c^*\vert\right) \varphi\left( \mathcal{G}-b\right) \Vert \nabla \mathcal{G} \Vert \left(1-\frac{\langle  \nabla \mathcal{A}, \nabla \mathcal{G}\rangle }{\Vert \nabla \mathcal{A}\Vert \Vert \nabla \mathcal{G}\Vert}\right) \ge 0.
		\end{align*}
		
		We have thus built a flow which leaves all the sublevels of $\mathcal{A}$ invariant and the set $\{\mathcal{G}\ge b\}$ as well. To reach a contradiction, we have to show that there exists $s>0$ such that $\eta_s$ maps $X_{\ve/2}$ to $X_{-\ve/2}$. Assume that this does not happen, i.e., for every $s$ there is a point $q_s\in X_{\ve/2}$ whose image at time $s$ does not belong to $ X_{-\ve/2}$. Since $\{\mathcal{G} \ge b\}$ is invariant, this means that $\mathcal{G}(\eta_s(q_s)) < b$ for all $s$ and $\mathcal{A}(\eta_s(q_s))>c^*-\ve/2$. However, since $\mathcal A$ is decreasing on flow lines, this means that \[ c^*+\ve/2>
		\mathcal{A}(\eta_t(q_s))>c^*-\ve/2, \forall t \in (0,s].
		\]
		However, for $s$ sufficiently large, from \eqref{eq:linear_growth} we have
		\[
			\mathcal{A}(\eta_s( q_s ) )- \mathcal{A}(q_s)\le -\ve \ve_2 s<-\ve,
		\]
		a contradiction. Thus there exists $s_0$ such that $\eta_{s_0}(X_{\ve/2}) \subset X_{-\ve/2}$.
        We can thus define the following functions on $X_{\ve/2}$
        \[
           T_1(x) =\inf_{t\ge0}\{ \eta_t(x) \in \{\mathcal{A}\le c^*\}\}, \quad T_2(x) = \inf_{t\ge0}\{ \eta_t(x) \in \{\mathcal{G}\ge b\}\}.
           \]
        Since $Z(q)$ is transverse to $\{\mathcal{A}\le c^*\}$ and $\{\mathcal{G}\ge b\}$, the maps $T_1(x)$ and $T_2(x)$ are continuous when finite. Since $X_{\ve/2}$ is mapped to $X_{-\ve/2}$ the map $T(x) = \min\{T_1(x),T_2(x)\}$ is always finite and continuous. It follows that $h(t,x) = \eta_{tT(x)}(x)$ gives a deformation retract contradicting
        the assumption $vi)$ of the Lemma.
		
	\end{proof}
\end{lemma}

\noindent {\bf Statement about conflict of interest.}  The authors declare that there are no conflict of interest.\\
\noindent {\bf Statement about data availability.}  The authors declare that all data involved in this research are fully available.

\bibliographystyle{plain}
%\bibliography{ref}

\end{document}